\documentclass[a4paper, 11
 pt]{amsart}   
\usepackage{mathptmx, amssymb,amscd,latexsym, eulervm}   
\usepackage{amsmath, mathtools} 
\usepackage{amsthm}
\usepackage{mathdots}
\usepackage[colorlinks=true,linkcolor=blue,urlcolor=blue]{hyperref}
\usepackage{color}
\usepackage[onehalfspacing]{setspace}
\usepackage{tabularx}
\usepackage{amsfonts}
\usepackage{paralist}

\usepackage{aliascnt}
\usepackage{amscd}
\usepackage{blkarray}
\usepackage{mathbbol}
\usepackage{setspace}
\usepackage[inner=2.4cm,outer=2.4cm, bottom=3.2cm]{geometry}
\usepackage{tikz, tikz-cd}
\usepackage{calligra,mathrsfs}

\usepackage{tikz}
\usetikzlibrary{matrix}
\usetikzlibrary{arrows,calc}
\allowdisplaybreaks

\AtBeginDocument{%
	\def\MR#1{}
}

\newtheorem{theorem}{Theorem}[section]

\newtheorem{headthm}{Theorem}

\newaliascnt{headcor}{headthm}

\aliascntresetthe{headcor}

\newaliascnt{headconj}{headthm}

\aliascntresetthe{headconj}

\newaliascnt{corollary}{theorem}
\newtheorem{corollary}[corollary]{Corollary}
\aliascntresetthe{corollary}

\newaliascnt{claim}{theorem}

\aliascntresetthe{claim}

\newaliascnt{lemma}{theorem}
\newtheorem{lemma}[lemma]{Lemma}
\aliascntresetthe{lemma}

\newaliascnt{conjecture}{theorem}

\aliascntresetthe{conjecture}

\newaliascnt{proposition}{theorem}
\newtheorem{proposition}[proposition]{Proposition}
\aliascntresetthe{proposition}

\theoremstyle{definition}
\newaliascnt{definition}{theorem}
\newtheorem{definition}[definition]{Definition}
\aliascntresetthe{definition}

\newaliascnt{notation}{theorem}

\aliascntresetthe{notation}

\newaliascnt{example}{theorem}
\newtheorem{example}[example]{Example}
\aliascntresetthe{example}

\newaliascnt{examples}{theorem}

\aliascntresetthe{examples}

\newaliascnt{remark}{theorem}
\newtheorem{remark}[remark]{Remark}
\aliascntresetthe{remark}

\newaliascnt{question}{theorem}
\newtheorem{question}[question]{Question}
\aliascntresetthe{question}

\newaliascnt{questions}{theorem}

\aliascntresetthe{questions}

\newaliascnt{problem}{theorem}

\aliascntresetthe{problem}

\newaliascnt{construction}{theorem}

\aliascntresetthe{construction}

\newaliascnt{setup}{theorem}
\newtheorem{setup}[setup]{Setup}
\aliascntresetthe{setup}

\newaliascnt{setupdef}{theorem}

\aliascntresetthe{setupdef}

\newaliascnt{algorithm}{theorem}

\aliascntresetthe{algorithm}

\newaliascnt{observation}{theorem}

\aliascntresetthe{observation}

\newaliascnt{defprop}{theorem}

\aliascntresetthe{defprop}

\DeclareMathOperator{\Cone}{Cone}
\DeclareMathOperator{\dist}{dist}
\DeclareMathOperator{\Hom}{Hom}
\DeclareMathOperator{\rank}{rank}
\DeclareMathOperator{\reg}{reg}
\DeclareMathOperator{\vol}{vol}
\DeclareMathOperator{\nil}{nil}

\DeclareMathOperator{\Supp}{Supp}
\DeclareMathOperator{\Spec}{Spec}
\DeclareMathOperator{\Min}{Min}
\def \ba{\backslash}
\def \bu{\bullet}

\def \f0{\mathbf 0}
\def \fa{\mathbf a}
\def \fb{\mathbf b}
\def \fu{\mathbf u}
\def \fv{\mathbf v}
\def \fx{\mathbf x}
\def \fy{\mathbf y}
\def \fz{\mathbf z}

\def \FF{\mathbb F}
\def \kk{\Bbbk}
\def \NN{\mathbb N}
\def \QQ{\mathbb Q}
\def \RR{\mathbb R}
\def \ZZ{\mathbb Z}

\def \fm{\mathfrak{m}}

\def\equationautorefname~#1\null{(#1)\null}
\def\sectionautorefname~#1\null{Section #1\null}
\def\subsectionautorefname~#1\null{\S #1\null}

\numberwithin{equation}{section}

\begin{document}
\title{Asymptotic colengths for families of ideals: an analytic approach}
\author{Sudipta Das and Cheng Meng}

\address{Sudipta Das\\School of Mathematical and Statistical Sciences, Arizona State University, P.O. Box 871804, Tempe, AZ 85287-18041, USA. \emph{Email:} {\rm sdas137@asu.edu}}

\address{Cheng Meng\\ Yau Mathematical Sciences Center, Tsinghua University, Beijing 100084, China. \emph{Email:} {\rm cheng319000@tsinghua.edu.cn}}

\begin{abstract}
This article focuses on the existence of asymptotic colengths for families of $\fm_{R}$-primary ideals in a Noetherian local ring $(R,\fm)$. In any characteristic, we generalize graded families to weakly graded families of ideals, and in prime characteristic, we explore various families such as weakly $p$-families and weakly inverse $p$-families. The main contribution of this paper is providing a unified analytic method to prove the existence of limits. Additionally, we establish Brunn-Minkowski type inequalities, positivity results, and volume = multiplicity formulas for these families of ideals.
\end{abstract}
\maketitle

\tableofcontents

\section{Introduction} \label{intro}

This article investigates the asymptotic colengths of specific families of ideals indexed by $\NN$, in the context of rings of general characteristic, and by powers of a prime $p > 0$, when the ambient ring is of characteristic $p > 0$.

The most classical example is the following limit:

\begin{equation} \label{eq_1.1}
\lim_{n \to \infty}d! \cdot \dfrac{\ell(R/I^{n})}{n^d}\end{equation}
for a fixed $\fm$-primary ideal $I$ in the Noetherian local ring $(R, \fm)$. In this situation $\ell_R(R/I^n)$ is a polynomial of degree $d$ for large $n$ and the limit is a positive integer called Hilbert-Samuel multiplicity, denoted by $e(I,R)$. This multiplicity is a fundamental invariant in commutative ring theory and singularity theory. Classically, the Hilbert-Samuel multiplicity is used to define intersection numbers for varieties. Another significant application is a Rees' criterion for integral dependence (\cite{rees1978theorem,sally1983cohen,sally1977associated}).

If we consider a \emph{graded family} i.e. $I_{\bu}=\{I_n\}_{n \in \NN}$, with the property  $I_{n}I_{m} \subset I_{n+m}$, the polynomial behavior described earlier does not necessarily hold for a general graded  family of $\fm$-primary ideals. This is primarily because the associated Rees ring $\bigoplus_{n\ge 0}I_n$ is often not a finitely generated $R$-algebra. Consequently, we cannot generally expect  $\ell_R(R/I_n)$ to  be polynomial  for large $n$ in general. Moreover in  \cite{cutkosky1993problem} Cutkosky and Srinivasan showed that there exists a graded family of ideals such that the limit in \autoref{eq_1.1} with $I^{n}$ replaced by $I_{n}$ can be an irrational number. Despite this, it is remarkable that under very general conditions, this length asymptotically approaches a polynomial of degree 
$d$, confirming that the limit exists \cite{cutkosky2014asymptotic}.

  This type of limit was first considered in the work of Ein, Lazarsfeld and Smith \cite{ein2003uniform} and Musta\c{t}\u{a} \cite{mustactǎ2002multiplicities} . In their work Lazarsfeld and Musta\c{t}\u{a} \cite{lazarsfeld2009convex} established that
the limit exists for all graded families of $m_R$-primary ideals in $R$ if $R$ is a domain which is essentially of finite type over an algebraically closed field $\kk$ with $R/\fm=\kk$. Each of these conditions is essential for their proof. Their approach involves  reducing the problem to one on graded linear series on a projective variety, and then using a method introduced by Okounkov  \cite{okounkov2003would} to reduce the problem to one of counting points in an integral semigroup.

Later Cutkosky in \cite{cutkosky2013multiplicities} first showed that
the limit exists for all graded families of $m_R$-primary ideals in $R$ if $R$ is analytically unramified, i.e., $\hat R$ is reduced (where $\hat R$ is the completion of $R$ with respect to the maximal ideal $\fm$), equicharacteristic and $R/\fm$ is perfect. Finally he settled the problem in \cite{cutkosky2014asymptotic} for any graded family of $\fm$ primary ideals with the condition that $\dim \left(\nil(\hat{R})\right)<\dim R$, where the nilradical $N(R)$ of a ring $R$ is defined by
$$
N(R)=\{x\in R\mid x^n=0 \mbox{ for some positive integer $n$}\}.
$$


\textbf{Limits in positive characteristic:} In the context of Noetherian local rings of  characteristic $p >0$, a  central area of study involves the Frobenius endomorphism $F : R \to R$ \phantom{}defined by $r \mapsto r^p$ for all $r \in R$. Kunz's result \cite{kunz1969characterizations} established a link between the flatness of the Frobenius and the regularity of the ring, revealing that the behavior of the Frobenius endomorphism plays a crucial role in understanding the singularities of these rings.

Two key numerical invariants that measure the failure of flatness under iterated Frobenius actions are the Hilbert-Kunz multiplicity and the $F$-signature.

The \emph{Hilbert-Kunz multiplicity} \cite{monsky1983hilbert} is defined as follows: $$ e_{HK}(I,R)= \lim_{e \to \infty}  \dfrac{\ell(R/I^{[p^e]})}{p^{ed}},$$ where $\ell(R/I^{[p^e]})$ denotes the $R$-module length of the quotient $R/I^{[p^e]}$.

 The \emph{$F$-signature} \cite{smith1997simplicity} can be defined by
$$ s(R)= \lim_{e \to \infty}  \dfrac{\ell(R/I_{p^e})}{p^{ed}},$$ using the sequence of ideals $I_{p^e}=\{r \in R \mid \phi(F_{*}^er) \in \fm \quad  \text{for all} \quad \phi \in Hom_{R}(F_{*}^eR,R)\}$, where $F_{*}^{e}$ is the  \emph{$e$-th iterated Frobenius}, providing an $R$-module structure on $F_{*}^{e}R$.

Over the past two decades, these invariants have been the subject of extensive research, addressing various questions regarding their existence \cite{monsky1983hilbert, tucker2012f}, semicontinuity \cite{polstra2018uniform, smirnov2016upper}, positivity \cite{hochster1994f, aberbach2002f}, and irrationality \cite{brenner2013irrational}.

These two numerical invariants come as limits of $p$-families of ideals, i.e.,  a sequence of ideals $ I_{\bullet}=\{I_{q}\}_{q=1}^{\infty}$ with $I_{q}^{[p]} \subseteq I_{pq}$ for all $q=p^{e}$ as shown in \cite{tucker2012f}. Moreover in \cite{polstra2018f} Polstra and Tucker showed that the limit exists for any  $p$-family when $R$ is an $F$-finite local domain. In the subsequent work,  Hernández and Jeffries \cite{hernandez2018local} showed the limit exists if $\dim \left(\nil(\hat{R})\right)<\dim R$.

\subsection{Results in present work}

\phantom{}\\

Throughout this subsection, we assume that every family of ideals indexed by 
$\NN$ satisfies the condition $\fm^{cn} \subset I_{n}$. For families of ideals indexed by powers of 
$p$, we require the condition $\fm^{cq} \subset I_{q}$, for all $q=p^e$. Families of ideals that meet these criteria are referred to as being bounded below linearly (BBL for short). Similar condition for a pair of families of ideals has been considered in \cite{swanson1997powers,cid2022mixed,hernandez2018local,cutkosky2024epsilon}.

We denote $R^\circ=R\ba (\cup_{P \in \Min(R)}P)$, where $\Min(R)=\{P \in \Spec(R) \mid \dim R/P=\dim R\}$. First we consider the general characteristic case. In this case, we consider the following families of ideals $I_\bu=\{I_n\}_{n \in \NN}$ indexed by $\NN$. 
\begin{definition} \label{intro_weakly_graded_defn}
  We say $I_\bu$ is a \emph{weakly graded family} of ideals if there exists $c \in R^\circ$ independent of $m,n$ such that $cI_mI_n \subset I_{m+n}$ for all $m,n \in \NN$.

\end{definition}
\phantom{}\\
For rings of characteristic $p$, we will consider the following families of ideals $I_\bu=\{I_{p^e}\}_{e \in \NN}$ indexed by powers of $p$.

\begin{definition} \label{intro_several_types_p_families}
 We say:
\begin{enumerate}
\item $I_\bu$ is a \emph{weakly $p$-family} of ideals, if there exists $c \in R^\circ$ such that for any $q=p^e$, $cI^{[p]}_q \subset I_{pq}$.
\item $I_\bu$ is a \emph{weakly inverse $p$-family} of ideals, if there exists $c \in R^\circ$ such that for any $q=p^e$, $cI_{pq} \subset I^{[p]}_q$.

\end{enumerate}
\end{definition}

\begin{remark}
Note that,
\begin{enumerate}

\item By applying \autoref{graded_p_all_BBL}, we can conclude that if $I_{1}$ is $\fm$-primary then all the families previously discussed can be classified as either BBL weakly graded families or BBL weakly 
$p$-families.

\item Additionally, several other interesting classes of examples are presented in  \autoref{weakly_graded_example_1}, \autoref{weakly_graded_example_2},  \autoref{weakly_inverse_p_example}, \autoref{weakly_p_example}, \autoref{F_graded_is_weakly_p_family}.
\end{enumerate}
\end{remark}

In this paper, we demonstrate the existence of limits for the families discussed in \autoref{intro_weakly_graded_defn} and \autoref{intro_several_types_p_families} by employing a suitable valuation on the ring 
$R$, referred to as the OK valuation (see \autoref{OK_val_def}). 

Note that Cutkosky's existence proof \cite{cutkosky2014asymptotic} for graded families and Hernández, Jeffries'  proof \cite{hernandez2018local} for $p$-families can be summarized as follows: Reduce the problem to a complete local domain, then establish a suitable valuation on the ring 
$R$. Recognize that values in a graded family form a semigroup structure and for $p$-families it form a $p$-system. Calculate the $R$-module lengths by counting points within this structure, and show that the growth rate of these points aligns with the Euclidean volume of Okounkov bodies(for graded families) and Euclidean volume of $p$-bodies (for $p$-families).

Therefore, it becomes apparent that as the properties of the family are modified, it is essential to adapt new combinatorial structures to establish the existence of asymptotic colengths. In this paper, we address this challenge by providing a unified approach through analysis to demonstrate the existence of limits for more general classes of ideals. \emph{The primary contribution of this paper is this comprehensive, one-for-all method to prove limit existence.} To achieve this, we analyze specific characteristic functions derived from the conditions $\fm^{cn} \subset I_{n}$ or $\fm^{cp^e} \subset I_{p^e}$. We establish that the limit exists if the volume of a specific region in Euclidean space is zero. The core technical part of this paper lie in this analysis, elaborated in \autoref{section_3}, \autoref{section_5}, \autoref{section_6}, and \autoref{section_7}. The reduction to the case of complete local domains follows a methodology similar to Cutkosky's arguments.

The underlying idea of the proof in the context of weakly graded families and weakly 
$p$-families is based on simple analysis. However, establishing the existence of the corresponding limit associated with weakly inverse 
$p$-families is more subtle and constrained, reflecting deeper insights  into the limiting behavior of the OK valuation. The challenge of using valuations arises from the following observation: for any valuation $\nu$, we have
$$p\nu(I)+S \subset \nu(I^{[p]}).$$
However, the reverse containment does not necessarily hold, as the sum of two elements from $I^{[p]}$ may yield a smaller valuation. For inverse 
$p$-families, we observe the following relationships: $$p\nu(I_q) \subset \nu(I^{[p]}_q), \nu(I_{pq}) \subset \nu(I^{[p]}_q).$$
Importantly, there are no containment relations between the sets $p\nu(I_q),\nu(I_{pq})$. To address this issue, we introduce the concept of an OK basis (see \autoref{OK_basis_defn}) and establish the existence of limits when 
$R$ is an $F$-finite local domain with a perfect residue field. Note that in this setup $R$ always has an OK basis, and we get the required containment from the other side ( see \autoref{OK_basis_existence}, \autoref{OK_basis_gives_containment_from_the_other_side}, and \autoref{OK_basis_gives_containment_from_the_other_side2}).\\

We are now ready to state the first main result of this paper.

\begin{headthm}\label{Thm_A}
Let $R$ be a Noetherian local ring with $\dim \left(\nil(\hat{R})\right)<\dim R$. 

\begin{enumerate} 
    \item [{\rm (a)}](\autoref{general_weakly_graded_limit_exist}) For any BBL weakly graded families $I_\bu$ of $R$-ideals in any characteristic,

$$\lim_{n \to \infty}\frac{\ell(R/I_n)}{n^d}$$

exists.

\item [{\rm (b)}](\autoref{general_weakly_p_and_weakly_inverse_p_exists}) For any BBL families $I_\bu$ of $R$-ideals in positive characteristic,which is either weakly $p$-family, or weakly inverse $p$-family when $R$ is $F$-finite with perfect residue field, then

$$\lim_{q \to \infty}\frac{\ell(R/I_q)}{q^d}$$
exists.

\end{enumerate}
\end{headthm}

\begin{remark}

Note that:

\begin{enumerate}
\item The assumption $\dim \left( \nil(\hat{R})\right)<\dim R$ is the most general assumption we can add. Counterexamples can be found in (\cite{cutkosky2014asymptotic},\cite{hernandez2018local}) for rings $R$ that violates this condition, i.e., we can always construct graded families and $p$-families whose corresponding limits do not exist.

\item  In \autoref{alternate_limit_existence_proof_for_weakly_inverse_p_family}, we provide an alternative proof of limit existence for weakly inverse 
$p$-families that does not rely on the assumption that 
$R$ has a perfect residue field. This approach avoids the analytic techniques developed in this article and instead utilizes methods from \cite{polstra2018f}, this method can also be applied to show limit existence for weakly $p$-family in a $F$-finite local domain. Nonetheless, the approach used in \autoref{Thm_A} is instrumental for obtaining several important applications, including the Volume = Multiplicity formula and the Minkowski inequality.

\item When considering inverse graded families defined by $I_{m}I_{n} \supset I_{m+n}$, we showed in \autoref{inverse_graded_limit_may_not_exist} that the limit in \autoref{Thm_A}, part $(a)$ may not exist.
\end{enumerate}
\end{remark}

\textbf{Positivity}: A natural question arises: under what conditions are the limits mentioned in \autoref{Thm_A} positive?

To address this inquiry, we focus on the BAL (bounded above linearly) family of ideals discussed in \autoref{Positivity subsection}. The consideration of this type of family is motivated by the work of Mustață \cite{mustactǎ2002multiplicities} and Hernández and Jeffries \cite{hernandez2018local}.

\begin{definition}
Assume $I_\bu$ is a family of $R$-ideals indexed either by $\NN$ or powers of $p$. We say $I_\bu$ is bounded above linearly, or BAL for short, if:
\begin{enumerate}
\item When $I_\bu$ is a family indexed by $\NN$, there exists $c \in \NN$ such that $I_{n} \subset \fm^{\lfloor n/c \rfloor}$;
\item When $I_\bu$ is a family indexed by powers of $p$, there exists $q_0 \in \NN$ such that $I_{qq_0} \subset \fm^{[q]}$.
\end{enumerate}
\end{definition}

In this article, we present the following theorem that addresses the question of positivity.

\begin{headthm} (\autoref{positivity_for_BAL_in_OK_domain})\label{Thm_B}
Let $R$ be an OK domain.
\begin{enumerate} 
\item [{\rm (a)}] Let $I_\bu$ be a BBL weakly graded family of ideals. Then $I_\bu$ is BAL if and only if
$$\lim_{n \to \infty}\frac{\ell(R/I_n)}{n^d}>0.$$
\item [{\rm (b)}] 
Assume $R$ has characteristic $p>0$. Assume either $I_\bu$ is a BBL weakly $p$-family of ideals, or a BBL weakly inverse $p$-family of ideals when $R$ is $F$-finite with perfect residue field. Then $I_\bu$ is BAL if and only if
$$\lim_{q \to \infty}\frac{\ell(R/I_q)}{q^d}>0.$$
\end{enumerate}
\end{headthm}

\subsection{Applications}

\phantom{}\\

\textbf{F-Graded System}: Assume $R$ is $F$-finite of characteristic $p$. An example of a weakly $p$-family which is not necessarily a $p$-family is an $F$-graded system of ideals (see \autoref{defn_F_graded_system}, \autoref{F_graded_is_weakly_p_family}). This system is relevant in the study of Cartier subalgebras of 
$R$ and characterizes all Cartier subalgebras of Gorenstein rings. The existence of limits associated with such families is established by the result of Polstra and Tucker in \cite{polstra2018f}.
However, in Brosowsky's thesis \cite{thesis}, a conjecture is proposed regarding the description of this limit, suggesting that it can be realized as the volume of certain subsets in $\RR^d$. The original conjecture assumes that $R$ is a polynomial ring. As a corollary of \autoref{limit_existence_weakly_p_family_OK_domain}, we have proven a generalized version of the conjecture (see \autoref{main_result_for_BBL_F_graded_system}), demonstrating that if 
$R$ is an 
$F$-finite OK domain, the conjecture still holds true.\\

This new analytic approach not only gives us a uniform way to show limit existence for several different classes of families of ideals but also managed to provide new simplified proof of the following important results.\\

\textbf{Volume = Multiplicity formula}: The following \emph{Volume = Multiplicity} formula has been proven for valuation ideals associated to an Abhyankar valuation in a regular local ring which is essentially of finite type over a field in \cite{ein2003uniform}; when $R$ is a local domain which is essentially of finite type over an algebraically closed field $\kk$ with $R/\mathfrak{m} = \kk$ in \cite{lazarsfeld2009convex}; it is proven when $R$ is analytically unramified with perfect residue field in \cite{cutkosky2014asymptotic}, and for any $d$-dimensional Noetherian local ring $R$ with $\dim N(\hat{R}) < d$ in \cite{cutkosky2015general}. For $p$-families of ideals in a ring of positive characteristic with $\dim N(\hat{R}) < d$,  this has been proven by the first author in \cite{das2023volume}.

 \begin{headthm}\label{Thm_C}
Let $R$ be a Noetherian local ring of dimension $d$ with $\dim \left(\nil(\hat{R})\right)<\dim R$. 

\begin{enumerate} 
\item [{\rm (a)}](\autoref{general_volume_multiplicity_formula_for_weakly_graded_family}) Let $I_\bullet$ be a BBL weakly graded family, then Volume =  Multiplicity formula holds. That is,
\begin{equation*} 
\lim_{n \to \infty} d!\frac{\ell(R/I_n)}{n^d}=\lim_{n \to \infty}\frac{e(I_n,R)}{n^d}
\end{equation*}

\item [{\rm (b)}](\autoref{general_volume_multiplicity_formula_for_weakly_p_weakly_inverse_p_family}) Assume $R$ has positive characteristic. Then for a BBL family $I_\bullet$, if either $I_\bu$ is weak $p$-family, or $I_\bu$ is weak inverse $p$-family and $R$ is $F$-finite with perfect residue field, then  volume = multiplicity formula holds. That is,
\begin{equation*} 
\lim_{q \to \infty}\frac{\ell(R/I_q)}{q^d}=\lim_{q \to \infty}\frac{e_{HK}(I_q,R)}{q^d}.
\end{equation*}
\end{enumerate}
\end{headthm}

\phantom{Then for a BBL family $I_\bullet$, assume either $I_\bu$ is weak $p$-family, or $I_\bu$ is weak inverse $p$-family and $R$ is $F$-finite with perfect residue field, then  volume = multiplicity formula holds. That is,
}

\phantom{}

\textbf{Minkowski inequality}: The Minkowski inequality for powers of $\fm$-primary ideals $I$ and $J$ was first studied  by Teissier \cite{teissier1977inegalite} and Rees and Sharp \cite{rees1978theorem}.  \autoref{Thm_C} has been proven for arbitrary graded families of $\fm$-primary ideals in a regular local ring with algebraically closed residue field in \cite{mustactǎ2002multiplicities, kaveh2012newton}. Later Cutkosky in \cite{cutkosky2015asymptotic} showed this for any graded family in any characteristic with $\dim N(\hat{R}) < d$, and later Hernández and Jeffries \cite{hernandez2018local} showed this for any $p$-family in a positive characteristic ring $R$ with $\dim N(\hat{R}) < d$.   In \phantom{}\autoref{Minkoswki section} we prove \autoref{Thm_C} which recovers all the above results, and give positive answers to new types of families of ideals.

\begin{headthm}\label{Thm_D}
Let $R$ be a Noetherian local ring of dimension $d$ with $\dim \left(\nil(\hat{R})\right)<\dim R$. 

\begin{enumerate} 
    \item [{\rm (a)}](\autoref{Minkowski_general_for_weakly_graded}) If  $I_\bullet$ and $J_\bullet$ are two BBL weakly graded families, then so is $I_\bullet J_\bullet$, and
$$\left(\lim_{n \to \infty}\frac{\ell(R/I_n)}{n^d}\right)^{1/d}+\left(\lim_{n \to \infty}\frac{\ell(R/J_n)}{n^d}\right)^{1/d} \geqslant \left(\lim_{n \to \infty}\frac{\ell(R/I_nJ_n)}{n^d}\right)^{1/d}.$$

\item [{\rm (b)}](\autoref{Minkowski_general_weakly_p_weakly_inverse_p})  
   If  $I_\bullet$ and $J_\bullet$ are two BBL weakly $p$-families, then so is $I_\bullet J_\bullet$. Moreover if $R$ is $F$-finite with perfect residue field, and $I_\bullet$ and $J_\bullet$ are two BBL weakly inverse $p$-families, then so is $I_\bullet J_\bullet$, and in either case,
$$\left(\lim_{q \to \infty}\frac{\ell(R/I_q)}{q^d}\right)^{1/d}+\left(\lim_{q \to \infty}\frac{\ell(R/J_q)}{q^d}\right)^{1/d} \geqslant \left(\lim_{q \to \infty}\frac{\ell(R/I_qJ_q)}{q^d}\right)^{1/d}.$$

\end{enumerate}
\end{headthm}

\phantom{OUtline}\\

\textbf{Outline.} The structure of the paper is as follows. In \autoref{section_2}, we fix the notation and recall some results that are needed throughout the rest of the paper. \autoref{section_3} contents properties, characterization and applications of height function (see \autoref{ht_fun_existence}, \autoref{ht_func_definition}). In \autoref{section_4}, we list all the conditions on the family of ideals we consider, and provide some examples. In \autoref{section_5} we discuss bounds for the valuation of BBL families. \autoref{section_6} and \autoref{section_7} are the technical heart of this paper, where we prove the convergence of the limit for weakly graded families, and weakly $p$-families when the ring $R$ is an OK domain (see \autoref{OK_val_def}, \phantom{d}\autoref{OK_domain_defn}). In \autoref{section_8} we introduce the notion of $OK$ basis, and prove the existence of limit for weakly inverse $p$-families, when the ring is an $F$-finite domain with perfect residue field. \autoref{section_9} contains the proof of volume = multiplicity formula, Minkowski inequality, and positivity result when $R$ is an OK domain. The most general way as it is stated in \autoref{Thm_A}, \autoref{Thm_C} and \autoref{Thm_D} has been proven in \autoref{section_10}.\\

\begin{center}
 {\textbf{ Acknowledgments}}
\end{center} 
The authors are grateful to Jonathan Montaño for his valuable comments, which significantly improved the presentation of this paper. This material is based upon work supported by the National Science Foundation under Grant No. DMS-1928930 and by the Alfred P. Sloan Foundation under grant G-2021-16778, while the authors were in residence at the Simons Laufer Mathematical Sciences Institute (formerly MSRI) in Berkeley, California, during the Spring 2024 semester.

\section{Conventions and basic notions} \label{section_2}
In this section, we will introduce some settings and notations that will be used in the following sections.
\subsection{Notations in Euclidean space}

Let \( d \geq 1 \) be a positive integer, and let \( \mathbb{R}^d \) represent the \( d \)-dimensional Euclidean space with origin denoted as \( \mathbf{0} \). We utilize \( \langle \cdot, \cdot \rangle \) for the standard inner product, \( \|\cdot\| \) for the standard norm, and \( \dist(\cdot, \cdot) \) for the standard distance in \( \mathbb{R}^d \).

For two sets \( U, V \subset \mathbb{R}^d \) and a scalar \( \alpha \in \mathbb{R} \), we define the Minkowski sum as
\[
U + V = \{ \mathbf{u} + \mathbf{v} : \mathbf{u} \in U, \mathbf{v} \in V \}
\]
and the scalar multiplication as
\[
\alpha U = \{ \alpha \mathbf{u} : \mathbf{u} \in U \}.
\]
If \( U \subset \mathbb{R}^d \) and \( \mathbf{v} \in \mathbb{R}^d \), then \( U + \mathbf{v} = U + \{\mathbf{v}\} \) represents the translation of \( U \). We define \( U - V = U + (-1)V \) and \( U - \mathbf{v} = U + (-\mathbf{v}) \).

For \( U, V \subset \mathbb{R}^d \), we denote the complement of \( V \) in \( U \) as
\[
U \setminus V = \{ \mathbf{u} \in \mathbb{R}^d : \mathbf{u} \in U, \mathbf{u} \notin V \}.
\]
The cardinality of \( U \) is denoted by \( \#(U) \). The interior, closure, and boundary of \( U \) in the Euclidean topology are represented by \( U^\circ \), \( \overline{U} \), and \( \partial U = \overline{U} \setminus U^\circ \), respectively. If \( U \) is measurable, its Lebesgue measure is denoted by \( \vol(U) \). All integrals in this paper will be Lebesgue integrals.

For \( \mathbf{u} \in \mathbb{R}^d \) and \( \delta > 0 \), the open ball centered at \( \mathbf{u} \) with radius \( \delta \) is denoted \( B(\mathbf{u}, \delta) \), and the closed ball is denoted \( \overline{B(\mathbf{u}, \delta)} \). For a set \( U \subset \mathbb{R}^d \), we define the distance from a point \( \mathbf{v} \) to \( U \) as
\[
\dist(\mathbf{v}, U) = \inf \{ \dist(\mathbf{v}, \mathbf{u}) : \mathbf{u} \in U \}.
\]
The set \( B(U, d) \) is defined as \( U + B(\mathbf{0}, d) \), and \( \overline{B(U, d)} = U + \overline{B(\mathbf{0}, d)} \). Notably, if \( K \) is compact, then for any \( \mathbf{x} \in \mathbb{R}^d \), we have \( \dist(\mathbf{x}, K) = \dist(\mathbf{x}, \mathbf{u}) \) for some \( \mathbf{u} \in K \), leading to
\[
\overline{B(K, d)} = \{ \mathbf{x} : \dist(\mathbf{x}, K) \leq d \}.
\]

We define a set \( C \subset \mathbb{R}^d \) as a convex cone if it is closed under linear combinations of its elements with nonnegative coefficients. Specifically, for any \( \alpha_1, \alpha_2 \geq 0 \), it follows that \( \alpha_1 C + \alpha_2 C \subset C \). For \( U \subset \mathbb{R}^d \), the minimal closed convex cone containing \( U \) is denoted \( \Cone(U) \), representing the closure of the set formed by all nonnegative linear combinations of elements in \( U \). A cone \( C \) is termed full-dimensional if it spans \( \mathbb{R}^d \) as a vector space, which is equivalent to stating \( C^\circ \neq \emptyset \).

A circular cone with vertex \( \mathbf{0} \) and axis \( \mathbf{a} \) (or parallel to \( \mathbf{a} \)) is defined as follows:
\[
C_{\mathbf{a}, c} = \{ \mathbf{x} : \langle \mathbf{x}, \mathbf{a} \rangle \geq c \|\mathbf{x}\| \}
\]
for some \( 0 < c < \|\mathbf{a}\| \). If \( c = \|\mathbf{a}\| \), the set reduces to a ray parallel to \( \mathbf{a} \); if \( c > \|\mathbf{a}\| \), the set is just $ \{\textbf{0}\} $. We fix \( \mathbf{0} \neq \mathbf{a} \in \mathbb{R}^d \) and \( \alpha \in \mathbb{R} \). We define the following sets:
\[
H_{\mathbf{a}, = \alpha} = \{ \mathbf{x} \in \mathbb{R}^d : \langle \mathbf{x}, \mathbf{a} \rangle = \alpha \},
\]
\[
H_{\mathbf{a}, < \alpha} = \{ \mathbf{x} \in \mathbb{R}^d : \langle \mathbf{x}, \mathbf{a} \rangle < \alpha \},
\]
\[
H_{\mathbf{a}, \leq \alpha} = \{ \mathbf{x} \in \mathbb{R}^d : \langle \mathbf{x}, \mathbf{a} \rangle \leq \alpha \}.
\]
Thus, \( H_{\mathbf{a}, = \alpha} \) describes a hyperplane that may not pass through \( \mathbf{0} \), \( H_{\mathbf{a}, < \alpha} \) denotes an open half-space, and \( H_{\mathbf{a}, \leq \alpha} \) represents a closed half-space. For \( \mathbf{u} \in \mathbb{R}^d \), we define:
\[
H_{\mathbf{a}, = \mathbf{u}} = H_{\mathbf{a}, = \langle \mathbf{u}, \mathbf{a} \rangle} = \{ \mathbf{x} \in \mathbb{R}^d : \langle \mathbf{x}, \mathbf{a} \rangle = \langle \mathbf{u}, \mathbf{a} \rangle \},
\]
\[
H_{\mathbf{a}, < \mathbf{u}} = H_{\mathbf{a}, < \langle \mathbf{u}, \mathbf{a} \rangle} = \{ \mathbf{x} \in \mathbb{R}^d : \langle \mathbf{x}, \mathbf{a} \rangle < \langle \mathbf{u}, \mathbf{a} \rangle \},
\]
\[
H_{\mathbf{a}, \leq \mathbf{u}} = H_{\mathbf{a}, \leq \langle \mathbf{u}, \mathbf{a} \rangle} = \{ \mathbf{x} \in \mathbb{R}^d : \langle \mathbf{x}, \mathbf{a} \rangle \leq \langle \mathbf{u}, \mathbf{a} \rangle \}.
\]
We will omit the reference to \( \mathbf{a} \) when the context makes it clear.

\subsection{Pointed cones and upward cones}

\begin{definition}
Let \( C \subset \mathbb{R}^d \) be a cone and let \( \mathbf{0} \neq \mathbf{a} \in \mathbb{R}^d \). We say that \( C \) is pointed at the direction of \( \mathbf{a} \) if it is contained within a circular cone with vertex at \( \mathbf{0} \) and axis along \( \mathbf{a} \). Furthermore, we say that \( C \) is upward to the direction of \( \mathbf{a} \) if it is pointed and also contains a circular cone with the same vertex and axis.
\end{definition}

More formally, a cone \( C \) is pointed if there exists a constant \( 0 < c_1 < \|\mathbf{a}\| \) such that
\[
C \subset \{ \mathbf{x} : \langle \mathbf{x}, \mathbf{a} \rangle \geq c_1 \|\mathbf{x}\| \}.
\]

A cone \( C \) is considered upward if there exist constants \( 0 < c_1 < c_2 < \|\mathbf{a}\| \) such that
\[
\{ \mathbf{x} : \langle \mathbf{x}, \mathbf{a} \rangle \geq c_2 \|\mathbf{x}\| \} \subset C \subset \{ \mathbf{x} : \langle \mathbf{x}, \mathbf{a} \rangle \geq c_1 \|\mathbf{x}\| \}.
\]

It is important to note that  by inducing a metric on the hyperplane defined by \( H = H_{=0} \)  we guarantee the existence of at least one linear isometry \( \sigma_0 : H \to \mathbb{R}^{d-1} \).

Moreover, by fixing any base point \( \mathbf{x} \in H \), we can define the isometry \( \sigma_{\alpha,\mathbf{x}} = \sigma_0(\cdot) - \sigma_0(\mathbf{x}) : H \to \mathbb{R}^{d-1} \). This isometry preserves the convexity of sets, mapping convex sets in \( H \) to convex sets in \( \mathbb{R}^{d-1} \). While the choice of \( \sigma_0 \) or \( \sigma_{\alpha,\mathbf{x}} \) is not unique, we will select one specific choice for our proof later on.

\begin{remark} \label{truncation_correspondence}

Fix any $\alpha>0$. Then there is a $1-1$ correspondence
$$\{\textup{Convex cones inside } H_{>0}\cup\{\f0\}\}\Longleftrightarrow\{\textup{Convex sets in } H_{=\alpha}\}.$$
For every convex cone $C$, it corresponds to $C \cap H_{=\alpha}$ which is a convex set in $H_{=\alpha}$; for a convex set $C_1$ in $H_{=\alpha}$, the set $\{\lambda\fu:\lambda \geqslant 0,\fu \in C_1\}$ is a convex cone in $H_{>0}\cup\{\f0\}$. Thus the intersection of such a cone with $H_{=\alpha}$ contains enough information of the cone. For any $\alpha>0$, we will call $H_{=\alpha}$ a truncating hyperplane, and this $1-1$ correspondence the \textit{truncation correspondence}.
\end{remark}

We fix a truncating hyperplane $H$. Under the truncation correspondence, closed convex cones in $H_{>0}\cup\{\f0\}$ corresponds to compact convex sets in $H$, and circular cone corresponds to closed balls with center the unique point in $\RR\fa \cap H$. In particular, all closed convex cones in $H_{>0}\cup\{\f0\}$ are pointed.

\begin{proposition} \label{trunc_bound}
Let $C$ be a closed convex cone pointed at direction $\fa$ in $H_{\fa,>0}\cup\{\f0\}$. For any $\alpha>0$, $C \cap H_{\fa,<\alpha}$ is a bounded set; there exists $c_1,c_2>0$ that only depends on $C$ and does not depend on $\alpha$ such that $C \cap B(\f0,c_1\alpha) \subset C \cap H_{\fa,<\alpha} \subset C \cap B(\f0,c_2\alpha)$.   
\end{proposition}

\begin{proof}
Using \autoref{truncation_correspondence} and the fact that $\alpha H_{\fa, <1}= H_{\fa, < \alpha}$, we get the result.
\end{proof}
Fix the pointed direction $\fa$, then for any $\alpha>0$, we will call $H_{<\alpha}$ a truncating half space.
\begin{definition}
Let $C$ be a pointed full-dimensional cone. Define $C^\bu=C^\circ\cup\{\f0\}$.  
\end{definition}
\begin{remark}
If $C$ is a cone, then $C^\bu$ is also a cone which is not necessarily closed.    
\end{remark}
\begin{proposition}
Let $C$ be a cone pointed at direction $\fa$. Then $C$ is upward to direction $\fa$ if and only if $\fa \in C^\circ$, if and only if $\lambda\fa \in C^\circ$ for all $\lambda>0$.   
\end{proposition}
\begin{proof}
Since $\fa \neq \f0$ and $C^\bu=C^\circ\cup \{\f0\}$ is a cone, $\fa \in C^\circ$ if and only if $\lambda\fa \in C^\circ$ for $\lambda>0$. If the circular cone $C_{\fa,c} \subset C$, then $\fa \in C_{\fa,c}^\circ \subset C^\circ$. On the other hand, suppose $\fa \in C^\circ$. Choose the truncating hyperplane $H$ passing through $\fa$, then $\fa \in (C \cap H)^{\circ H}$, where $\circ H$ denotes the interior in $H \cong \RR^{d-1}$. Thus there is a ball of dimension $d-1$, say $B_H(\fa,c_1) \subset C \cap H$. Therefore $\Cone(B_H(\fa,c_1)) \subset C$, and $\Cone(B_H(\fa,c_1))$ is a circular cone, so we are done.  
\end{proof}

\subsection{Semigroups, semigroup ideals, and cones}
A semigroup  $ S \subset \RR^d$ is defined as a set containing the zero vector $\f0 \in S$ and satisfying $S+S \subset S$. We say that $S_1$ is a subsemigroup of $S_2$, if it is a subset of $S_2$, and the addition operations of both semigroups are compatible through the embedding map.

Let $S_1 \subset S_2$ be two semigroups in $\RR^d$. We say $T \subset S_2$ is an $S_1$-ideal in $S_2$ if $S_1+T \subset T$. Each semigroup $S$ generates a cone $\Cone(S)$. Conversely, any cone $C$ in $\RR^d$ can also be viewed as a semigroup. Note that $C^\bu \subset C$ is a cone is a cone and serves as a subsemigroup of $C$. Therefore,  we can discuss  $C^\bu$-ideals within $C$. It is important to note that $C^\bu=C^\circ\cup \{\f0\}$; thus, for any  $T \subset C$, $C^\bu+T \subset T$ if and only if $C^\circ+T \subset T$. Consequently, we will say that $T$ is a $C^\circ$-ideal in $C$, if it is a $C^\bu$-ideal in $C$.

We denote the group generated by a semigroup $S$ by $\ZZ S$. If $S$ is a subsemigroup of $\ZZ^d$, then $\ZZ S$ is a free abelian group of rank at most $d$. If  it has rank $d$, then there exist a linear transformation $\sigma$ with integer coefficients such that $\sigma(\ZZ^d)=\ZZ S$. Therefore, we can assume, up to a linear transformation, that $\ZZ S=\ZZ^d$. We define $S$ as a standard semigroup if $\ZZ S=\ZZ^d$.

For a cone $C \subset \RR^d$, denote $\reg(C)=C\cap \ZZ^d$. If $S \subset \ZZ^d$ is a semigroup, then $S \subset \reg(\Cone(S))$. By definition, for any $n \in \NN$, $1/n\reg(C)=C\cap 1/n\ZZ^d$.

The following sequences of approximation lemmas deal with a semigroup $S$ and the lattice points inside $\Cone(S)$. Roughly speaking, if $S$ is standard, then up to finitely many points, $S$ fills up all the lattice points that are inside $\Cone(S)$ and not so close to $\partial(\Cone(S))$.
\begin{lemma} \label{union_interior_closed_upward_cone}
Let $C$ be a full-dimensional pointed cone. Then $C^\bu$ is the union of the interior of all closed pointed subcone in $C^\bu$. Moreover if $C$ is upward, then $C^\bu$ is also the union of the interior of all closed upward subcone in $C^\bu$.   
\end{lemma}
\begin{proof}
Choose a truncating hyperplane $H$. According to the truncation correspondence, the first statement is equivalent to the following: for every convex compact set $K$ with $K^\circ \neq \emptyset$, the interior $K^\circ$ can be expressed as the union of the interiors of all compact convex subsets of $K^\circ$. This holds true because, for any point  $\fx \in K^\circ$, there exists a closed ball centered at 
$\fx$ that lies entirely within $K^\circ$.

The second statement is equivalent to the following: for every convex compact set $K$ with $K^\circ \neq \emptyset$ and for any point $\fx_0 \in K^\circ$, $K^\circ$ is the union of the interior of all compact 
convex subsets of $K^\circ$ that contains $\fx_0$ in its interior. This is true because, given any $\fx \in K^\circ$, we can find two closed balls $B_1$ and $B_2$ contained in $K^\circ$- one centered at $\fx$ and the other centered at $\fx_0$. The convex hull of $B_1\cup B_2$ then forms a closed convex set within $K^\circ$ that contains both $\fx$ and $\fx_0$ in its interior.   
\end{proof}
\begin{lemma}\cite[Theorem 1]{kaveh2012newton} \label{K_K_imp_lemma}
Let $S \subset \ZZ^d$ be a semigroup with $\ZZ S =\ZZ^d$. Assume $C=\Cone(S)$ is pointed. Fix a full-dimensional closed subcone $C' \subset C^\bu$. Then there exists $N>0$ such that whenever $\fx \in \reg(C')$ and $\|\fx\| \geqslant N$, $\fx \in S$. 
\end{lemma}
\begin{corollary} \label{K_K_appl}
Let $S \subset \ZZ^d$ be a semigroup with $\ZZ S =\ZZ^d$. Assume $C=\Cone(S)$ is pointed. Fix a full-dimensional closed subcone $C' \subset C^\bu$ and $\delta>0$. Then there exists $N>0$ such that whenever $n \geqslant N$, $\fx \in 1/n\reg(C')$ and $\|\fx\| \geqslant \delta$, $\fx \in 1/nS$. 
\end{corollary}
\begin{proof}
Apply \autoref{K_K_imp_lemma} to $n\fx$ and note that $n\delta \to \infty$ as $n \to \infty$.    
\end{proof}

\begin{corollary}\label{K_K_appl2}
Let $S \subset \ZZ^d$ be a semigroup with $\ZZ S =\ZZ^d$. Assume $C=\Cone(S)$ is pointed. Fix a point $0 \neq \fx \in C^\circ$ and a full-dimensional closed subcone $C' \subset C^\bu$ such that $\fx \in C'^\circ$. 
\begin{enumerate}
\item We choose a sequence of elements $\fy_k$ such that $\fy_k \to \fx$ as $k \to \infty$, and an increasing sequence of positive integers $m_k \to \infty$. Then there exists $N>0$ such that whenever $k \geqslant N$, $[\fy_k]_{m_k} \in 1/m_kS\cap C'^\circ$.
\item We choose a bisequence of elements $\fy_{k,l}$ such that $\fy_{k,l} \to \fx$ as $k \to \infty$ and this convergence is uniform with respect to $l$, and choose an increasing sequence of positive integers $m_k \to \infty$. Then there exists $N>0$ independent of $l$ such that whenever $k \geqslant N$, $[\fy_{k,l}]_{m_k} \in 1/m_kS\cap C'^\circ$.
\end{enumerate}
\end{corollary}
\begin{proof}
(1): we see $\|\fy_k-[\fy_k]_{m_k}\| \leqslant d/m_k \to 0$ and $\|\fx-\fy_k\| \to 0$, so $[\fy_k]_{m_k} \to \fx$ as $k \to \infty$. First, this implies $[\fy_k]_{m_k} \in C'^\circ$ for sufficiently large $k$, say $k \geqslant N_1$. Second, this also implies for sufficiently large $k$, say $k \geqslant N_2$, $\|[\fy_k]_{m_k}\| \geq 1/2\|\fx\|>0$. We take $\delta=1/2\|\fx\|$ in \autoref{K_K_appl}, then for some sufficiently large $N_3$, $m_k \geq N_3$ implies $[\fy_k]_{m_k} \in 1/m_kS$. We assume whenever $k \geq N_4$, $m_k \geq N_3$. Then for $N=\max\{N_1,N_2,N_4\}$, $k \geq N$ implies $[\fy_k]_{m_k} \in 1/m_kS\cap C'^\circ$. 

For (2): just replace $\fy_{k}$ with $\fy_{k,l}$ in the argument above and note that all the inequalities are independent of $l$.
\end{proof}
We will usually use the above corollary in two cases: $m_k=k=n$ or $m_k=p^k,k=e$.

\subsection{OK valuation}
Fix $\fa \in \RR^d$, and assume the components of $\fa$ are $\QQ$-linearly independent elements in $\RR$. Then there is an embedding $\langle \bu,\fa \rangle:\ZZ^d \to \RR$. Since $\RR$ is totally ordered, this embedding induces a
total order on $\ZZ^d$, denoted by $\leqslant_{\fa}$, or $\leqslant$ when the context is clear. If, moreover, the components of $\fa$ are all positive, then for every $\fu \in \NN^d$, there are only finitely many $\fv \in \NN^d$ with $\fv \leqslant \fu$.

Now, suppose $\FF$ is a field. A valuation on $\FF$ is a map $\nu:\FF^\times \to G$ where $(G,+,\leqslant)$ is an ordered abelian group, satisfying the properties  $\nu(xy)=\nu(x)+\nu(y)$ and $\nu(x+y)\geqslant \min\{\nu(x),\nu(y)\}$. For any subset $N \subset \FF$, we denote $N^\times=N\ba \{0\}$, and define $\nu(N)=\nu(N^\times)$. In particular, $\nu(\FF)$ is a subgroup of $G$, called the value group of $\nu$. If $R$ is a domain and $\FF$ is its field of fraction, then $\nu(\FF)$ is the subgroup of $G$ generated by the semigroup $\nu(R)$.

For $g \in G$, denote $\FF_{\geqslant g}=\{x \in \FF^\times,\nu(x)\geqslant g\}\cup\{0\}$ and $\FF_{>g}=\{x \in \FF^\times,\nu(x)> g\}\cup\{0\}$. The valuation ring of $\nu$ is $(V_\nu,\fm_\nu,\kk_\nu)$. That is, $V_\nu=\FF_{\geqslant 0}$, $\fm_\nu=\FF_{>0}$, and $\kk_\nu=V_\nu/\fm_\nu$. For any $g \in G$, $\FF_{\geqslant g}$ and $\FF_{>g}$ are $V_\nu$-modules. For a local ring $(R,\fm,\kk) \subset \FF$, we say $R$ is dominated by $\nu$ if $R \subset V_\nu$ and $\fm \subset \fm_\nu$. In this case $\kk \to \kk_\nu$ is a field extension.

Now we discuss the concept of OK valuation c.f. \cite[Definition 3.1]{hernandez2018local}.
\begin{definition} \label{OK_val_def}
Let $(R,\fm,\kk)$ be a Noetherian local domain of dimension $d$ with fraction field $\FF$, and fix $\fa \in \RR^d$ which gives an embedding $\ZZ^d \subset \RR$ and an order  on $\ZZ^d$. A valuation $\nu:\FF^\times \to \ZZ^d$ is said to be \emph{OK relative to $R$} if it satisfies the following properties:
\begin{enumerate}
\item $\ZZ\nu(R)=\nu(\FF)=\ZZ^d$;
\item $R$ is dominated by $\nu$;
\item $\Cone(\nu(R))$ is pointed at direction $\fa$;
\item $[\kk_\nu:\kk]<\infty$; 
\item There exists $\fv \in \ZZ^d$ such that $R \cap \FF_{\geqslant n\fv} \subset \fm^n$ for any $n \in \NN$.
\end{enumerate}
\begin{definition} \label{OK_domain_defn}
    We say $R$ is an \emph{OK domain} if there exists a valuation $\nu$ on $\FF$ that is OK relative to $R$.
\end{definition}
\end{definition}
\begin{remark}
Condition (2) already implies that $\nu(R) \subset H_{\fa,>0} \cup \{0\}$. Furthermore, if $\nu(R)$ is finitely generated as a semigroup, then condition (2) also implies condition (3). However, it is important to note that $\nu(R)$ is not necessarily finitely generated in general.    
\end{remark}
\begin{definition}
Let $R$ be a domain, $\FF$ be its fraction field, $\nu$ be a valuation on $\FF$. For an $R$-ideal $I$, denote $\nu^c(I)=\nu(R)\ba\nu(I)$. If the value group of $\nu$ is $\ZZ^d$, then $\nu^c(I) \subset \ZZ^d$.   
\end{definition}

\subsection{Lattice points and cubes in \texorpdfstring{$\RR^n$}{Lg}} \label{des_of_L}
For $\fx=(x_1,\ldots,x_d) \in \RR^d$ and $\epsilon>0$, define $\Pi_{\fx,\epsilon}=\Pi_{1 \leqslant i \leqslant d}[x_i,x_i+\epsilon)$. It is a cube of edge length $\epsilon$. For  $U \subset \RR^d$, the characteristic function of $U$ is denoted by $\chi(U)$, i.e., $\chi(U)(\fx)=1$ if $\fx \in U$ and $\chi(U)(\fx)=0$ if $\fx \notin U$. Denote $\pi_{\fx,\epsilon}=\chi(\Pi_{\fx,\epsilon})$. For $n \in \NN$ and $L \subset (1/n)\ZZ^d \subset \RR^d$, define $\nabla_{L,1/n}=\cup_{\fx \in L}\Pi_{\fx,1/n}$, which is a disjoint union of cubes. Define $F_{L,1/n}=\chi(\nabla_{L,1/n})=\sum_{\fx \in L}\pi_{\fx,1/n}$. For $\fx=(x_1,\ldots,x_d) \in \mathbb{R}^d$, denote $\lfloor \fx \rfloor=(\lfloor x_1 \rfloor,\ldots,\lfloor x_d \rfloor)$. For $n \in \NN$, denote $[\fx]_n=1/n\lfloor n\fx \rfloor$. 

\begin{proposition} \label{prop_of_[]_and_nabla}
For any $\fx \in \RR^d$ and $n \in \NN$, we have:
\begin{enumerate}
\item $\|\fx-[\fx]_n\| \leqslant d/n$;
\item $[\fx]_n \in 1/n\mathbb{Z}^d$;
\item For $L \subset 1/n\ZZ^d$, $\fx \in \nabla_{L,1/n}$ if and only if $[\fx]_n \in \nabla_{L,1/n}$, if and only if $[\fx]_n \in L$.
\end{enumerate}
\end{proposition}
\begin{proof}
The proof easily follows using the description mentioned in \autoref{des_of_L}.
\end{proof}
\subsection{Some theorems in real analysis}
We introduce some theorems in real analysis that will be used in the following parts of the paper. They can be found in a typical textbook on real analysis. The first one is Fatou's lemma:
\begin{lemma}[Fatou's lemma] \label{Fatou's lemma}
Let $F_n(x)$ be a sequence of nonnegative measurable functions on $\RR^d$, then
$$\int_{\RR^d} \displaystyle \liminf_{n \to \infty}F_n(x)dx \leqslant \displaystyle \liminf_{n \to \infty}\int_{\RR^d}F_n(x)dx. $$
\end{lemma}
Under some more assumptions on the sequence $F_n$, we have an inequality in the other direction.
\begin{lemma}[Reverse Fatou's lemma] \label{reverse Fatou's lemma}
Let $F_n(x)$ be a sequence of nonnegative measurable functions on $\RR^d$. Suppose there exists a bounded set $K$ such that $F_n(x)=0$ for $x \notin K$ and all $n \in \NN$, and there exists a constant $C$ such that $F_n(x) \leqslant C$ for $x \in K$. then
$$\int_{\RR^d} \displaystyle \limsup_{n \to \infty}F_n(x)dx \geqslant \displaystyle \limsup_{n \to \infty}\int_{\RR^d}F_n(x)dx. $$    
\end{lemma}
\begin{proof}
We may replace $K$ by a closed ball to assume $K$ is a measurable set. Then apply Fatou's lemma to the sequence of functions $C\chi(K)-F_n$, and notice that $\int_{\RR^d}C\chi(K)(x)dx=C\vol(K)<\infty$.    
\end{proof}
We also have the bounded convergence lemma.

\begin{lemma}[Bounded convergence lemma]
Let $F_n(x)$ be a sequence of nonnegative measurable functions on $\RR^d$. Suppose there exists a bounded set $K$ such that $F_n(x)=0$ for $x \notin K$ and all $n \in \NN$, and there exists a constant $C$ such that $F_n(x) \leqslant C$ for $x \in K$. Also, suppose $F(x)= \displaystyle\lim_{n \to \infty}F_n(x)$ exists almost everywhere. Then
$$\int_{\RR^d}\lim_{n \to \infty}F_n(x)dx=\lim_{n \to \infty}\int_{\RR^d}F_n(x)dx. $$    
\end{lemma}
\begin{proof}
Follows from \autoref{Fatou's lemma} and \autoref{reverse Fatou's lemma}.
\end{proof}

\section{Pointed cone, upward cone, and height function} \label{section_3}
In this section, we will examine the concepts of pointed cones and upward cones. We will establish that every pointed cone can be viewed as an upward cone in the direction of a carefully selected vector. Additionally, we will characterize the interior of every upward cone through a function defined on a hyperplane, known as the \emph{height function}. Furthermore, we will utilize the height function to describe important properties of the cone ideals within an upward cone.
\begin{lemma}
Let $C$ be a convex cone in $\RR^d$. Assume $C$ is full-dimensional. Then $\partial C=\partial C^\circ$.    
\end{lemma}
\begin{proof}
Since $(C^\circ)^\circ=C^\circ$, $\partial C^\circ=\overline{C^\circ}\ba C^\circ \subset \overline{C}\ba C^\circ=\partial C$. If $\fx \in \partial C$, choose any $\epsilon>0$. There exists a closed ball $B \subset C^\circ$, and there exists $\fy \in C$ with $\dist(\fx,\fy)<\epsilon/2$. Note that the convex hull $K$ of $B \cup y$ lies in $C$, so $K^\circ \subset C^\circ$. But it contains a right circular cone with vertex $\fy$, so there exists $\fz \in K^\circ \subset C^\circ$ with $\dist(\fy,\fz)<\epsilon/2$. Thus we find $\fz \in C^\circ$ with $\dist(\fx,\fz)<\epsilon$. Since $\epsilon$ is arbitrary, $\fx \in \partial C^\circ$.   
\end{proof}
\begin{lemma} \label{pointed_implies_upward}
Let $C$ be a full-dimensional closed cone pointed at the direction of $\fa$. Then there exists another vector $\mathbf{b}$ such that $C$ is upward to the direction of $\mathbf{b}$.     
\end{lemma}
\begin{proof}
In $\RR^1$, the only cones are the positive or negative rays, and the statement is obviously true. We start with $\mathbb{R}^2$. It is easy to see any full-dimensional closed convex cone $C$ in $\mathbb{R}^2$ is the interior and the boundary of an angle of radian $\theta$. Since $C$ is pointed, $C-\f0$ lies in the open half plane $ \langle \fx,\fa \rangle >0$. Therefore $\theta<\pi$. Thus we can choose a vector $\mathbf{b}$ on the bisector of this angle, therefore $\mathbf{b} \in C^\circ$, and $C-\f0$ still lies in the open half plane $ \langle\fx,\mathbf{b}\rangle>0$. 

Now assume we proved the claim in dimension $d-1 \geqslant 2$. Let $C$ be a closed convex cone in $\RR^d$ pointed at the direction of $\fa$. If $\fa \in C^\circ$, then $C$ is already upward to the direction of $\fa$ and there is nothing to prove. If $\fa \in \partial C=\partial C^\circ$, then we can choose $\mathbf{b} \in C^\circ$ sufficiently close to $\fa$. We fix two balls of finite radius $B(\f0,r_1) \subset B(\f0,r_2)$ where $0<r_1<r_2$, then the following two conditions are equivalent since $C$ is a cone:
\begin{enumerate}
\item For any $\fx \in C-\f0$, $\langle \fx,\fa \rangle>c_1\|\fx\|;$
\item For any $\fx \in C-\f0$ and $r_1\leqslant \|\fx\|\leqslant r_2$, $\langle\fx,\fa\rangle>c_1\|\fx\|.$
\end{enumerate}
That is, we only need to verify the upward condition on a spherical shell inside the convex cone. However, the set $K=C \cap (\overline{B(\f0,r_2)}-B(\f0,r_1))$ is a compact set. So any continuous function on $K$ must achieve a minimum. So there exist a $\fy \in C \cap (\overline{B(\f0,r_2)}-B(\f0,r_1))$ depending on $\mathbf{b}$ such that $\frac{\langle\fx,\mathbf{b}\rangle}{\|\fx\|} \geqslant \frac{\langle \fy,\mathbf{b}\rangle}{\|\fy\|}$. We see $|\frac{\langle\fy,\mathbf{b}\rangle}{\|\fy\|}-\frac{\langle\fy,\fa\rangle}{\|\fy\|}|=\frac{|\langle\fy,\mathbf{b}-\fa\rangle|}{\|\fy\|}\leqslant \|\mathbf{b}-\fa\|$ and this upper bound is independent of $\fy$. So let $\|\mathbf{b}-\fa\|<c_1$ and $c_3=c_1-\|\mathbf{b}-\fa\|>0$, we have $\langle\fx,\mathbf{b}\rangle\geqslant c_3\|\fx\|$ on $K$, hence on all of $C$. Therefore $C$ is pointed at the direction of $\mathbf{b}$. But by the choice of $\mathbf{b}$ it lies in $C^\circ$, so $C$ is upward to the direction of $\mathbf{b}$.

Now we assume $\fa \notin \overline{C}=C$. After rescaling if necessary we may assume $\|\fa\|=1$. We choose a truncating hyperplane $H=H_{\fa,=1}$, and fix a map $\sigma:H \to \mathbb{R}^{d-1}$ which is linear, isometric and maps $\fa$ to the origin $\f0$. The set $C \cap H$ is a compact convex set, so $C_1=\sigma(C\cap H)$ is a compact convex set in $\mathbb{R}^{d-1}$ with $\f0 \notin C_1$. And since $C$ is full-dimensional, $C_1$ has a nonempty interior. The convex cone $C'$ generated by $C_1$ is full-dimensional and pointed because $\f0 \notin C_1$ and $C_1$ is full-dimensional and compact. Therefore by induction, we can choose $\mathbf{a'} \in C'^\circ$ such that $C'$ is pointed at the direction of $\mathbf{a'}$. After rescaling, we may assume $\mathbf{a'} \in C_1^\circ$. So for any $\fy' \in C'$, $\langle\fy',\mathbf{a'}\rangle\geqslant c_1\|\fy'\|$. Let $\mathbf{b}=\sigma^{-1}\mathbf{a'}$. We see for any $\fy \in C\cap H$,
$\|(\fy-\fa)-(\fb-\fa)\|=\|\fy-\fb\|=\|\sigma(\fy)-\sigma(\fb)\|$. This implies $\langle \fy - \fa, \fb-\fa \rangle=\langle \sigma(\fy), \sigma(\fb)\rangle \geqslant c_{1} \|\sigma(\fy)\|= c_{1} \| \fy -\fa\|$, as $\sigma(\fa)=0$. Therefore we have, $\langle\fy-\fa,\mathbf{b}-\fa\rangle \geqslant c_1\|\fy-\fa\|$. Note that $\fy-\fa$ and $\mathbf{b}-\fa$ are both perpendicular to $\fa$, which implies
$$\langle\fy,\mathbf{b}\rangle=\langle\fy-\fa+\fa,\mathbf{b}-\fa+\fa\rangle=\langle\fy-\fa,\mathbf{b}-\fa\rangle+\|\fa\|^2.$$
So $\langle\fy,\mathbf{b}\rangle \geqslant c_1\|\fy-\fa\|+\|\fa\|^2$. Now the function $\fy \to (c_1\|\fy-\fa\|+\|\fa\|^2)/\|\fy\|$ is continuous and positive everywhere on the compact set $C \cap H$, so it has a positive minimum. Therefore, there exists $c_2>0$ such that $c_1\|\fy-\fa\|+\|\fa\|^2\geqslant c_2\|\fy\|$. Thus on $C \cap H$, $\langle\fy,\mathbf{b}\rangle\geqslant c_2\|\fy\|$. Since this is a linear inequality and $C$ is a cone, this is true on all of $C$, so $C$ is pointed at the direction of $\mathbf{b}$. Finally by our choice $\mathbf{a'} \in C_1^\circ$, so $\mathbf{b}=\sigma^{-1}\mathbf{a'} \in C^\circ$. So $C$ is upward to the direction of $\mathbf{b}$.
\end{proof}
\begin{remark}
Note that \autoref{pointed_implies_upward} implies that if we define an OK valuation such that the cone generated by the valuation of ring elements is pointed at the direction of $\fa$, then it is upward to direction $\mathbf{b}$ for some $\mathbf{b}$. Note that in this case, the order on $\nu(R)$ is still given by $\fa$.   
\end{remark}
Next we will see that we can describe an ideal in the convex cone up to interior using a single function on a hyperplane. \begin{setup} \label{convex_upward_cone_setup}

Let $C$ is a convex cone upward to the direction of $\fb$ which is not necessarily closed.
\end{setup} 
\begin{proposition} \label{ht_fun_existence}
\sloppy Adopt \autoref{convex_upward_cone_setup}. Let $\Delta$ be a nonzero $C^\circ$-ideal in $C$. Let $H$ be the hyperplane $\{\fx:\langle\fx,\fb\rangle=0\}$. Then for every $\fx \in H$, there exists a unique $\varphi(\fx)\geqslant 0$ such that $\fx+t\fb \in \Delta$ for $t>\varphi(\fx)$ and $\fx+t\fb \notin \Delta$ for $t<\varphi(\fx)$.    
\end{proposition}
\begin{proof}
\sloppy First we claim for any $\fx$ there exist $t$ with $\fx+t\fb \in \Delta$ and $\fx+t\fb \notin \Delta$. Let $t<0$, then  $\langle\fx+t\fb,\fb\rangle=t\|\fb\|^2<0$, so $\fx+t\fb \notin C$ since $C$ is pointed, so $\fx+t\fb \notin \Delta$. Choose $0 \neq \fy \in \Delta$. It suffices to prove there exist $t$ with $\fx+t\fb \in \fy+C^\circ$, or $\fx-\fy+t\fb \in C^\circ$. Since $\fb \in C^\circ$, for large enough $t$, we have $B(\fb,2\|\fx-\fy\|/t) \in C^\circ$. So $1/t(\fx-\fy)+\fb \in C^\circ$, so $\fx-\fy+t\fb \in C^\circ$ and $\fx+t\fb \in \fy+C^\circ \in \Delta$. Now consider two sets $R_1=\{t:\fx+t\fb \in \Delta\}$, and $R_2=\{t:\fx+t\fb \notin \Delta\}$. Since $\Delta$ is a $C^\circ$-ideal, for $t \in R_1$, $s>0$, $t+s \in R_1$. And $R_2=\mathbb{R}\ba R_1$, so for $t \in R_2$, $s>0$, $t-s \in R_2$. So $R_1,R_2$ forms a Dedekind cut of the real line $\mathbb{R}$, so there exist a real number $\varphi(\fx)$ depending on $\fx$ such that $R_1=[\varphi(\fx),\infty),R_2=(-\infty,\varphi(\fx))$ or $R_1=(\varphi(\fx),\infty),R_2=(-\infty,\varphi(\fx)]$. Since $t<0$ implies $\fx+t\fb \notin \Delta$, $(-\infty, t) \subset (-\infty,\varphi(\fx)]$, so $\varphi(\fx) \geq 0$ and we are done.
\end{proof}
\begin{definition} \label{ht_func_definition}

We can consider the function $\varphi(\fx)$ as mentioned in \autoref{ht_fun_existence} on $H$ depending on the $C^\circ$-ideal $\Delta$. We call it the \textbf{height function} of $\Delta$, denoted by $\varphi_\Delta(\fx)$.
\end{definition}
\begin{proposition} \label{ht_func_bound}
Adopt \autoref{convex_upward_cone_setup}. Assume  $0<c_2<\|\fb\|$ satisfies $\{\fx:\langle\fx,\fb\rangle \geqslant c_2\|\fx\|\} \subset C^\bu$. Then $0\leqslant\varphi_C(\fy)\leqslant \varphi_{C^\circ}(\fy)\leqslant c_3\|\fy\|$ for some constant $c_3$ and any $\fy \in H$.   
\end{proposition}
\begin{proof} The first inequality is trivial because $C^\circ \subset C$.
For $t \in \mathbb{R}$ and $\fy \in H$, we have
$$\frac{\langle\fy+t\fb,\fb\rangle}{\|\fy+t\fb\|}=\frac{t\|\fb\|^2}{\sqrt{\|\fy\|^2+t^2\|\fb\|^2}}.$$
For $0<c_2<\|\fb\|$, solving the inequality in $t$
$$\frac{t\|\fb\|^2}{\sqrt{\|\fy\|^2+t^2\|\fb\|^2}}>c_2$$
is equivalent to solving
$$t>0,t^2(\|\fb\|^4-c^2_2\|\fb\|^2)>c^2_2\|\fy\|^2.$$
Since we assume $c_2<\|\fb\|$, the above inequality has solution
$$t>\frac{c_2}{\sqrt{\|\fb\|^4-c^2_2\|\fb\|^2}}\|\fy\|,$$
and whenever $t$ satisfies this inequality, $\langle\fy+t\fb,\fb\rangle > c_2\|\fy+t\fb\|$, so $\fy+t\fb \in C^\circ$ and $t \geqslant \varphi_{C^\circ}(\fy)$. Setting $c_3=\frac{c_2}{\sqrt{\|\fb\|^4-c^2_2\|\fb\|^2}}$, we are done.
\end{proof}
\begin{proposition} \label{ht_func_char}
Adopt \autoref{convex_upward_cone_setup}. A function $\varphi:H \to \RR$ can be realized as the height function of some $C^\circ$-ideal in $C$ if and only if it satisfies:
\begin{enumerate}
\item $\varphi(\fx) \geqslant \varphi_C(\fx)$ for all $\fx \in H$.
\item $\varphi(\fx)+\varphi_{C^\circ}(\fy) \geqslant \varphi(\fx+\fy)$ for all $\fx,\fy \in H$.
\end{enumerate}
\end{proposition}
\begin{proof}
\sloppy Suppose $\varphi=\varphi_{\Delta}$ for some $C^\circ$-ideal $\Delta$ in $C$. Then $\Delta \subset C$ implies $\varphi(\fx) \geqslant \varphi_C(\fx)$ for all $\fx \in H$. Choose any $t>\varphi(\fx)$ and $s>\varphi_{C^\circ}(\fy)$, then $\fx+t\fb \in \Delta$ and $\fy+s\fb \in C^\circ$. Since $\Delta$ is a $C^\circ$-ideal, $\fx+\fy+(s+t)\fb \in \Delta$, so $s+t \geqslant \varphi(\fx+\fy)$. Since $s,t$ are arbitrary, $\varphi(\fx)+\varphi_{C^\circ}(\fy) \geqslant \varphi(\fx+\fy)$.

Suppose the function $\varphi$ satisfies (1) and (2). Define
$$\Delta=\{\fx+t\fb:\fx \in H,t>\varphi(\fx)\},$$
then the height function of $\Delta$ is just $\varphi$, and $\varphi(\fx) \geqslant \varphi_C(\fx)$ implies that $\Delta \subset C$. So it suffices to prove that $\Delta$ is a $C^\circ$-ideal in $C$. Note that $\RR\fb$ is the orthogonal complement of $H$, so every element in $\RR^d$ has the form $\fx+t\fb$ for some $\fx \in H$ and $t \in \RR$. Choose $\fx+t\fb \in \Delta$ and $\fy+s\fb \in C^\circ$, then by definition of $\Delta$,  $ t>\varphi(\fx)$, and by definition of $\varphi_{C^\circ}$, $s \geqslant \varphi_{C^\circ}(\fy)$, so $t+s>\varphi(\fx)+\varphi_{C^\circ}(\fy)\geqslant \varphi(\fx+\fy)$, so $\fx+t\fb+\fy+s\fb=\fx+\fy+(s+t)\fb \in \Delta$. So $\Delta$ is a $C^\circ$-ideal.
\end{proof}
\begin{remark}
Actually we can prove that if $C$ is closed and upward, every $C^\circ$-ideal $\Delta$ in $C$ corresponds to a height function satisfying the above conditions and a subset $H_0$ of $H$ in the sense that
$$\Delta=\{\fx+t\fb,\fx \in H,t>\varphi(\fx)\textup{ if }x \notin H_0, t \geqslant \varphi(\fx)\textup{ if }x \in H_0\}.$$
We will not use this fact in our proof.
\end{remark}
\begin{proposition} \label{ht_func_Lipschitz}
Adopt \autoref{convex_upward_cone_setup}. Let $\Delta$ be a $C^\circ$-ideal in $C$. Assume for some $0<c_2< \|\fb\|$, the circular cone $\{\fx:\langle\fx,\fb\rangle \geqslant c_2\|\fx\|\}$ lies in $C^\circ$. Then there exists a positive constant $c_3$ such that for any $\fx,\fy \in H$, $|\varphi_\Delta(\fx)-\varphi_\Delta(\fy)|\leqslant c_3\|\fx-\fy\|$, i.e., $\varphi_\Delta$ is Lipschitz continuous with Lipschitz constant $c_3$. 
\end{proposition}
\begin{proof}
By \autoref{ht_func_char}, we have
$$\varphi_{\Delta}(\fx)+\varphi_{C^\circ}(\fy-\fx)\geqslant \varphi_{\Delta}(\fy)$$
and
$$\varphi_{\Delta}(\fy)+\varphi_{C^\circ}(\fx-\fy)\geqslant \varphi_{\Delta}(\fx).$$
Thus
$$|\varphi_{\Delta}(\fx)-\varphi_{\Delta}(\fy)|\leqslant \max\{\varphi_{C^\circ}(\fy-\fx),\varphi_{C^\circ}(\fx-\fy)\}.$$
By \autoref{ht_func_bound}, there exists $c_3>0$ such that
$$\max\{\varphi_{C^\circ}(\fy-\fx),\varphi_{C^\circ}(\fx-\fy)\}\leqslant c_3\|\fx-\fy\|.$$
So we are done.
\end{proof}
\begin{theorem} \label{ht_func_prop_for_closure_interior_boundary}
Adopt \autoref{convex_upward_cone_setup}. Let $\Delta$ be a $C^\circ$-ideal in $C$. Then:
\begin{enumerate}
\item $\Delta^\circ=\{\fx+t\fb:\fx \in H,t>\varphi_\Delta(\fx)\};$
\item $\overline{\Delta}=\{\fx+t\fb:\fx \in H,t\geqslant\varphi_\Delta(\fx)\};$
\item $\partial\Delta=\{\fx+t\fb:\fx \in H,t=\varphi_\Delta(\fx)\}$, and this set has zero measure.
\end{enumerate}
\end{theorem}
\begin{proof}
(1) We denote $\Delta_1=\{\fx+t\fb,\fx \in H,t>\varphi_\Delta(\fx)\}$, $\Delta_2=\{\fx+t\fb,\fx \in H,t \geqslant \varphi_\Delta(\fx)\}$, $\Delta_3=\{\fx+t\fb,\fx \in H,t<\varphi_\Delta(\fx)\}$. Then $\Delta_1 \subset \Delta$ and $\Delta_3 \subset \mathbb{R}^d\ba\Delta$. Since $\varphi_{\Delta}$ is a continuous function, $\Delta_1$ is an open set. So $\Delta_1=\Delta_1^\circ \subset \Delta^\circ$. We see $\Delta \subset \Delta_1 \cup \Delta_2$. But for any $(\fx+\varphi_\Delta(\fx)\fb) \in \Delta_2$, we can choose $s<\varphi_\Delta(\fx)$ which is arbitrarily close to $\varphi_\Delta(\fx)$. Then $\fx+s\fb\in \Delta_3$, and it can be arbitrarily close to $\fx+\varphi_\Delta(\fx)\fb$, so any point in $\Delta_2 \ba \Delta_1$ cannot lie in $\Delta^\circ$. So $\Delta^\circ=\Delta_1^{\circ}=\Delta_1$.

(2) Similar to (1) we can prove $\Delta_3=(\mathbb{R}^d\ba\Delta)^\circ$. So taking complement we get $\Delta_1\cup\Delta_2=\overline{\Delta}$.

(3) The first equality is true by (1) and (2), and the graph of a continuous function always has zero measure.
\end{proof}
\begin{corollary} \label{ht_func_equality}
Adopt \autoref{convex_upward_cone_setup}. Let $\Delta$ be an $C^\circ$-ideal. Consider its interior $\Delta^\circ$ as an $C^\circ$-ideal in $C$, and its closure $\overline{\Delta}$ as an $C^\circ$-ideal in $\overline{C}$. Then $\varphi_\Delta(\fx)=\varphi_{\Delta^\circ}(\fx)=\varphi_{\overline{\Delta}}(\fx)$ for all $\fx$. Also, the sets $\Delta^\circ$ and $\overline{\Delta}$ are uniquely determined by the height function $\varphi$.   
\end{corollary}
Now we compare two $C^\circ$-ideals in $C$. 
\begin{corollary} \label{comp_two_C_circ_ideals}
Adopt \autoref{convex_upward_cone_setup}. Let $\Delta_1$, $\Delta_2$ be two $C^\circ$-ideals in $C$ with $\Delta_1 \subset \Delta_2$. Then:
\begin{enumerate}
\item $\varphi_{\Delta_1}(\fx) \geqslant \varphi_{\Delta_2}(\fx)$ for all $\fx$.
\item If $\varphi_{\Delta_1}(\fx)= \varphi_{\Delta_2}(\fx)$ for all $\fx$, then $\Delta_1$ and $\Delta_2$ have the same interiors and closures.
\item If $\varphi_{\Delta_1}(\fx) \neq \varphi_{\Delta_2}(\fx)$ for some $\fx$, then $\Delta^\circ_2 \nsubseteq \overline{\Delta_1}$.
\end{enumerate}
\end{corollary}
\begin{proof}
(1) is true by definition of $\varphi$. (2) is true by \autoref{ht_func_equality}. For (3), pick $\fx$ such that $\varphi_{\Delta_1}(\fx) \neq \varphi_{\Delta_2}(\fx)$, then $\varphi_{\Delta_1}(\fx)> \varphi_{\Delta_2}(\fx)$. Choose $t$ such that $\varphi_{\Delta_1}(\fx)>t>\varphi_{\Delta_2}(\fx)$ and let $\fy=\fx+t\fb$, then $\fy \in \Delta_2^\circ\backslash\overline{\Delta_1}$.  
\end{proof}

\section{Families of ideals} \label{section_4}
In this section, we will list all the conditions on the family of ideals we consider. We assume in this section $(R,\fm,\kk)$ is a Noetherian local ring with maximal ideal $\fm$ and residue field $\kk$. We will use the symbols $I,I_1,I_2,\ldots,J,J_1,J_2,\ldots$ for ideals in $R$. We denote $R^\circ=R\ba (\cup_{P \in \Min(R)}P)$, where $\Min(R)=\{P \in \Spec(R) \mid \dim R/P=\dim R\}$.

First we consider the general characteristic case. In this case, we consider families of ideals $I_\bu=\{I_n\}_{n \in \NN}$ indexed by $\NN$.
\begin{definition}
We say a family of ideals $I_\bu$ indexed by $\NN$ is a \emph{graded family} of ideals, if 
$I_0=R$, $I_mI_n \subset I_{m+n}$ for all $m,n \in \NN$. We say $I_\bu$ is a \emph{weakly graded family} of ideals if there exists $c \in R^\circ$ independent of $m,n$ such that $cI_mI_n \subset I_{m+n}$ for all $m,n \in \NN$.   
\end{definition}
\begin{remark}
If $R$ is a domain, then $c \in R^\circ$ is equivalent to $0 \neq c \in R$.   
\end{remark}
Then we consider the case where $R$ has prime characteristic $p>0$. In this setting, we will use $e$ for a nonnegative integer, and always denote $q=p^e$. We recall some notions for rings of characteristic $p$. For an ideal $I \subset R$, the $e$-th Frobenius power of $I$ is $I^{[p^e]}=(a^{p^e},a \in I)$. Denote the Frobenius endomorphism of $R$ by $F$. Then $F^e:R \to R$ is an endomorphism, and we consider the target as a module over the source ring $R$, denoted by $F^e_*R$. We say $R$ is $F$-finite if $F_*R$ is a finitely generated $R$-module. 
\begin{remark}
Note that if $R$ is an $F$-finite local domain of dimension $d$. Then $\rank_R F_*R=[F_*\kk:\kk]p^{ed}$.    
\end{remark}
We denote $\alpha(R)=\textup{log}_p[F_*\kk:\kk]$. Then 
$\rank_R F_*R=p^{\alpha(R)+d}$ and $\rank_R F^e_*R=p^{e(\alpha(R)+d)}$. If $\kk$ is perfect, then $\alpha(R)=0$.

For rings of characteristic $p$, we will also consider families of ideals $I_\bu=\{I_q\}_{e \in \NN}$ indexed by powers of $p$.
\begin{definition}
Let $R$ be a Noetherian local ring of characteristic $p$, $I_\bu$ is a family of ideals indexed by powers of $p$. We say:
\begin{enumerate}
\item $I_\bu$ is a \emph{$p$-family} of ideals, if for any $q=p^e$, $I^{[p]}_q \subset I_{pq}$.
\item $I_\bu$ is a \emph{weak $p$-family} of ideals, if there exists $c \in R^\circ$ such that for any $q=p^e$, $cI^{[p]}_q \subset I_{pq}$.
\item $I_\bu$ is an \emph{inverse $p$-family} of ideals, if for any $q=p^e$, $I_{pq} \subset I^{[p]}_q$.
\item $I_\bu$ is a \emph{weakly inverse $p$-family} of ideals, if there exists $c \in R^\circ$ such that for any $q=p^e$, $cI_{pq} \subset I^{[p]}_q$.

\end{enumerate}
\end{definition}

Now we consider rings of general characteristic.
\begin{definition} \label{bbl_defn}
Let $(R,\fm,\kk)$ be a Noetherian local ring with maximal ideal $\fm$.
\begin{enumerate}
\item Let $I_\bu$ be a family of ideals indexed by $\NN$. We say $I_\bu$ is \emph{bounded below linearly}, or BBL for short, if there exists a constant $c \in \NN$ such that $\mathfrak{m}^{cn} \subset I_n$ for all $n$.
\item Assume moreover $R$ has characteristic $p>0$. Let $I_\bu$ be a family of ideals indexed by powers of $p$. We say $I_\bu$ is \emph{bounded below linearly}, or BBL for short, if there exists a constant $c \in \NN$ such that $\mathfrak{m}^{cq} \subset I_q$ for all $q$.
\end{enumerate}
\end{definition}
By definition, the ideals in a BBL family are all $\fm$-primary.
\begin{lemma} \label{graded_p_all_BBL}
Let $(R,\fm,\kk)$ be a Noetherian local ring with maximal ideal $\fm$.
\begin{enumerate}
\item If $I_\bu$ is a graded family of ideals indexed by $\NN$ and $I_1$ is $\fm$-primary, then it is BBL.
\item If $I_\bu$ is a $p$-family of ideals indexed by powers of $p$ and $I_1$ is $\fm$-primary, then it is BBL.
\end{enumerate}    
\end{lemma}
\begin{proof}
(1): Choose $c \in \NN$ such that $\fm^c \subset I_1$. Since $I_\bu$ is a graded family, $\fm^{cn}\subset I^n_1 \subset I_n$.  
(2): Choose $c \in \NN$ such that $\fm^c \subset I_1$. Assume $I_1$ is generated by $\mu$ elements, then for any $q=p^e$, $I^{\mu q} \subset I^{[q]}$ by pigeon hole principle. Since $I_\bu$ is a $p$-family, $\fm^{cq\mu}\subset I^{q\mu}_1 \subset I^{[q]}_1 \subset I_q$. 
\end{proof}
Now we can talk about examples of families. The following examples are families indexed by $\NN$.
\begin{example}
Take an ideal $I \subset R$. Then $I_{\bu}=\{I^n\}_{n \in \NN}$ is a graded family of ideals.    
\end{example}

It is apparent that one may employ standard algebraic techniques to generate new families from the given ones. We list few propositions to demonstrate that.

\begin{proposition}
Let $\hat{\bu}$ be an operation on ideals. That is, for every $R$-ideal $I$ we assign an ideal $\hat{I}$. Suppose for any ideal $I,J$, $\widehat{I} \cdot \widehat{J} \subset \widehat{IJ}$. Then for any graded family $I_\bu$, $\widehat{I_\bu}=\{\widehat{I_n}\}_n$ is a graded family. If moreover we have $\widehat{cI} \subset c \cdot \widehat{I}$, then for any weakly graded family $I_\bu$, $\widehat{I_\bu}$ is a weakly graded family.
\end{proposition}
\begin{proof}
This follows from definition.
\end{proof}
\begin{example}
The following operations satisfies $\widehat{I} \cdot \widehat{J}\subset \widehat{IJ}$, therefore plugging in the graded family $I_{\bu}$ produces graded families.
\begin{enumerate}
\item Taking integral closure $I \to \overline{I}$ produces $\overline{I_{\bu}}$;
\item When $R$ has characteristic $p$, taking tight closure $I \to I^*$ produces $(I_{\bu})^*$;
\item Let $R \to S$ be any ring homomorphism and $I \to IS \cap R$ produces $I_{\bu} S \cap R$;
\item Take $S=W^{-1}R$, $W=R-\cup_{P \in \textup{Min}(I)} P$ in the above example, and if $I_{\bu}=\{I^n\}$, then it produces the symbolic power $I_{\bu}=\{I^{(n)}\}$;
\item For any ideal $J \subset R$, taking saturation with respect to $J$, $I \to I:J^\infty$ produces $I_{\bu}:J^\infty$.
\end{enumerate}
\end{example}

\begin{example} \label{weakly_graded_example_1}
Let $I_\bu$ be a graded family of $R$-ideals, $J$ be an $R$-ideal with $J \cap R^\circ \neq \emptyset$, then $I_\bu:J$ is a weakly graded family of $R$-ideals.    
\end{example}
\begin{example} \label{weakly_graded_example_2}
Let $I$ be an $R$-ideal with $I \cap R^\circ \neq \emptyset$. Assume $\alpha \in \RR$. Let $\lfloor\bu\rfloor$ be the floor function and $\lceil\bu\rceil$ be the ceiling function. Then $\{I^{\lfloor n\alpha \rfloor}\}$ is a weakly graded family, and $\{I^{\lceil n\alpha \rceil}\}$ is a graded family.     
\end{example}
\begin{remark}
Let $J_\bu=\{J_n\}_{n \in \NN}$ be an arbitrary family of ideals.

Let $I_n=\sum_{k \geqslant 1, n_1+\ldots+n_k=n} J_{n_1}\ldots J_{n_k}$, then $I_\bu$ is a graded family. It is the smallest graded family containing $J_\bu$.    
\end{remark}
Now we assume $R$ has characteristic $p>0$ and families are indexed by powers of $p$.
\begin{example}
Take an ideal $I \subset R$. Then $I^{[\bu]}=\{I^{[q]}\}_{e \in \NN}$ is a $p$-family of ideals, an inverse $p$-family of ideals.  
\end{example}
\begin{proposition}\label{extension_of_p_family}
Let $\hat{\bu}$ be an operation on ideals such that $\widehat{I^{[p]}} \subset (\widehat{I})^{[p]}$. Then for any $p$-family $I_\bu$, $\widehat{I_\bu}=\{\widehat{I_n}\}_n$ is a graded family. If moreover we have $\widehat{cI} \subset c \cdot \widehat{I}$, then for any weakly $p$-family $I_\bu$, $\widehat{I_\bu}$ is a weakly graded family.
\end{proposition}
\begin{proof}
This follows from definition.
\end{proof}
\begin{example} \label{weakly_inverse_p_example}
For an ideal $I \in R$, $(I^{[\bu]})^*=\{(I^{[q]})^*\}$ is a $p$-family of ideals. If $R$ has a test element, then it is also a weakly inverse $p$-family.     
\end{example}
\begin{example} \label{weakly_p_example}
For two ideals $I,J \subset R$, $I^{[q]}:J^\infty$ is a $p$-family of ideals. If $J \cap R^\circ \neq \emptyset$, then $I^{[q]}:J$ is a weakly $p$-family of ideals.   
\end{example}
\begin{example} \cite{tucker2012f}
Let $J$ be an ideal of $R$, and let $R$ be an $F$-finite local domain. Denote
$$I_e(J)=\{a \in R: \forall \phi \in \Hom_R(F^e_*R,R), \phi(F^e_*a) \in J\}.$$
Then $I_e(J)$ is a $p$-family.
\end{example}
\begin{example}
The following operation on $p$-families of ideals produces another $p$-family.
\begin{enumerate}
\item Let $R \to S$ be any map, $I_\bu$ is a $p$-family of $R$-ideals. Then $I_\bu S$ is a $p$-family of $S$-ideals.
\item Let $R \to S$ be any map, $I_\bu$ is a $p$-family of $S$-ideals. Then $I_\bu \cap R$ is a $p$-family of $R$-ideals.
\item Let $I_\bu$ be a $p$-family of $R$-ideals, $J$ be an $R$-ideal, then $\{I_q:J^{[q]}\}_{e \in \NN}$ is a $p$-family of $R$-ideals, and $\{I_n+J\}$ is a $p$-family of $R$-ideals.
\item Let $I_\bu,J_\bu$ be two $p$-families of $R$-ideals, then $I_\bu \cap J_\bu$, $I_\bu J_\bu$, $I_\bu+J_\bu$ are three $p$-families of $R$-ideals.
\end{enumerate}
\end{example}
\begin{proof}
This follows from \autoref{extension_of_p_family}.
\end{proof}
\begin{remark}
Let $J_\bu=\{J_q\}_{e \in \NN}$ be an arbitrary family of ideals. Denote $I_q=\sum_{0 \leqslant k \leqslant e} J^{[p^{e-k}]}_{p^k}$, then $I_\bu$ is a $p$-family. It is the smallest $p$-family containing $J_\bu$.    
\end{remark}
We come back to general characteristic case. Among all the examples above in general characteristic or characteristic $p$, we will point out one particular case that allows us to do reduction in the last chapter.
\begin{theorem} \label{weakly_p_weakly_graded_stays_same_under_projection}
Let $R \to S$ be a map that maps $R^\circ$ to $S^\circ$. Let $I_\bu$ be a weakly graded family or a weakly $p$-family or a weakly inverse $p$-family, then $I_\bu S$ is a family of the same type. In particular, this holds for the natural projection $R \to R/P$ with $P \in \Min(R)$.   
\end{theorem}

Here we insert a small lemma into this chapter which will be used in the remark and example below.
\begin{lemma} \label{length_func_formula}
Let $(R,\fm,\kk)=\kk[[x_1,\ldots,x_d]]$ be a power series ring over a field. Let $I$ be an $\fm$-primary monomial ideal minimally generated by monomials $u_1,\ldots,u_\mu$. Fix integers $n_1,\ldots,n_d$. For each monomial $u_i=x_1^{a_1^{i}}\ldots x_d^{a_d^{i}}$, denote the monomial $v_i=x_1^{n_1a_1^i}\ldots x_d^{n_da_d^i}$, which can be seen as the monomial $u_i$ evaluated at $(x_1^{n_1},\ldots,x_d^{n_d})$. Let $J=(v_1,\ldots,v_\mu)$ be the ideal generated by all $v_i$'s. Then $J$ is minimally generated by $v_i$'s, $\fm$-primary, and $\ell(R/J)=n_1\ldots n_d \cdot \ell(R/I)$.   
\end{lemma}
\begin{proof}
Since $I$ is minimally generated by $u_i$'s, for each pair $(u_i,u_j)$, $u_i$ does not divide $u_j$. By definition $v_i$ does not divide $v_j$, so $J$ is minimally generated by $v_i$'s. Let $S=R=\kk[[x_1,\ldots,x_d]]$, $\phi:R \to S$ be the ring homomorphism which maps $x_i$ to $x_i^{n_i}$. Then $\phi$ makes $S$ a free $R$-module of rank $n_1\ldots n_i$. Note that if we consider $J$ as an $S$-ideal, then $J=\phi(I)S$. So $\ell(S/J)=\rank_R S\cdot \ell(R/I)=n_1\ldots n_d \cdot \ell(R/I)<\infty$.   
\end{proof}

\begin{remark}
We may define a family of ideals $I_\bu$ indexed by $\NN$ to be an inverse graded family of ideals if for any $n,m \in \NN$, $I_{n+m} \subset I_nI_m$. This definition appears in \cite{mustactǎ2002multiplicities} named by reverse-graded sequence there. However, for such families, the corresponding limit may not exist. The following example comes from a paper by Hernandez, Teixeira and Witt \cite{hernandez2020frobenius}.
\end{remark}
\begin{example} \label{inverse_graded_limit_may_not_exist}
Let $R=\kk[[x_1,\ldots,x_d]]$ be a power series ring of dimension $d \geqslant 2$ over a field of any characteristic. Let $\fm=(x_1,\ldots,x_d)$ be the maximal ideal of $R$. Fix a prime integer $p>0$, and denote $\fm^{[p^e]}=(x_1^{p^e},\ldots,x_d^{p^e})$. We can define the Frobenius power of $\fm$ of arbitrary integer exponent: for $n \in \NN$ with base $p$ expansion $n=n_0+n_1p+\ldots+n_ep^e$, define $\fm^{[n]}=\fm^{n_0}(\fm^{[p]})^{n_1}\ldots(\fm^{[p^e]})^{n_e}$. Then use carrying in base $p$ and the fact that $\fm^{[pq]} \subset (\fm^{[q]})^p$, we can prove $\fm^{[n+m]}\subset \fm^{[n]}\fm^{[m]}$. Let $I_\bu=\{\fm^{[n]}\}_{n \in \NN}$. If $n=p^e$, then $I_n=\fm^{[p^e]}=(x_1^{p^e},\ldots,x_d^{p^e})$ and $\ell(R/I_n)=n^d$. Thus
$$\lim_{e \to \infty}\frac{\ell(R/I_{p^e})}{(p^e)^d}=1.$$
Now we calculate $\ell(R/I_n)$ for $n=p^e-1$. The base $p$ expansion of $n$ is $(p-1)p^{e-1}+(p-1)p^{e-2}+\ldots+(p-1)$. For $0 \leqslant k \leqslant e-1$, Set
$$L_k=\fm^{[(p-1)p^k]}=(x^{p^k}_1,x^{p^k}_2,\ldots,x^{p^k}_d)^{p-1}.$$
Then
$$I_n=L_{e-1}L_{e-2}\ldots L_0.$$
For $0 \leqslant k \leqslant e-1$, denote
$$J_k=L_{e-1}\ldots L_k.$$
Let $\mu$ be the number of minimal generators of an ideal. We make the following observation:
\begin{enumerate}
\item $L_k,J_k$ are all monomial ideals for any $k$.
\item $J_k=J_{k+1}L_k$.
\item The exponent of the minimal monomial generators of $L_k,J_k$ are multiples of $p^k$.
\item $L_k$ is an $\fm$-primary ideal whose minimal monomial generators have exponents at most $(p-1)p^k$, which is at most $p^{k+1}-1$.
\item $J_{k+1}/L_kJ_{k+1}$ are free $R/L_k$-modules of rank $\mu(J_{k+1})$.
\item $J_{k+1}/\fm L_kJ_{k+1}$ are free $R/\fm L_k$-modules of rank $\mu(J_{k+1})$.
\item $\mu(J_k)=\mu(L_kJ_{k+1})=\mu(L_k)\mu(J_{k+1})$.
\item $\mu(J_{e-1})=\mu(L_k)=\mu(\fm^{p-1})={\binom{p+d-2}{d-1}}.$
\item $\mu(J_k)={\binom{p+d-2}{d-1}}^{e-k}$.
\item $\ell(R/L_k)=p^{kd}\ell(R/\fm^{p-1})=p^{kd}{\binom{p+d-2}{d}}$. In particular, $$\ell(R/J_{e-1})=\ell(R/L_{e-1})=p^{(e-1)d}{p+d-2\choose d}.$$.
\item $\ell(R/J_k)-\ell(R/J_{k+1})=\mu(J_{k+1})\ell(R/L_k)={p+d-2\choose d-1}^{e-k+1}p^{kd}{p+d-2\choose d}=\frac{p-1}{d}{p+d-2\choose d-1}^{e-k}p^{kd}.$
\item We have
\begin{align*}
\ell(R/J_0)=\frac{p-1}{d}\sum_{0 \leqslant k \leqslant e-1}{p+d-2\choose d-1}^{e-k}p^{kd}\\
=\frac{p-1}{d}{p+d-2\choose d-1}\frac{p^{ed}-{p+d-2\choose d-1}^e}{p^d-{p+d-2\choose d-1}}\\
={p+d-2\choose d}\frac{p^{ed}-{p+d-2\choose d-1}^e}{p^d-{p+d-2\choose d-1}}.
\end{align*}
\item We have
$$\lim_{e \to \infty}\frac{\ell(R/I_{p^e-1})}{(p^e-1)^d}=\frac{{p+d-2\choose d}}{p^d-{p+d-2\choose d-1}}$$
which is not always equal to $1$.
\end{enumerate}
Here (1), (2), (4) is true by definition, (3) is true by (2) and induction. Now we prove (5) and (6). We set $J_{k+1}=(u_1,\ldots,u_\mu)$, where $u_i$'s are monomial minimal generators, and $\mu=\mu(J_{k+1})$. Suppose there is a relation of $u_i$ in $J_{k+1}/\fm L_kJ_{k+1}$, say
$$\sum_{1 \leqslant i \leqslant \mu} a_iu_i=0$$
Then we may lift this relation to an equation in $J_{k+1}$, say
$$\sum_{1 \leqslant i \leqslant \mu} a'_iu_i=\sum_{1 \leqslant i \leqslant \mu} a''_iu_i$$
where $a''_i \in \fm L_k$ for any $i$. Set $b_i=a'_i-a''_i$, then in $J_{k+1}$ we have
$$\sum_{1 \leqslant i \leqslant \mu} b_iu_i=0$$
So $b_i$ produces a syzygy of $u_i$. But $u_i$'s are monomials, so the syzygies are generated by binomial syzygies
$$s_iu_i-s_iu_j=0$$
where $s_i=\textup{gcd}(u_i,u_j)/u_i,s_j=\textup{gcd}(u_i,u_j)/u_j$. By (3), each exponent of $u_i,u_j$ are either the same or differ by $p^{k+1}$. Thus all $s_i$'s lie in $(x_1^{p^{k+1}},\ldots,x_d^{p^{k+1}})$, so $b_i \in (x_1^{p^{k+1}},\ldots,x_d^{p^{k+1}})$ for any $i$. Now $\fm L_k$ is an $\fm$-primary ideal minimally generated in degree at most $p^{k+1}$, so $x_i^{p^{k+1}} \in \fm L_k$. Therefore all $b_i$'s lie in $\fm L_k$. Since $a''_i \in \fm L_k$, $a'_i \in \fm L_k$, so the images of $a_i$ in $R/\fm L_k$ is $0$. So every $R/\fm L_k$-syzygy of $J_{k+1}/\fm L_kJ_{k+1}$ must be a zero syzygy, so $J_{k+1}/\fm L_kJ_{k+1}$ is free over $R/\fm L_k$, so (6) is true; (5) is true by (6) mod out $L_k$. By (5) and (6),
$$L_kJ_{k+1}/\fm L_kJ_{k+1}=(L_k/\fm L_k)^{\oplus \mu(J_{k+1})}$$
taking lengths on both sides, we see (7) is true. (8) is true by \autoref{length_func_formula} applied to $L_k$ and $\fm^{p-1}$. (9) is true by (7), (8) and induction. (10) is true by  \autoref{length_func_formula} applied to $L_k$ and $\fm^{p-1}$. (11), (12), and (13) are just computation coming from (9) and (10). In (13), we may take $p=d=2$, then the limit is $1/2 \neq 1$.

We also see $(x_1^n,\ldots,x_d^n) \subset I_n \subset \fm^n$ for any $n \in \NN$, which gives
$$\frac{1}{d!}\leqslant \displaystyle \liminf_{n \to \infty}\frac{\ell(R/I_n)}{n^d}\leqslant \displaystyle \limsup_{n \to \infty}\frac{\ell(R/I_n)}{n^d}\leqslant 1$$
The last inequality implies
$$ \displaystyle \limsup_{n \to \infty}\frac{\ell(R/I_n)}{n^d}=1$$
and when $p=d=2$, the first inequality implies
$$ \displaystyle \liminf_{n \to \infty}\frac{\ell(R/I_n)}{n^d}=1/2$$
because this limit inferior is achieved by a subsequence $I_{p^e-1}$. In other cases, it is not clear. For example, when $p=d=3$, then the corresponding limit for $I_{p^e-1}$ is $4/21$, which is strictly bigger than $1/d!=1/6$. So we ask the following question.
\end{example}
\begin{question}
Assume $I_\bu=\{\fm^{[n]}\}_{n \in \NN}$. In general, what is the limit inferior  
$$\displaystyle \liminf_{n \to \infty}\frac{\ell(R/I_n)}{n^d}$$
depending on $p$ and $d$?
\end{question}
Also, the inequalities on the limit inferior and limit superior lead to the following question.
\begin{question}
Let $R$ be a Noetherian local ring, $I_\bu$ be a BBL family of ideals indexed by $\NN$ satisfying $I_{n+m} \subset I_nI_m$. Do we have   
$$d! \displaystyle \liminf_{n \to \infty}\frac{\ell(R/I_n)}{n^d}\geqslant \displaystyle \limsup_{n \to \infty}\frac{\ell(R/I_n)}{n^d}?$$
\end{question}

\section{Approximation of sets of valuations and bounds for the valuation of BBL families} \label{section_5}
In this section we will show how to express the colength of an ideal $I$ using an OK valuation $\nu$.

Throughout this section, we assume $(R,\fm,\kk)$ is a Noetherian local ring which is an OK domain, $\FF$ is the fraction field of $R$, $\nu:\FF \to \ZZ^d$ is an OK valuation with $\ZZ^d$ ordered by $\leqslant_{\fa},\fa \in \RR^d$ and $S=\nu(R)$ is pointed at direction $\fa$. We will use $h$ to refer to a positive integer with $1 \leqslant h \leqslant [\kk_\nu:\kk]$. For an $R$-submodule $M$ of $\mathbb{F}$, denote
$$\nu^{(h)}(M)=\{\mathbf{u} \in \mathbb{Z}^d|\dim_{\kk}(\frac{M\cap\mathbb{F}_{\geqslant \mathbf{u}}}{M\cap \mathbb{F}_{>\mathbf{u}}})\geqslant h\}.$$
\begin{proposition} \cite[Remark 3.10]{hernandez2018local}
\begin{enumerate}
\item $\nu^{(1)}(M)=\nu(M)$
\item $\nu^{(h)}(M)$ is descending in $h$ as subsets of $\mathbb{Z}^d$. In particular, for ideal $I \subset R$, $\nu^{(h)}(I) \subset \nu(R)$.
\item $\nu^{(h)}(M)+S \subset \nu^{(h)}(M)$. In particular, for an ideal $I \subset R$, $\nu^{(h)}(I)$ are $S$-ideals.
\end{enumerate}    
\end{proposition}
For an $R$-ideal $I$, we denote $\nu^{h,c}(I)=\nu(R)-\nu^{(h)}(I)$. 
\begin{lemma} \cite[Lemma 3.12]{hernandez2018local} \label{approx_v^h_to_v}
There is an $\fv \in S$ such that $\nu(M)+\fv \subset \nu^{(h)}(M) \subset \nu(M)$ for any $M$ and $1 \leqslant h \leqslant [\kk_\nu:\kk]$. 
\end{lemma}
\begin{remark}\label{approx_v^h_to_v remark}
In the lemma above, we can assume that $\fv \in C^\circ$ where $C=\Cone(S)$. In fact, if $\fv \in C\backslash C^\circ$, then it lies on the boundary, $\fv \in \partial C$. Since $C$ is full-dimensional, we can always select a point $\fv_0 \in S \cap C^\circ.$ Consequently, we have $\nu(M)+\fv+\fv_0 \subset \nu^{(h)}(M)$ as $\nu^{(h)}(M)$ consists of $S$-ideals, and it follows that $\fv+\fv_0 \subset C+C^\circ \subset C^\circ$.
\end{remark}
\begin{lemma} \label{trunc_length_formula}
Let $\fu \in \RR^d$. Assume $I \subset J$ be two $R$-ideals such that $I\cap \mathbb{F}_{\geqslant \fu}=J\cap \mathbb{F}_{\geqslant \fu}$. Let $H=H_{<\fu}$ be the truncating half space, then
$$\ell_{R}(J/I)=\sum_{1 \leqslant h \leqslant [\kk_\nu:\kk]} \#(\nu^{(h)}(J)\cap H)-\#(\nu^{(h)}(I)\cap H).$$
\end{lemma}
\begin{proof}
If $\fu \in \ZZ^d$, then the exact sequences 
\[ 0  \to \frac{ I }{ I \cap \FF_{\geq \fu } } \to \frac{ J }{ I \cap \FF_{\geq \fu } } \to \frac{ J }{ I  } \to 0 \] 
and 
\[ 0  \to \frac{ J \cap \FF_{\geq \fu } }{ I \cap \FF_{\geq \fu } } \to \frac{ J }{ I \cap \FF_{\geq \fu } } \to \frac{ J }{ J \cap \FF_{\geq \fu } } \to 0 \vspace{2mm} \] 
show that the length of $J / I$ equals 
\begin{equation*} 
\label{length relations: e}
\ell_R \left( \frac{ J \cap \FF_{\geq \fu } }{ I \cap \FF_{\geq \fu } }\right)  + \ell_R \left( \frac{ J }{ J \cap \FF_{\geq \fu } } \right)  - \ell_R \left( \frac{ I }{ I \cap \FF_{\geq \fu } } \right)  .
\end{equation*}
Note that the reason why length of $J / I$  is finite follows from \cite[Remark 5.13]{hernandez2018local}. Now using the fact that $I\cap \mathbb{F}_{\geqslant \fu}=J\cap \mathbb{F}_{\geqslant \fu}$ we get,
\begin{equation*}
\ell_{R}(J/I)= \ell_R \left( \frac{ J }{ J \cap \FF_{\geq \fu } } \right)  - \ell_R \left( \frac{ I }{ I \cap \FF_{\geq \fu } } \right).
\end{equation*}
Moreover by \cite[Lemma 3.11]{hernandez2018local} we get $$\ell_{R}(J/I)=\sum_{1 \leqslant h \leqslant [\kk_\nu:\kk]} \#(\nu^{(h)}(J)\cap H)-\#(\nu^{(h)}(I)\cap H).$$
\end{proof}
\begin{corollary}
    
 \label{length_formula_compl}
Assume $I$ be an $R$-ideal such that $R\cap \mathbb{F}_{\geqslant \fu}\subset I$. Let $H=H_{<\fu}$ be the truncating half space, then
$$\ell(R/I)= \sum_{1 \leqslant h \leqslant [\kk_\nu:\kk]}\#(\nu^{h,c}(I)\cap H)
-\sum_{1 \leqslant h \leqslant [\kk_\nu:\kk]}\#(\nu^{h,c}(R)\cap H)$$
\end{corollary}
\begin{proof}
Using  \autoref{trunc_length_formula} we have:
\begin{align*}
\ell(R/I)=\sum_{1 \leqslant h \leqslant [\kk_\nu:\kk]} \#(\nu^{(h)}(R)\cap H)-\#(\nu^{(h)}(I)\cap H)\\
=\sum_{1 \leqslant h \leqslant [\kk_\nu:\kk]}(\#(\nu(R)\cap H)-\#(\nu^{(h)}(I)\cap H))
-\sum_{1 \leqslant h \leqslant [\kk_\nu:\kk]}(\#(\nu(R)\cap H)-\#(\nu^{(h)}(R)\cap H))\\
=\sum_{1 \leqslant h \leqslant [\kk_\nu:\kk]}\#(\nu(R)\cap H\ba\nu^{(h)}(I)\cap H)
-\sum_{1 \leqslant h \leqslant [\kk_\nu:\kk]}\#(\nu(R)\cap H\ba \nu^{(h)}(R)\cap H)\\
=\sum_{1 \leqslant h \leqslant [\kk_\nu:\kk]}\#(\nu^{h,c}(I)\cap H)
-\sum_{1 \leqslant h \leqslant [\kk_\nu:\kk]}\#(\nu^{h,c}(R)\cap H).
\end{align*}
\end{proof}

\begin{remark}
Note that even if $I$ is $\mathfrak{m}$-primary, it is not clear whether $\nu^{h,c}(I)$ is bounded or not; however, we only need to consider its intersection with a halfspace $H$ which is bounded.
\end{remark}

Now we turn to a BBL family.
\begin{lemma} \label{val_bbl_prop}
Let $I_\bu$ be a BBL family indexed by $\NN$. Then there exists $\fu \in \ZZ^d$ such that $R \cap \FF_{\geqslant n\fu} \subset I_n$. For such $\fu$, denote $H=H_{<\fu}$, then $\nu^c(I_n) \subset \Cone(S) \cap nH$ for any $n \in \NN$.  
\end{lemma}
\begin{proof}
By  \autoref{bbl_defn}, there exists $c>0$ such that $\fm^{cn} \subset I_n$. By \autoref{OK_val_def}, there exists $\fv \in \ZZ^d$ such that $R \cap \FF_{\geqslant n\fv} \subset \fm^n$. Thus taking $\fu=c\fv$ we have $R \cap \FF_{\geqslant n\fu} \subset I_n$. Moreover,  $\nu^c(I_n)=\nu(R)\backslash\nu(I_n) \subset \nu(R)\backslash\nu(R)\cap H_{\geqslant n\fu}=\nu(R)\cap H_{<n\fu}\subset \Cone(S)\cap H_{<n\fu}=\Cone(S)\cap nH$.  
\end{proof}
\begin{corollary}
Let $I_\bu$ be a BBL family indexed by $\NN$. Then for a truncating half space $H=H_{<\fu}$ satisfying $R \cap \FF_{\geqslant n\fu} \subset I_n$, we have:
\begin{enumerate}
\item $\ell(R/I_n)=\sum_{1 \leqslant h \leqslant [\kk_\nu:\kk]}\#(\nu^{h,c}(I)\cap nH))
-\sum_{1 \leqslant h \leqslant [\kk_\nu:\kk]}\#(\nu^{h,c}(R)\cap nH).$
\item $\ell(R/I_n)=\sum_{1 \leqslant h \leqslant [\kk_\nu:\kk]}\#(1/n\nu^{h,c}(I)\cap H))
-\sum_{1 \leqslant h \leqslant [\kk_\nu:\kk]}\#(1/n\nu^{h,c}(R)\cap H).$
\end{enumerate}
\end{corollary}
\begin{proof}
(1) Apply  \autoref{length_formula_compl} and  \autoref{val_bbl_prop}, and note that $nH_{<\fu}=H_{<n\fu}$.

(2) is trivial  because rescaling does not change the cardinality.
\end{proof}
\begin{corollary}\label{value_group_uniformly_bounded}
Let $I_\bu$ be a BBL family indexed by $\NN$. Then there exist a truncating half space $H=H_{<\fu}$ and constant $C>0$ such that for any $n \in \NN$, 
$$1/n\nu^c(I_n) \subset \Cone(S)\cap H \subset B(0,C).$$   
\end{corollary}
\begin{proof}
The first containment is true by  \autoref{val_bbl_prop}, and the second containment is true by  \autoref{trunc_bound}.   
\end{proof}
If $R$ has characteristic $p>0$, the following results follow similarly as before:
\begin{lemma}
Let $I_\bu$ be a BBL family indexed by powers of $p$. Then there exists $\fu \in \ZZ^d$ such that $R \cap \FF_{\geqslant q\fu} \subset I_q$. For such $\fu$, denote $H=H_{<\fu}$, then we have $\nu^c(I_q) \subset \Cone(S) \cap qH$ for any $e \in \NN$.     
\end{lemma}
\begin{corollary}
Let $I_\bu$ be a BBL family indexed by powers of $p$. Then for a truncating half space $H=H_{<\fu}$ satisfying $R \cap \FF_{\geqslant q\fu} \subset I_q$, we have:
\begin{enumerate}
\item $\ell(R/I_q)=\sum_{1 \leqslant h \leqslant [\kk_\nu:\kk]}\#(\nu^{h,c}(I)\cap qH))
-\sum_{1 \leqslant h \leqslant [\kk_\nu:\kk]}\#(\nu^{h,c}(R)\cap qH).$
\item $\ell(R/I_q)=\sum_{1 \leqslant h \leqslant [\kk_\nu:\kk]}\#(1/q\nu^{h,c}(I)\cap H))
-\sum_{1 \leqslant h \leqslant [\kk_\nu:\kk]}\#(1/q\nu^{h,c}(R)\cap H).$
\end{enumerate}
\end{corollary}
\begin{corollary}
Let $I_\bu$ be a BBL family indexed by powers of $p$. Then there exist a truncating half space $H=H_{<\fu}$ and constant $C>0$ such that for any $q$, 
$$1/q\nu^c(I_q) \subset \Cone(S)\cap H \subset B(0,C).$$   
\end{corollary}

\section{Proof of the convergence for weakly graded families} \label{section_6}
Throughout this section, we follow the setup below.
\begin{setup} \label{general_setup}
we assume $(R,\fm,\kk)$ is an OK domain of any characteristic of dimension $d$. The OK valuation is denoted by $\nu:\FF^\times \to \ZZ^d$, where $\FF$ is the field of fractions of $R$, and the order on $\ZZ^d$ is induced by $\fa \in \RR^d$. Let $S=\nu(R)$, $C=\Cone(S)$ which is pointed at direction $\fa$ and upward to direction $\fb$.    
\end{setup}

In this context, we will focus on families of 
$R$-ideals indexed by 
$\NN$. For a BBL family $I_\bu$ indexed by $\NN$, we fix a truncating half space $H=H_{<\fu}$ such that $R \cap \FF_{\geqslant n\fu} \subset I_n$.
Additionally, we will examine the convergence of the functions $F_{1/n\nu^{h,c}(I_n)\cap H,1/n}(\fx)$ mentioned in \autoref{des_of_L}.

However, the convergence may depend on $h$ and the choice of $H$.  Therefore, we will first examine the convergence of the function
$F_{1/n\nu^c(I_n),1/n}(\fx)$ 
and compare their converging property at all points in $\RR^d$.

The aim of this section is to prove existence of limit for BBL weakly graded families in an OK domain.
\begin{proposition} \label{type_0_F_conv}
Let $I_\bu$ be a BBL family. Suppose $H$ is a truncating half space as in \autoref{value_group_uniformly_bounded} such that $1/n\nu^c(I_n) \subset C \cap H$ for all $n$. 
\begin{enumerate}
\item If $\fx \notin C \cap \overline{H}$, then for sufficiently large $n$, $F_{1/n\nu^c(I_n),1/n}(\fx)=0$.
\item If $\fx \notin \overline{B(C \cap \overline{H},d)}$, then for any $n \geqslant 1$, $F_{1/n\nu^c(I_n),1/n}(\fx)=0$. 
\item $F_{1/n\nu^c(I_n),1/n}(\fx)$ is uniformly bounded and supported on a bounded set.
\end{enumerate}
\end{proposition}
\begin{proof}
By \autoref{trunc_bound}, $C \cap \overline{H}$ is a closed bounded set. Assume for some $n$, $F_{1/n\nu^c(I_n),1/n}(\fx) \neq 0$. This is equivalent to $[\fx]_n \in 1/n\nu^c(I_n)$, therefore $[\fx]_n \in C\cap \overline{H}$. But $\|\fx-[\fx]_n\| \leqslant d/n$, so $\dist(\fx,C\cap \overline{H})\leqslant d/n$. Therefore, if $\fx \notin C\cap \overline{H}$, we have $\dist(\fx,C\cap \overline{H})>0$ since $C\cap \overline{H}$ is a closed set, and for $n>d/\dist(\fx,C\cap \overline{H})$, we have $F_{1/n\nu^c(I_n),1/n}(\fx)=0$. If $\fx \notin \overline{B(C \cap \overline{H},d)}$, then $d/\dist(\fx,C\cap \overline{H})<1$, so $n>d/\dist(\fx,C\cap \overline{H})$ is true for all positive integer $n$. The set $\overline{B(C \cap \overline{H},d)}$ is a bounded set which contains the supports of all functions $F_{1/n\nu^c(I_n),1/n}(\fx)$, and since these functions are characteristic functions, their values are bounded by $1$.
\end{proof}
\begin{proposition} \label{type_0_F_conv_with_H}
Let $I_\bu$ be a BBL family. Let $H=H_{<\fu}$ be any truncating half space.
\begin{enumerate}
\item If $\fx \notin C \cap \overline{H}$, then for sufficiently large $n$, $F_{1/n\nu^{h,c}(I_n)\cap H,1/n}(\fx)=0$.
\item If $\fx \notin \overline{B(C\cap \overline{H},d)}$, then for any $n \geqslant 1$, $F_{1/n\nu^{h,c}(I_n)\cap H,1/n}(\fx)=0$. 
\item $F_{1/n\nu^{h,c}(I_n)\cap H,1/n}(\fx)$ is uniformly bounded and supported on a bounded set.
\end{enumerate}
\end{proposition}
\begin{proof}
Following the same argument as in \autoref{type_0_F_conv}.
\end{proof}
\subsection{Convergence of $F_{1/n\nu^c(I_n),1/n}(\fx)$ for BBL families} Let $I_\bu$ be a BBL family. Choose a truncating half space $H$ such that $1/n\nu^c(I_n) \subset C \cap H$ for all $n$. The convergence on points outside $C \cap \overline{H}$ is straightforward, and since $\partial(C \cap \overline{H}) \subset \partial C \cup \partial H$ has zero measure, we can focus on the convergence of $F_{1/n\nu^c(I_n),1/n}(\fx)$ within $C^\circ$.

For each $\fx \in C^\circ$, the functions
$F_{1/n\nu^c(I_n),1/n}(\fx)$ have values equal to either $0$ or $1$. This gives rise to three distinct cases:
\begin{enumerate}
\item $F_{1/n\nu^c(I_n),1/n}(\fx)=1$ for all large $n$.
\item $F_{1/n\nu^c(I_n),1/n}(\fx)=0$ for all large $n$.
\item $F_{1/n\nu^c(I_n),1/n}(\fx)=1$ for $n$ lying in a sequence $n_k, k \in \mathbb{N}$ and $F_{1/n\nu^c(I_n),1/n}(\fx)=0$ for $n$ lying in another sequence $n_m, m \in \mathbb{N}$. That is, the value of $F_{1/n\nu^c(I_n),1/n}(\fx)$ oscillate between $0$ and $1$.
\end{enumerate}
Note that the limit $\lim_{n \to \infty}F_{1/n\nu^c(I_n),1/n}(\fx)$ exists if and only if $\fx$ lies in case (1) or (2).

We define some sets associated to the family $I_\bullet$:
$$\nabla^{up}=\{\fx \in C^\circ: F_{1/n\nu^c(I_n),1/n}(\fx)=1 \textup{ for infinitely many }n\}$$
$$\nabla^{low}=\{\fx \in C^\circ: F_{1/n\nu^c(I_n),1/n}(\fx)=1 \textup{ for almost all }n\}$$
$$\Delta^{up}=C^\circ\backslash\nabla^{up}$$
$$\Delta^{low}=C^\circ\backslash\nabla^{low}.$$
Also by the argument above, $\nabla^{low} \subset \nabla^{up} \subset C^\circ \cap\overline{H}$, and $C^\circ\ba\overline{H} \subset \Delta^{up} \subset \Delta^{low}$. Actually, for $\fx \in C^{\circ}\ba\overline{B(C\cap \overline{H},d)}$, $F_{1/n\nu^c(I_n),1/n}(\fx)=0$ for all $n$.
\begin{proposition}\label{second_definition_of_delta_and_nabla}
For a given BBL family $I_\bu$, we have:
\begin{enumerate}
\item $\nabla^{up}=\{\fx \in C^\circ: [\fx]_n \in 1/n\nu^c(I_n) \textup{ for infinitely many }n\}$;
\item $\nabla^{low}=\{\fx \in C^\circ: [\fx]_n \in 1/n\nu^c(I_n) \textup{ for }n \gg 0\}$;
\item $\Delta^{up}=\{\fx \in C^\circ: [\fx]_n \notin 1/n\nu^c(I_n) \textup{ for }n \gg 0\}$;
\item $\Delta^{low}=\{\fx \in C^\circ: [\fx]_n \notin 1/n\nu^c(I_n) \textup{ for infinitely many }n\}$;
\item $\Delta^{up}=\{\fx \in C^\circ: [\fx]_n \in 1/n\nu(I_n) \textup{ for }n \gg 0\}$;
\item $\Delta^{low}=\{\fx \in C^\circ: [\fx]_n \in 1/n\nu(I_n) \textup{ for infinitely many }n\}$.
\end{enumerate}
\end{proposition}
\begin{proof}
(1)-(4) is true by definition and \autoref{prop_of_[]_and_nabla}, so we only need to prove (5) and (6); if $\fx \in C^\circ$, then $\fx \neq 0$, and for large enough $n$, $[\fx]_n \in 1/nS$ by \autoref{K_K_appl2}, so $[\fx]_n \notin 1/n\nu^c(I_n)$ if and only if $[\fx]_n \in 1/n\nu(I_n)$, so (5) and (6) are true.
\end{proof}
\begin{lemma} \label{ideal_prop_of_delta_up_and_low}
Using the notations above, then for any family $I_\bullet$ which is not necessarily weakly graded or BBL, the corresponding $\Delta^{up}$ and $\Delta^{low}$ are closed under adding a nonzero element in $C^\circ$. In other words, they are $C^\circ$-ideals.    
\end{lemma}
\begin{proof}
We first prove for $\Delta^{up}$. We take $\fx \in \Delta^{up}$ and $0 \neq \fx_0 \in C^\circ$. Using  \autoref{union_interior_closed_upward_cone} we have $C^\bu$ is the union of the interior of all upward closed cones $C'$ in $C^\bu$, we may take one such $C'$ such that $\fx,\fx_0 \in C'^\circ$.

By \autoref{second_definition_of_delta_and_nabla}, $[\fx]_n \in 1/n\nu(I_n)$ for $n \gg 0$. When $n \to \infty$, $[\fx+\fx_0]_n-[\fx]_n \to \fx_0$. Apply \autoref{K_K_appl2}, we see for large enough $n$, $[\fx+\fx_0]_n-[\fx]_n \in 1/nS \cap C'^\circ$. Since $1/n\nu(I_n)$ is an $1/nS$-ideal of $1/nS$, $[\fx+\fx_0]_n \in 1/n\nu(I_n)$. This is true for all large enough $n$, so $\fx+\fx_0 \in \Delta^{up}$.

The proof for $\Delta^{low}$ is similar; we just replace the condition $[\fx]_n \in 1/n\nu(I_n)$ for $n \gg 0$ by $[\fx]_n \in 1/n\nu(I_n)$ for infinitely many $n$.
\end{proof}

\subsection{Convergence of $F_{1/n\nu^{h,c}(I_n),1/n}(\fx)$ for BBL families.}

\phantom{}\\
Now we come back to the function
$$F_{1/n\nu^{h,c}(I_n)\cap H,1/n}(\fx).$$
We assume the family $I_\bu$ is BBL, and fix a truncating half space $H=H_{<\fu}$ such that $R \cap \FF_{\geqslant n\fu} \subset I_n$. Therefore by  \autoref{val_bbl_prop}, $H$ is a truncating half space such that $1/n\nu^c(I_n) \subset C \cap H$ for all $n$. 
\begin{proposition}\label{nabla_up_low_containment_relations}
Under the notations above,
\begin{enumerate}
\item $\nabla^{low} \subset \nabla^{up} \subset C^\circ\cap \overline{H}$;
\item $C^\circ\ba \overline{H} \subset \Delta^{up} \subset \Delta^{low}$;
\item $\nabla^{low\circ} \subset \nabla^{up\circ} \subset C^\circ \cap H^\circ$;
\item $C^\circ\cap H^\circ \subset \nabla^{low\circ} \cup (\partial\nabla^{low}\cap C^\circ) \cup (\Delta^{low\circ}\cap\nabla^{up\circ}) \cup (\partial \nabla^{up}\cap C^\circ) \cup (\Delta^{up\circ} \cap H^\circ)$;
\item $\partial\nabla^{low}\cap C^\circ =\partial \Delta^{low}\cap C^\circ$, $\partial\nabla^{up}\cap C^\circ= \partial \Delta^{up}\cap C^\circ$.
\end{enumerate}   
\end{proposition}
\begin{proof}
By  \autoref{type_0_F_conv}, we see if $x \notin C\cap \overline{H}$, $\lim_{n \to \infty}F_{1/n\nu^c(I_n),1/n}(\fx)=0$. This leads to (1). Taking complement in $C^\circ$, we get (2). Note that $(C^\circ \cap \overline{H})^\circ=C^\circ \cap H^\circ$, so taking interior for (1) we get (3). We filtrate $C^\circ \cap H^\circ$ by the following sequence of subsets
$$\emptyset \subset \nabla^{low\circ} \subset \overline{\nabla^{low}}\cap \nabla^{up\circ} \subset \nabla^{up\circ} \subset \overline{\nabla^{up}}\cap C^\circ \cap H^\circ \subset C^\circ \cap H^\circ. $$
Then the five difference of these six sets lies in
$\nabla^{low\circ}, (\partial\nabla^{low}\cap C^\circ), (\Delta^{low\circ}\cap\nabla^{up\circ}), (\partial \nabla^{up}\cap C^\circ),(\Delta^{up\circ} \cap H^\circ)$
respectively, so (4) is true. (5) is true since $\Delta^{up}=C^\circ\ba \nabla^{up}$ and $\Delta^{low}=C^\circ\ba \nabla^{low}$.
\end{proof}
\begin{lemma}
Let $I_\bullet$ be a BBL family. Let $H$ be a truncating half space. Then for any $\fx \in \nabla^{low\circ}$, we have
$$\lim_{n \to \infty}F_{1/n\nu^{h,c}(I_n)\cap H,1/n}(\fx)=1.$$
\end{lemma}
\begin{proof}
We see $\nu^{(h)}(I) \subset \nu(I)$, so $\nu^c(I) \subset \nu^{h,c}(I)$. So if $F_{1/n\nu^{c}(I_n),1/n}(\fx)=1$ and $\fx \in H^\circ$, then for sufficiently large $n$, $[\fx]_n \in H^\circ$, and for such $n$, $F_{1/n\nu^{h,c}(I_n)\cap H,1/n}(\fx)=1.$ So if $\fx \in \nabla^{low\circ}\subset H^\circ$, then $\lim_{n \to \infty}F_{1/n\nu^{c}(I_n)\cap H,1/n}(\fx)=1$.  
\end{proof}
\begin{lemma}
Let $I_\bullet$ be a BBL family. Then for any $\fx \in \Delta^{up\circ}\cap H^\circ$, we have
$$\lim_{n \to \infty}F_{1/n\nu^{h,c}(I_n)\cap H,1/n}(\fx)=0.$$
\end{lemma}
\begin{proof}
We fix $\fv \in S \cap C^\circ$ as in  \autoref{approx_v^h_to_v} and \autoref{approx_v^h_to_v remark}. We may choose a closed subcone $C' \subset C^\bu$ such that 
$\fv, \fx \in C'^\circ$. Since $\fx \in \Delta^{up\circ}$, there exists $N_1$ such that for $n \geqslant N_1$, $\fx-1/n\fv \in \Delta^{up}\cap C'$. In particular $\fx-1/N_1\fv \in \Delta^{up}$, so for the fixed vector $\fx-1/N_1\fv$ 
and large enough $n$,
$[\fx-1/N_1\fv]_n \in 1/n\nu(I_n)$. This means $\lfloor n\fx-n/N_1\fv \rfloor \in \nu(I_n)$. By  \autoref{approx_v^h_to_v}, $\lfloor n\fx-n/N_1\fv+\fv \rfloor \in \nu^{(h)}(I_n)$, so  $[\fx-1/N_1\fv+1/n\fv]_n \in 1/n\nu^{(h)}(I_n)$. Now we consider
$$\fy_n=[\fx]_n-[\fx-1/N_1\fv+1/n\fv]_n.$$
We see $\lim_{n \to \infty}\fy_n=1/N_1\fv \in C'^\circ$, and $1/N_1\fv \neq 0$. Thus by \autoref{K_K_appl2}, for large eonugh $n$, $\fy_n=[\fy_n]_n \in 1/nS$. Since $[\fx-1/N_1\fv+1/n\fv]_n \in 1/n\nu^{(h)}(I_n)$ and $1/n\nu^{(h)}(I_n)$ is an $1/nS$-ideal, we see $[\fx]_n \in 1/n\nu^{(h)}(I_n)$. Also since $\fx \in H^\circ$, for large enough $n$, $[\fx]_n \in H^\circ$. Thus
$$F_{1/n\nu^{h,c}(I_n)\cap H,1/n}(\fx)=F_{1/n\nu^{h,c}(I_n)\cap H,1/n}([\fx]_n)=0$$
is true for large $n$, and
$$\lim_{n \to \infty}F_{1/n\nu^{h,c}(I_n)\cap H,1/n}(\fx)=0.$$
\end{proof}
\begin{corollary}
For any $\fx \in C^\circ\cap H^\circ$ we have  
$$\lim_{n \to \infty}F_{1/n\nu^{h,c}(R)\cap H,1/n}(\fx)=0.$$
\end{corollary}
\begin{proof}
We just take the family $I_n=R$ for any $n$. Note that in this case $\nu^c(I_n)=\emptyset$ for any $n$. So $\nabla^{up}=\nabla^{low}=\emptyset$, and $\Delta^{up\circ}=C^\circ$.    
\end{proof}
In summary, for any $1 \leqslant k \leqslant [\kk_\nu:\kk]$, in terms of the convergence of the functions $F_{1/n\nu^{h,c}(I_n)\cap H,1/n}$, we get:
\begin{enumerate}
\item For $\fx \notin C \cap \overline{H}$,
$$\lim_{n \to \infty}F_{1/n\nu^{h,c}(I_n)\cap H,1/n}(\fx)=0.$$
\item For $\fx \in \nabla^{low\circ}\subset C^\circ \cap H^\circ$,
$$\lim_{n \to \infty}F_{1/n\nu^{h,c}(I_n)\cap H,1/n}(\fx)=1.$$
\item For $\fx \in \Delta^{up\circ}\cap H^\circ$,
$$\lim_{n \to \infty}F_{1/n\nu^{h,c}(I_n)\cap H,1/n}(\fx)=0.$$
\item For $\fx \in \Delta^{low\circ}\cap \nabla^{up\circ}$, the limit
$$\lim_{n \to \infty}F_{1/n\nu^{h,c}(I_n)\cap H,1/n}(\fx)$$
is not clear.
\item For $\fx$ not lying in all the sets above, then $\fx \in E=(\partial C \cup \partial H \cup \partial \Delta^{up} \cup\partial \Delta^{low})\cap(C \cap \overline{H})$. The set $E$ is bounded and has zero measure. Actually, the set $C \cap \overline{H}$ is bounded. The sets $\partial C, \partial \Delta^{low}, \partial \Delta^{up}$ have zero measure because $C,\Delta^{low},\Delta^{up}$ are $C^\circ$-ideals. $\partial H$ is a hyperplane, so has zero measure.
\end{enumerate}
Therefore we have:
\begin{corollary} \label{char_nabla_up_and_low_bound}
For any $\fx \in \mathbb{R}^d$ and $1 \leqslant h \leqslant [\kk_\nu:\kk]$,
$$\chi(\nabla^{low\circ})(\fx) \leqslant \displaystyle \liminf_{n \to \infty}F_{1/n\nu^{h,c}(I_n)\cap H,1/n}(\fx) \leqslant \displaystyle \limsup_{n \to \infty}F_{1/n\nu^{h,c}(I_n)\cap H,1/n}(\fx) \leqslant \chi(\nabla^{up\circ})(\fx)+\chi(E)(\fx).$$
\end{corollary}

\begin{theorem} \label{volume_formula_for_BBL_family}
Let $I_\bullet$ be a BBL family. Then:
$$[\kk_\nu:\kk]\vol(\nabla^{low})\leqslant \displaystyle \liminf_{n \to \infty}\frac{\ell(R/I_n)}{n^d}\leqslant \displaystyle \limsup_{n \to \infty}\frac{\ell(R/I_n)}{n^d} \leqslant [\kk_\nu:\kk]\vol(\nabla^{up}).$$
In particular, when $\vol(\nabla^{up}\ba\nabla^{low})=0$, then the limit
$$\lim_{n \to \infty}\frac{\ell(R/I_n)}{n^d}=[\kk_\nu:\kk]\vol(\nabla^{low})=[\kk_\nu:\kk]\vol(\nabla^{up})$$
exists and is finite.
\end{theorem}
\begin{proof}
Since $I_\bullet$ is BBL, we take $\fu \in \mathbb{R}^d$ such that $R\cap \mathbb{F}_{\geqslant n\fu} \subset I_n$. Let $H=H_{<\fu}$. By  \autoref{length_formula_compl},
\begin{align*}
\ell(R/I_n)=\sum_{1 \leqslant h \leqslant [\kk_\nu:\kk]}\#(1/n\nu^{h,c}(I_n)\cap H)
-\sum_{1 \leqslant h \leqslant [\kk_\nu:\kk]} \#(1/n\nu^{h,c}(R)\cap H)\\
=n^d(\sum_{1 \leqslant h \leqslant [\kk_\nu:\kk]}\int_{\mathbb{R}^d}F_{1/n\nu^{h,c}(I_n)\cap H,1/n}(\fx)d\fx-\sum_{1 \leqslant h \leqslant [\kk_\nu:\kk]}\int_{\mathbb{R}^d}F_{1/n\nu^{h,c}(R)\cap H,1/n}(\fx)d\fx).
\end{align*}
Therefore
$$\frac{\ell(R/I_n)}{n^d}=\sum_{1 \leqslant h \leqslant [\kk_\nu:\kk]}\int_{\mathbb{R}^d}F_{1/n\nu^{h,c}(I_n)\cap H,1/n}(\fx)d\fx-\sum_{1 \leqslant h \leqslant [\kk_\nu:\kk]}\int_{\mathbb{R}^d}F_{1/n\nu^{h,c}(R)\cap H,1/n}(\fx)d\fx.$$
We have for any $\fx \in \mathbb{R}^d$,
$$\displaystyle \liminf_{n \to \infty}F_{1/n\nu^{h,c}(I_n)\cap H,1/n}(\fx) \geqslant \chi(\nabla^{low\circ})(\fx).$$
Using  \autoref{Fatou's lemma} we have,
\begin{align*}
\displaystyle \liminf_{n \to \infty}\int_{\mathbb{R}^d}F_{1/n\nu^{h,c}(I_n)\cap H,1/n}(\fx)d\fx\\
\geqslant \int_{\mathbb{R}^d}\displaystyle \liminf_{n \to \infty}F_{1/n\nu^{h,c}(I_n)\cap H,1/n}(\fx)d\fx\\
\geqslant \int_{\mathbb{R}^d}\chi(\nabla^{low\circ})(\fx)d\fx=\vol(\nabla^{low\circ})=\vol(\nabla^{low}).
\end{align*}

We have for any $\fx \in \mathbb{R}^d$,
$$ \displaystyle \limsup_{n \to \infty}F_{1/n\nu^{h,c}(I_n)\cap H,1/n}(\fx) \leqslant \chi(\nabla^{up\circ})(\fx)+\chi(E)(\fx).$$
Note that the sequence of functions $F_{1/n\nu^{h,c}(I_n)\cap H,1/n}(\fx)$ is uniformly bounded and supported on a bounded set $\overline{B(C \cap \overline{H},d)}$ by  \autoref{type_0_F_conv_with_H}. Therefore, we can apply  \autoref{reverse Fatou's lemma} to get
\begin{align*}
\displaystyle \limsup_{n \to \infty}\int_{\mathbb{R}^d}F_{1/n\nu^{h,c}(I_n)\cap H,1/n}(\fx)d\fx\\
\leqslant \int_{\mathbb{R}^d}\displaystyle \limsup_{n \to \infty}F_{1/n\nu^{h,c}(I_n)\cap H,1/n}(\fx)d\fx\\
\leqslant \int_{\mathbb{R}^d}\chi(\nabla^{up\circ})(\fx)+\chi(E)(\fx)d\fx=\vol(\nabla^{up\circ})=\vol(\nabla^{up}).
\end{align*}
In summary, we get inequalities
$$\vol(\nabla^{low}) \leqslant \displaystyle \liminf_{n \to \infty}\int_{\mathbb{R}^d}F_{1/n\nu^{h,c}(I_n)\cap H,1/n}(\fx)d\fx \leqslant \displaystyle \limsup_{n \to \infty}\int_{\mathbb{R}^d}F_{1/n\nu^{h,c}(I_n)\cap H,1/n}(\fx)d\fx \leqslant \vol(\nabla^{up}).$$
Now we set $I_n=R$, then $\nabla^{up}=\nabla^{low}=\emptyset$, and we immediately get
$$\lim_{n \to \infty}\int_{\mathbb{R}^d}F_{1/n\nu^{h,c}(R)\cap H,1/n}(\fx)d\fx=0.$$
Combining the equation expressing $\frac{\ell(R/I_n)}{n^d}$ and all the inequalities above, we get the result.
\end{proof}

\begin{corollary}
\label{result_for_BBL_with_h_equal_1}    

Let $I_\bullet$ be a BBL family. Then:
$$\vol(\nabla^{low})\leqslant \displaystyle \liminf_{n \to \infty}\frac{\#(\nu(R)\backslash\nu(I_n))}{n^d}\leqslant \displaystyle \limsup_{n \to \infty}\frac{\#(\nu(R)\backslash\nu(I_n))}{n^d} \leqslant \vol(\nabla^{up}).$$
In particular, when $\vol(\nabla^{up}\ba\nabla^{low})=0$, then the limit
$$\lim_{n \to \infty}\frac{\#(\nu(R)\backslash\nu(I_n))}{n^d}=\vol(\nabla^{low})$$
exists.
\end{corollary}
\begin{proof}
Apply  \autoref{char_nabla_up_and_low_bound} and use the same proof as  \autoref{volume_formula_for_BBL_family} when $h=1$.    
\end{proof}

\subsection{Estimates for BBL weakly graded families}
In this subsection, we assume $I_\bu$ is a weakly graded family.
\begin{lemma} \label{most_imp_theorem}
Assume $I_\bullet$ is a weakly graded family with $I_1 \neq 0$, $C=\Cone(S)$ is upward to direction $\fb$. Then for any $\fx \in \Delta^{low}$ and $\epsilon>0$, we have  $\fx+\epsilon\fb \in \Delta^{up}$.    
\end{lemma}
\begin{proof}
We fix $0 \neq a \in R$ such that $aI_mI_n \subset I_{m+n}$ for any $m,n \in \mathbb{N}$. Denote $\fv_1=\nu(a)$. Note that by induction we can prove for any $n_1,\ldots,n_k$, $a^{k-1}I_{n_1}\ldots I_{n_k} \subset I_{n_1+\ldots+n_k}$. In particular $a^{n-1}I^n_1 \subset I_n$.We choose $0 \neq b \in I_1$, then $(ab)^n \in I_n$. Denote $\fv_2=\nu(ab)$. For positive integers $n, N$, denote $q_{N,n}=\lfloor N/n \rfloor$ and $r_{N,n}=N-n\lfloor N/n \rfloor$. So we have $N=q_{N,n}n+r_{N,n}$, $q_{N,n}, r_{N,n}$ are integers, $0 \leqslant r_{N,n}<n$. We have $0 \leqslant 1/n-q_{N,n}/N=r_{N,n}/nN <1/N$.

Now we fix a closed upward cone $C' \subset C^\bu$ such that $\fx \in C'^\circ$. Since $\fx \in \Delta^{low}$, by \autoref{second_definition_of_delta_and_nabla} there exists an infinite sequence $\{n_k\}$ such that for any large enough $n_k$, $[\fx]_{n_k} \in 1/n_k\nu(I_{n_k})$.

For every pair of positive integers $m,N$ with $N \geqslant m$, the assumption of weakly graded family implies
$$a^{q_{N,m}}I_m^{q_{N,m}}I_{r_{N,m}} \subset I_N$$
Since $(ab)^{r_{N,m}} \in I_{r_{N,m}}$, $(ab)^m \in I_{r_{N,m}}$. Thus taking valuations on both sides, we have
$$q_{N,m}\fv_1+q_{N,m}\nu(I_m)+m\fv_2 \subset \nu(I_N)$$
Divide both sides by $N$, we have for any $m,N$,
$$(q_{N,m}/N)\fv_1+(mq_{N,m}/N)(\frac{1}{m}\nu(I_m))+(m/N)\fv_2 \subset 1/N\nu(I_N)$$
Use the above containment, we can always find an approximation of $\fx$ in $1/n\nu(I_n)$ for large $n$ within distance $\delta$ for arbitrarily small $\delta>0$. Roughly speaking:
$$(q_{n,m}/n)\fv_1+(mq_{n,m}/n)[\fx]_m+m/n\fv_2 \xrightarrow[]{\textup{fix} m, n \gg 0} 1/m\fv_1+[\fx]_m \xrightarrow[]{m \gg 0} \fx$$
Now we write down the above approximation in the $\epsilon-N$ language. Note that we replace the above terms with close terms 5 times, so in each step we approximate by $1/5\delta$. 

Since there are infinitely many $n_k$ with $[\fx]_{n_k} \in \frac{1}{n_k}\nu(I_{n_k})$, we can fix a sufficiently large $m$ such that $[\fx]_m \in 1/m\nu(I_m)$ with 
\begin{enumerate}
\item $\dist(\fx,[\fx]_m)<1/5\delta$;
\item $1/m\|\fv_1\|<1/5\delta$.
\end{enumerate}
Now for this fixed $m$, we have $q_{N,m}/N-1/m, mq_{N,m}/N-1, m/N$ both goes to $0$ as $N$ goes to infinity, so we can choose a sufficiently large $N_1$ according to $m$ such that the following holds for $n \geqslant N_1$:
\begin{enumerate}
\item $\mid(q_{n,m}/n-1/m)\mid \cdot \|\fv_1\|<1/5\delta$;
\item $\mid(mq_{n,m}/n-1)\mid \cdot\|[\fx]_m\|<1/5\delta$;
\item $m/n\|\fv_2\|<1/5\delta$.
\end{enumerate}
All the assumptions above implies the following inequality. When $n \geqslant N_1$, set $\tilde{\fx}_{m,n}=(q_{n,m}/n)\fv_1+(mq_{n,m}/n)[\fx]_m+m/n\fv_2$, then $\tilde{\fx}_{m,n} \in 1/n\nu(I_n)$ and
\begin{align*}
\|\fx-\tilde{\fx}_{m,n}\|\leqslant \|\fx-[\fx]_m\|+\|(mq_{n,m}/n-1)[\fx]_m\|+\|(q_{n,m}/n-1/m)\fv_1\|+\|1/m\fv_1\|+\|m/n\fv_2\|<\delta.  
\end{align*}
After changing $m$ according to $n$, we get a sequence $\{\tilde{\fx}_n\}$ such that $\tilde{\fx}_n \in 1/n\nu(I_n)$ and $\tilde{\fx}_n \to \fx$ as $n \to \infty$. Therefore,
$[\fx+\epsilon\fb]_n-\tilde{\fx}_n \to \fx+\epsilon\fb-\fx=\epsilon\fb \in C'^\circ$ as $n \to \infty$. We apply \autoref{K_K_appl2} to this sequence, then for large $n$,
$$[\fx+\epsilon\fb]_n-\tilde{\fx}_n \in 1/nS.$$
Now we have $\tilde{\fx}_n \in 1/n\nu(I_n)$, so for large $n$, $[\fx+\epsilon\fb]_n \in 1/n\nu(I_n)$. This means $\fx+\epsilon\fb\in \Delta^{up}$.
\end{proof}
\begin{corollary} \label{delta_low_and_up_same_ht_fn}
Using the notations above, we have:
\begin{enumerate}
\item $\Delta^{up} \subset \Delta^{low} \subset \overline{\Delta^{up}}$;
\item $\Delta^{low}$ and $\Delta^{up}$ have the same height function;
\item $\Delta^{low}$ and $\Delta^{up}$ have the same interior and closure, and $\Delta^{low}\ba\Delta^{up}$ has zero measure;
\item $\nabla^{up}\ba\nabla^{low}$ has zero measure.
\end{enumerate}
\end{corollary}
\begin{proof}
Using \autoref{most_imp_theorem}, \autoref{ht_func_prop_for_closure_interior_boundary} and \autoref{ht_func_equality} we get the result.
\end{proof}
\begin{theorem} \label{limit_existence_weakly_graded_OK_domain}
Let $I_\bullet$ be a BBL weakly graded family. Then the limit
$$\lim_{n \to \infty}\frac{\ell(R/I_n)}{n^d}=[\kk_\nu:\kk]\vol(\nabla^{low})=[\kk_\nu:\kk]\vol(\nabla^{up})$$
exists and is finite.
\end{theorem}
\begin{proof}
Using \autoref{volume_formula_for_BBL_family} and \autoref{delta_low_and_up_same_ht_fn} we get the result.
\end{proof}
\begin{corollary}
Let $I_\bullet$ be a BBL weakly graded family of ideals, then
$$\lim_{n \to \infty}\frac{\#(\nu(R)\ba\nu(I_n))}{n^d}=\vol(\nabla^{up})$$
exists.
\end{corollary}
\begin{proof}
Use \autoref{limit_existence_weakly_graded_OK_domain} and \autoref{result_for_BBL_with_h_equal_1}.
\end{proof}

There is an additional geometric property of the 
$C^\circ$-ideal associated with weakly graded families that may be of independent interest.
\begin{theorem}
Let $I_\bu$ be a BBL weakly graded family of ideals, then $\Delta^{up\circ}$ is a convex set.
\end{theorem}
\begin{proof}
We fix $a \in R^\circ$ such that $aI_mI_n \subset I_{m+n}$. We may assume $\fv=\nu(a) \in C^\circ$; if $\fv \in \partial C$, we may choose $b \in R^\circ$ such that $\nu(b) \in C^\circ$ and replace $a$ by $ab$, and in this case $\nu(ab)=\nu(a)+\nu(b) \in C^\circ$.

We first choose $\fx,\fy \in \Delta^{up\circ} \cap \QQ^d$, $\lambda \in [0,1]\cap \QQ$, and claim that $\lambda\fx+(1-\lambda)\fy \in \overline{\Delta^{low}}$. We choose an integer $N_1>0$ such that $N_1\fx,N_1\fy \in \ZZ^d$, and another integer $N_2>0$ such that $N_2\lambda \in \ZZ$. By \autoref{second_definition_of_delta_and_nabla}, for large enough $n$,
$$\fx=[\fx]_{nN_1} \in 1/nN_1\nu(I_{nN_1}),\fy=[\fy]_{nN_1} \in 1/nN_1\nu(I_{nN_1})$$
or in other words,
$$nN_1\fx\in \nu(I_{nN_1}),nN_1\fy\in\nu(I_{nN_1})$$
By the weakly graded assumption,
$$N_2\lambda\cdot nN_1\fx+(N_2-N_2\lambda)\cdot nN_1\fy+N_2\fv \in \nu(I_{nN_1N_2})$$
which implies
$$\lambda\fx+(1-\lambda)\fy+1/nN_1\fv \in 1/nN_1N_2\nu(I_{nN_1N_2})$$
This also implies that for large $n$,
$$\lambda\fx+(1-\lambda)\fy+1/N_1\fv \in 1/nN_1N_2\nu(I_{nN_1N_2})$$
because $1/N_1\fv-1/nN_1\fv \in 1/nN_1S$ and $1/nN_1N_2\nu(I_{nN_1N_2})$ is an $1/nN_1S$-ideal.
Since $\{nN_1N_2\}$ is an infinite sequence,
$$\lambda\fx+(1-\lambda)\fy+1/N_1\fv \in \Delta^{low}$$
This is true for all $N_1$ such that $N_1\fx,N_1\fy \in \ZZ^d$, so it is also true for any multiple of $N_1$, so for any $n$,
$$\lambda\fx+(1-\lambda)\fy+1/nN_1\fv \in \Delta^{low}$$
But $1/nN_1\fv \to \f0$ as $n \to \infty$, so
$$\lambda\fx+(1-\lambda)\fy \in \overline{\Delta^{low}}$$
Now we choose $\fx,\fy \in \Delta^{up\circ}$, $\lambda \in [0,1]$, and claim that $\lambda\fx+(1-\lambda)\fy \in \overline{\Delta^{low}}$. Actually, since $\Delta^{up\circ}$ is open, $\Delta^{up\circ}\cap \QQ^d$ is dense in $\Delta^{up\circ}$; and also, $[0,1]\cap \QQ$ is dense in $[0,1]$. So we can choose $\fx_n,\fy_n \in \Delta^{up\circ}\cap \QQ^d$ and $\lambda_n \in [0,1]\cap \QQ$ such that $\fx_n \to \fx,\fy_n\to\fy,\lambda_n\to\lambda$ when $n \to \infty$. By the claim $\lambda_n\fx_n+(1-\lambda_n)\fy_n \in \overline{\Delta^{low}}$ and $\lambda_n\fx_n+(1-\lambda_n)\fy_n \to \lambda\fx+(1-\lambda)\fy$. So $\lambda\fx+(1-\lambda)\fy \in \overline{\Delta^{low}}$.

Finally if $\fx,\fy \in \Delta^{up\circ}$, then there exists $\epsilon>0$ such that $\fx-\epsilon\fb,\fy-\epsilon\fb \in \Delta^{up\circ}$. By the claim above, $\lambda\fx+(1-\lambda)\fy-\epsilon\fb \in \overline{\Delta^{low}}$, and $\overline{\Delta^{low}}=\overline{\Delta^{up}}$ by weakly graded assumption. so $\lambda\fx+(1-\lambda)\fy \in \Delta^{up\circ}$, as $\Delta^{up\circ}$ is a $C^{\circ}$ ideal.
\end{proof}

\section{$p$-families and OK valuations}\label{section_7}
In this section we will use OK valuation to give an alternate proof of the convergence for weakly $p$-families. Throughout this section, we assume $(R,\fm,\kk)$ is an OK domain of characteristic $p>0$; we do not assume $R$ is $F$-finite. We also follow  \autoref{general_setup} , and denote the dimension, the valuation, the value semigroup and the cone by $d=\dim R, \nu,S=\nu(R),C=\Cone(S)$ respectively.

Let $I_\bu$ be a BBL family of ideals indexed by powers of $p$. We will check the convergence of
$$F_{1/q\nu^{h,c}(I_q)\cap H,1/q}(\fx)$$
using the convergence of
$$F_{1/q\nu^c(I_q),1/q}(\fx).$$
Most theorems in this chapter are parallel to theorems in Chapter 6 and have same proofs. So we omit the proof.

\begin{proposition}
Let $I_\bu$ be a BBL family. Suppose $H$ is a truncating half space such that $1/q\nu^c(I_q) \subset C \cap H$ for all $q$. 
\begin{enumerate}
\item If $\fx \notin C \cap \overline{H}$, then for sufficiently large $q$, $F_{1/q\nu^c(I_q),1/q}(\fx)=0$.
\item If $\fx \notin \overline{B((C \cap \overline{H}),d)}$, then for any $q \geqslant 1$, $F_{1/q\nu^c(I_q),1/q}(\fx)=0$. 
\item $F_{1/q\nu^c(I_q),1/q}(\fx)$ is uniformly bounded and supported on a bounded set.
\end{enumerate}
\end{proposition}
\begin{proposition}
Let $I_\bu$ be a BBL family. Let $H=H_{<\fu}$ be a truncating half space.
\begin{enumerate}
\item If $\fx \notin C \cap \overline{H}$, then for sufficiently large $q$, $F_{1/q\nu^{h,c}(I_q)\cap H,1/q}(\fx)=0$.
\item If $\fx \notin \overline{B(C\cap \overline{H},d)}$, then for any $q \geqslant 1$, $F_{1/q\nu^{h,c}(I_q)\cap H,1/q}(\fx)=0$. 
\item $F_{1/q\nu^{h,c}(I_q)\cap H,1/q}(\fx)$ is uniformly bounded and supported on a bounded set.
\end{enumerate}
\end{proposition}

We define some sets associated to the family $I_\bullet$:
$$\nabla^{up}=\{\fx \in C^\circ: F_{1/q\nu^c(I_q),1/q}(\fx)=1 \textup{ for infinitely many }q\}$$
$$\nabla^{low}=\{\fx \in C^\circ: F_{1/q\nu^c(I_q),1/q}(\fx)=1 \textup{ for almost all }q\}$$
$$\Delta^{up}=C^\circ\ba\nabla^{up}$$
$$\Delta^{low}=C^\circ\ba\nabla^{low}.$$

We have $\nabla^{low} \subset \nabla^{up} \subset C^\circ \cap\overline{H}$, and $C^\circ\ba\overline{H} \subset \Delta^{up} \subset \Delta^{low}$. Actually, for $\fx \in C^{\circ}\ba\overline{B(C \cap \overline{H},d)}$, $F_{1/q\nu^c(I_q),1/q}(\fx)=0$ for all $q$.

\begin{proposition}\label{second definition of delta and nabla p-family}
For a given BBL family $I_\bu$, we have:
\begin{enumerate}
\item $\nabla^{up}=\{\fx \in C^\circ: [\fx]_q \in 1/q\nu^c(I_q) \textup{ for infinitely many }q\}$;
\item $\nabla^{low}=\{\fx \in C^\circ: [\fx]_q \in 1/q\nu^c(I_q) \textup{ for }q \gg 0\}$;
\item $\Delta^{up}=\{\fx \in C^\circ: [\fx]_q \notin 1/q\nu^c(I_q) \textup{ for }q \gg 0\}$;
\item $\Delta^{low}=\{\fx \in C^\circ: [\fx]_q \notin 1/q\nu^c(I_q) \textup{ for infinitely many }q\}$;
\item $\Delta^{up}=\{\fx \in C^\circ: [\fx]_q \in 1/q\nu(I_q) \textup{ for }q \gg 0\}$;
\item $\Delta^{low}=\{\fx \in C^\circ: [\fx]_q \in 1/q\nu(I_q) \textup{ for infinitely many }q\}$.
\end{enumerate}
\end{proposition}

\begin{lemma}
Using the notations above, then for any family $I_\bullet$ which is not necessarily BBL, the corresponding $\Delta^{up}$ and $\Delta^{low}$ are closed under adding a nonzero element in $C^\circ$. In other words, they are $C^\circ$-ideals.    
\end{lemma}

\begin{proposition}
Under the notations above,
\begin{enumerate}
\item $\nabla^{low} \subset \nabla^{up} \subset C^\circ\cap \overline{H}$;
\item $C^\circ\ba \overline{H} \subset \Delta^{up} \subset \Delta^{low}$;
\item $\nabla^{low\circ} \subset \nabla^{up\circ} \subset C^\circ \cap H^\circ$;
\item $C^\circ\cap H^\circ \subset \nabla^{low\circ} \cup (\partial\nabla^{low}\cap C^\circ) \cup (\Delta^{low\circ}\cap\nabla^{up\circ}) \cup (\partial \nabla^{up}\cap C^\circ) \cup (\Delta^{up\circ} \cap H^\circ)$;
\item $\partial\nabla^{low}\cap C^\circ =\partial \Delta^{low}\cap C^\circ$, $\partial\nabla^{up}\cap C^\circ= \partial \Delta^{up}\cap C^\circ$.
\end{enumerate}   
\end{proposition}

\begin{lemma}
Take any $1 \leqslant h \leqslant [\kk_\nu:\kk]$. Under the notations above:  
\begin{enumerate}
\item For $\fx \notin C \cap \overline{H}$,
$$\lim_{n \to \infty}F_{1/q\nu^{h,c}I_q\cap H,1/q}(\fx)=0.$$
\item For $\fx \in \nabla^{low\circ}\subset H^\circ$,
$$\lim_{n \to \infty}F_{1/q\nu^{h,c}I_q\cap H,1/q}(\fx)=1.$$
\item For $\fx \in \Delta^{up\circ}\cap H^\circ$,
$$\lim_{n \to \infty}F_{1/q\nu^{h,c}I_q\cap H,1/q}(\fx)=0.$$
\item For $\fx \in \Delta^{low\circ}\cap \nabla^{up\circ}$, the limit
$$\lim_{n \to \infty}F_{1/q\nu^{h,c}I_q\cap H,1/q}(\fx)$$
is not clear.
\item For $\fx$ not lying in all the sets above, then $\fx \in E=(\partial C \cup \partial H \cup \partial \Delta^{up} \cup\partial \Delta^{low})\cap(C \cap \overline{H})$. The set $E$ is bounded and has zero measure.
\end{enumerate}
\end{lemma}
\begin{corollary}
For any $\fx \in \mathbb{R}^d$ and $1 \leqslant h \leqslant [\kk_\nu:\kk]$,
$$\chi(\nabla^{low\circ})(\fx) \leqslant \displaystyle \liminf_{n \to \infty}F_{1/q\nu^{h,c}(I_q)\cap H,1/q}(\fx) \leqslant \displaystyle \limsup_{n \to \infty}F_{1/q\nu^{h,c}(I_q)\cap H,1/q}(\fx) \leqslant \chi(\nabla^{up\circ})(\fx)+\chi(E)(\fx).$$
\end{corollary}
\begin{theorem} \label{volume_formula_for_BBL_family_indexed_by_p}
Let $I_\bullet$ be a BBL family. Then:
$$[\kk_\nu:\kk]\vol(\nabla^{low})\leqslant \displaystyle \liminf_{n \to \infty}\frac{\ell(R/I_q)}{q^d}\leqslant \displaystyle \limsup_{n \to \infty}\frac{\ell(R/I_q)}{q^d} \leqslant [\kk_\nu:\kk]\vol(\nabla^{up}).$$
In particular, when $\vol(\nabla^{up}\ba\nabla^{low})=0$, the limit
$$\lim_{n \to \infty}\frac{\ell(R/I_q)}{q^d}=[\kk_\nu:\kk]\vol(\nabla^{low})$$
exists.
\end{theorem}

Now comes the key step in our proof:
\begin{lemma} \label{key_step_weak_p family}
Using the notation above, and assume $I_\bullet$ is a weakly $p$-family. Then for any $\fx \in \Delta^{low}$ and $\epsilon>0$, then $\fx+\epsilon\fb \in \Delta^{up}$.    
\end{lemma}
\begin{proof}
We fix $0 \neq a \in R$ such that $aI_q^{[p]} \subset I_{pq}$ for all $q \geqslant 1$. Then by induction we can prove for any $q_1,q_2$
$$a^{(q_2-1)/(p-1)}I_{q_1}^{[q_2]}=a^{1+p+\ldots+q_2/p}I_{q_1}^{[q_2]}\subset I_{q_1q_2}.$$
This implies $a^{q_2}I_{q_1}^{[q_2]} \subset I_{q_1q_2}$. Let $\fv_1=\nu(a)$, then
$$q_2\fv_1+q_2\nu(I_{q_1}) \subset \nu(I_{q_1q_2}), \fv_1/q_1+1/q_1\nu(I_{q_1}) \subset 1/(q_1q_2)\nu(I_{q_1q_2}).$$
Now we fix a closed upward cone $C' \subset C^\bu$ such that $\fx \in C'^\circ$. By \autoref{second definition of delta and nabla p-family}, $\fx \in \Delta^{low}$ implies that for any large enough $q_k$, $[\fx]_{q_k} \in 1/q_k\nu(I_{q_k})$. This implies for $q \geq q_k$, $1/q_k\fv_1+[\fx]_{q_k} \subset 1/q\nu(I_q)$. Note that we can choose arbitrarily large $k$ and as $k \to \infty$, $q_k \to \infty$ and $1/q_k\fv_1+[\fx]_{q_k} \to \fx$. Therefore there exists a sequence $\{\tilde{\fx}_q\}$ such that $\tilde{\fx}_q \in 1/q\nu(I_q)$ and $\tilde{\fx}_q \to \fx$ as $q \to \infty$. Therefore,
$[\fx+\epsilon\fb]_q-\tilde{\fx}_q \to \fx+\epsilon\fb-\fx=\epsilon\fb \in C'^\circ$ as $q \to \infty$. We apply \autoref{K_K_appl2} to this sequence in the case $k=e,m_k=p^e=q$, then for large $q$,
$$[\fx+\epsilon\fb]_q-\tilde{\fx}_q \in 1/qS.$$
Now $\tilde{\fx}_q \in 1/q\nu(I_q)$ implies $[\fx+\epsilon\fb]_q \in 1/q\nu(I_q)$. So for large $q$, $[\fx+\epsilon\fb]_q \in 1/q\nu(I_q)$, which means $\fx+\epsilon\fb\in \Delta^{up}$.

\end{proof}
\begin{corollary} \label{delta_low_delta_up_same_ht_fn_for_weakly_p_family}
Using the notations above, we have:
\begin{enumerate}
\item $\Delta^{up} \subset \Delta^{low} \subset \overline{\Delta^{up}}$;
\item $\Delta^{low}$ and $\Delta^{up}$ have the same height function.
\item $\Delta^{low}$ and $\Delta^{up}$ have the same interior and closure, and $\Delta^{low}\ba\Delta^{up}$ has zero measure.
\item $\nabla^{up}\ba\nabla^{low}$ has zero measure.
\end{enumerate}
\end{corollary}
\begin{proof}
Use \autoref{key_step_weak_p family}, \autoref{ht_func_prop_for_closure_interior_boundary} and \autoref{ht_func_equality}.
\end{proof}
\begin{theorem} \label{limit_existence_weakly_p_family_OK_domain}
Let $I_\bullet$ be a BBL weakly $p$-family. Then the limit
$$\lim_{q \to \infty}\frac{\ell(R/I_q)}{q^d}=[\kk_\nu:\kk]\vol(\nabla^{up})$$
exists.
\end{theorem}

\begin{proof}
Use \autoref{volume_formula_for_BBL_family_indexed_by_p} and \autoref{delta_low_delta_up_same_ht_fn_for_weakly_p_family}.
\end{proof}
\textbf{Application to $F$-graded system.} \label{subsection_application_to_F_graded_system}\\

\phantom{}
In Brosowsky's thesis \cite{thesis} she considered a particular kind of families called \emph{$F$-graded systems}, see \autoref{defn_F_graded_system}. This concept is defined by Blickle in \cite{blickle2009test} .
\begin{definition} \label{defn_F_graded_system}
If $R$ has characteristic $p$ and $I_\bu$ is a family of ideals indexed by $p$, we say $I_\bu$ is an $F$-graded system, if for any $q_1,q_2$, $I^{[q_2]}_{q_1}I_{q_2} \subset I_{q_1q_2}$. 
\end{definition}

This definition is important in the description of Cartier subalgebras. We recall that for an $F$-finite ring $R$ of characteristic $p$, the Cartier algebra of $R$ is a graded algebra defined as follows
$$\mathcal{C}=\oplus \mathcal{C}_e \quad \text{and} \quad  \mathcal{C}_e=\Hom_R(F^e_*R,R),$$
and the multiplication is given by
$$(\phi_1 \in \mathcal{C}_e,\phi_2 \in \mathcal{C}_f) \to \phi_1\circ F^e_*\phi_2 \in \mathcal{C}_{e+f}.$$
For a family of ideals indexed by powers of $p$, we consider
$$\mathcal{D}=\oplus I_{p^e}\mathcal{C}_e$$
which is a graded $R$-submodule of $\mathcal{C}$. The following theorem is in Blickle-Schwede-Tucker's paper \cite{blickle2012f} and explains how $F$-graded system recovers Cartier subalgebra.
\begin{theorem}
If $I_\bu$ is an $F$-graded system, then $\mathcal{D}$ is a Cartier subalgebra. If moreover $R$ is Gorenstein, then every Cartier subalgebra arises in this way.    
\end{theorem}

\begin{proposition} \label{F_graded_is_weakly_p_family}
Assume for an $F$-graded system $I_\bu$, $I_1 \cap R^\circ \neq 0$. Then $I_\bu$ is a weakly $p$-family.    
\end{proposition}
\begin{proof}
Follows from the definition of $F$-graded system and weakly $p$-family.
\end{proof}

Let $R$ is a polynomial ring; we may replace $R$ by a power series ring to assume $R$ is local, and this does not change $\ell(R/I)$ for a monomial ideal $I$. Then $R$ has a monomial OK valuation. For a monomial ideal $I$, let $\textup{log}(I)$ be the set of exponents for monomials in $I$. Under a monomial OK valuation, $S=\NN^d$, $C=\RR^d_+$, $\textup{log}(I)=\nu(I)$, and 
$$\cup_{f \geqslant 0}\cap_{e \geqslant f}1/p^e\textup{log}(I_{p^e}) \cap C^\circ=\Delta^{up}.$$
Thus we can prove Conjecture IV.5.4 in Brosowsky's thesis \cite{thesis} in a more general setting, which becomes a theorem:
\begin{theorem} \label{main_result_for_BBL_F_graded_system}
Assume $R$ is an $F$-finite OK domain of characteristic $p>0$ and $I_\bu$ is an BBL $F$-graded system of $\fm$-primary ideals, then $\vol(I_\bu)=\vol(C\backslash\Delta(I_\bu))$. 
\end{theorem}
\begin{proof} Note that $C\backslash\Delta(I_\bu)$ and $\nabla^{up}_{I_\bu}$ differs by a subset of $\partial(C)\cap \partial\Delta(I_\bu)$ which has measure $0$, so $\vol(C\backslash\Delta(I_\bu))=\vol(\nabla^{up}_{I_\bu})$.
\end{proof}

\section{OK basis, estimate on $\nu(I^{[p]})$, and weakly inverse $p$-families} \label{section_8}
Throughout this section, we assume  \autoref{general_setup}. Here we use OK valuation to prove the existence for the corresponding limit associate to weakly inverse $p$-families of $R$-ideals. This way of proving the existence is more subtle and restricted. However, it reflects more about the limiting behaviour of the OK valuation.

The difficulty in using the OK valuation lies in the following fact: for any valuation $\nu$,
$$p\nu(I)+S \subset \nu(I^{[p]}).$$
However, the containment of the other side does not always hold, because two terms of two elements in $I^{[p]}$ may add up to a smaller valuation. Thus for an inverse $p$-families we have
$$p\nu(I_q) \subset \nu(I^{[p]}_q), \nu(I_{pq}) \subset \nu(I^{[p]}_q).$$
But there are no containment relations between the two sets $p\nu(I_q),\nu(I_{pq})$. To overcome this, we introduce a concept called \emph{OK basis}. Every OK basis consists of a set of basis in $R$ of $\mathbb{F}$ over the subfield $\mathbb{F}^p$ satisfying conditions on valuations. They generate an $R^p$-submodule of $R$ which is equal to $R$ locally at the generic point, and their valuations will give a control of $\nu(I^{[p]})$ from the other side.

Note that $\nu(R) \subset \mathbb{Z}^d$ generates $\mathbb{Z}^d$ as a group. This means we can extend $\nu$ to $F_*R$ and $F_*\mathbb{F}$, then $\nu(F_*R) \subset 1/p\mathbb{Z}^d$ and generates $1/p\mathbb{Z}^d$ as a group.
\begin{definition} \label{OK_basis_defn}
We say $u_1,\ldots,u_l$ is an \textbf{OK basis} if they satisfy:
\begin{enumerate}
\item $u_1,\ldots,u_l \in R$ are all nonzero;
\item The images of $\nu(u_i)$ under $\mathbb{Z}^d\to \mathbb{Z}^d/p\mathbb{Z}^d$ are pairwise distinct.
\item There exists $0 \neq c \in R$ such that $F_*(cR) \subset \sum RF_*(u_i)$.
\end{enumerate}
\end{definition}
\begin{lemma} \label{OK_basis_reln_with_val}
Let $u_1,\ldots,u_l \in R$ such that the images of $\nu(u_i)$ under $\mathbb{Z}^d\to \mathbb{Z}^d/p\mathbb{Z}^d$ are pairwise distinct. Take any $a_1,\ldots,a_l \in R$, then there exists $i$ such that
$$\nu(\sum_{1 \leqslant j \leqslant l}a_j^pu_j)=\nu(a_i^pu_i).$$
\end{lemma}
\begin{proof}
Since $\nu(a_j^pu_j)=p\nu(a_j)+\nu(u_j) \in p\mathbb{Z}^d+\nu(u_j)$, the condition on $u_j$ guarantees that $\nu(a_j^pu_j)$ are pairwise distinct, so the valuation of the sum must be the valuation of one term.    
\end{proof}
\begin{lemma} \label{OK_basis_characterization}
If $u_1,\ldots,u_l$ is an OK basis of a local ring $R$ of dimension $d$, then:
\begin{enumerate}
\item Every coset of $\mathbb{Z}^d/p\mathbb{Z}^d$ is represented by one $u_i$.
\item $l=p^d$.
\item $\sum RF_*u_i=\oplus RF_*u_i$.
\item $[F_*\mathbb{F}:\mathbb{F}]=p^d$ and the residue field of $R$ is perfect.
\end{enumerate}
\end{lemma}
\begin{proof}
(1): we choose $v_1,v_2 \in R^\times$ with $v=v_1/v_2 \in \mathbb{F}$ lying in any given coset of $\mathbb{Z}^d/p\mathbb{Z}^d$. Then the valuation of $v'=c^pv_1v_2^{p-1}$ lies in the same coset. Now $v' \in cR$, so there exists $a_1,\ldots,a_l \in R$ with
$$F_*v'=\sum_{1 \leqslant j \leqslant l}a_jF_*u_j, v'=\sum_{1 \leqslant j \leqslant l}a^p_ju_j.$$
Then for some $i$, $\nu(v')=\nu(a^p_iu_i) \in p\mathbb{Z}^d+\nu(u_i)$.

(2): This is true by (1) since $[\mathbb{Z}^d:p\mathbb{Z}^d]=p^d$.

(3) If $\sum RF_*u_i \neq \oplus RF_*u_i$, then there exists $a_1,\ldots,a_l \in R$ not all zero with
$$0=\sum_{1 \leqslant j \leqslant l}a_jF_*u_j, 0=\sum_{1 \leqslant j \leqslant l}a^p_ju_j.$$
Thus two of $\nu(a_j^pu_j)$ must be equal, but they live in different cosets of $\mathbb{Z}^d/p\mathbb{Z}^d$, which is a contradiction.

(4) We see $F_*u_i$ is an $\mathbb{F}$-basis of $F_*\mathbb{F}$, so $[F_*\mathbb{F}:\mathbb{F}]=p^d$. But for $F$-finite rings we get $[F_*\mathbb{F}:\mathbb{F}]=p^d[F_*\kk:\kk]$, so $F_*\kk=\kk$, that is, $\kk$ is perfect.
\end{proof}
The following set up will be used for the rest of this section.
\begin{setup} \label{F_finite_perefct_residue_field_set_up}
Let $R$ be an $F$-finite local OK domain of dimension $d$ with perfect residue field.
\end{setup}
From the above we see an $F$-finite OK domain with an OK basis must have perfect residue field. We claim the converse of this statement is true:
\begin{proposition} \label{OK_basis_existence}
Adopt \autoref{F_finite_perefct_residue_field_set_up}. Then $R$ has an OK basis.    
\end{proposition}
\begin{proof}
By the assumption above we have $\rank_R F_*R=p^d$. We choose $u_i \in R, 1 \leqslant i \leqslant p^d$ whose valuation consists of a set of representatives of $\mathbb{Z}^d/p\mathbb{Z}^d$. From the proof of \autoref{OK_basis_characterization}, we see $\oplus RF_*u_i=\sum RF_*u_i \subset F_*R$, and these two $R$-modules have the same rank $p^d$, so they are the same locally at the generic point, which means there exists $0 \neq c' \in F$ with $c'F_*R=F_*((c')^pR) \subset \sum RF_*u_i$. So we are done.  
\end{proof}
\begin{proposition} \label{OK_basis_gives_containment_from_the_other_side}
 Adopt \autoref{F_finite_perefct_residue_field_set_up}. Fix an OK basis $u_1,\ldots,u_{p^d}$ and $c \in R$ with $F_*(cR) \subset \sum RF_*u_i$. Then for any ideal $I \subset R$, 
$$\nu(I^{[p]}) \subset \cup_i (p\nu(I)+\nu(u_i)-\nu(c)) \subset (p\nu(I)-\nu(c)+S).$$
For this choice of $c$ and any $q$, we have
$$\nu(I^{[q]}) \subset q\nu(I)-q\nu(c)+S$$.
\end{proposition}
\begin{proof} 
We choose $a \in I^{[p]}$, then $F_*(ca) \in F_*(cI^{[p]}) \subset IF_*(cR) \subset \sum_i IF_*u_i$. So there exists $a_1,\ldots,a_{p^d} \in I$ with $F_*(ca)=\sum_i a_iF_*u_i$, which means $ca=\sum a_i^{p}u_i$. By \autoref{OK_basis_reln_with_val}, there exists $j$ such that $\nu(ca)=\nu(a_j^{p}u_j)$, so $\nu(a)=p\nu(a_j)+\nu(u_j)-\nu(c) \in p\nu(I)+\nu(u_j)-\nu(c)$. So $\nu(I^{[p]}) \subset \cup_i (p\nu(I)+\nu(u_i)-\nu(c)) \subset (p\nu(I)-\nu(c)+S)$. By induction on $q$, we get $\nu(I^{[q]}) \subset q\nu(I)-(1+p+p^2+\ldots+q/p)\nu(c)+(S+pS+\ldots+(q/p)S) \subset q\nu(I)-q\nu(c)+S$.
\end{proof}
\begin{proposition}\label{OK_basis_gives_containment_from_the_other_side2}
Adopt \autoref{F_finite_perefct_residue_field_set_up}. Then there exists $\fv_2 \in S$ such that for any $q$, $1/q\nu(I^{[q]}) \cap S \subset 1/q\nu(I^{[q]}) \cap \ZZ^d \subset \nu(I)-\fv_2$. 
\end{proposition}
\begin{proof}We fix an OK basis $u_1,\ldots,u_{p^d}$ and an element $c \in R$ as in \autoref{OK_basis_gives_containment_from_the_other_side}. Denote $\fu_i=\nu(u_i) \in S$ and $\fu=\nu(c) \in S$. \autoref{OK_basis_gives_containment_from_the_other_side} tells us that if $\fx \in \nu(I^{[p]})$ then there exists some $\fu_i$ such that $\fx \in p\nu(I)+\fu_i-\fu$. So by induction we can prove that for any $e \geqslant 1$ and $\fx \in \nu(I^{[q]})$, there exists a sequence $i_0,\ldots,i_{e-1}$ such that $\fx \in q\nu(I)+\sum_{0 \leqslant k \leqslant e-1}p^k\fu_{i_k}-(q-1)/(p-1)\fu$. Since $\nu(I^{[q]})$ is an $S$-ideal, $\nu(I^{[q]})+\fu \subset \nu(I^{[q]})$, $\nu(I^{[q]})\subset \nu(I^{[q]})-\fu$, so $\fx \in q\nu(I)+\sum_{0 \leqslant k \leqslant e-1}p^k\fu_{i_k}-q\fu$. We consider the semigroup 
$W$ generated by all $\fu_i$'s, then $W$ is a finitely generated subsemigroup of $S$. The above argument implies $\nu(I^{[q]}) \subset q\nu(I)+W-q\fu$. So $1/q\nu(I^{[q]}) \subset \nu(I)+1/qW-\fu$. Now intersect with $\ZZ^d$, we see $\nu(I) \subset \ZZ^d$ and $\fu \in \ZZ^d$, so
$$1/q\nu(I^{[q]})\cap \ZZ^d \subset \nu(I)-\fu+(1/qW\cap \ZZ^d)$$
But $W$ is a finitely generated semigroup. Therefore, by \cite[Theorem 1.4]{kaveh2012newton}, there is $\fu' \in W \subset S$ such that for any $q$,
$$1/qW\cap \ZZ^d \subset \reg(\Cone(W)) \subset W-\fu'$$
And note that $\nu(I)+W=\nu(I)$, we get $1/q\nu(I^{[q]})\cap \ZZ^d \subset \nu(I)-\fu-\fu'$. Since $S \subset \ZZ^d$, $\fv_2=\fu+\fu'$ satisfies the statement of this proposition.
\end{proof}
\begin{remark}
In the following proof we will use the following fact: let $G \subset \RR^d$ be a subgroup and $g \in G$, $S \subset \RR^d$, then
$$(g+S)\cap G=g+(S\cap (G-g))=g+(S\cap G)$$
So we may write $g+S \cap G$ without causing confusion.
\end{remark}
\begin{lemma} \label{equality_of-ht_fn_for_nabla_up_and_down_for_weakly_inverse_setup}
Adopt \autoref{F_finite_perefct_residue_field_set_up}. Let $I_\bullet$ is a weakly inverse $p$-family, then for any $\fx \in \nabla^{up}$ and $\epsilon>0$ with $\fx-\epsilon\fb \in C^\circ$, then $\fx-\epsilon\fb \in \nabla^{low}$. Therefore, for any $\fx \in \Delta^{low}$ and $\epsilon>0$, then $\fx+\epsilon\fb \in \Delta^{up}$.   
\end{lemma}
\begin{proof}
We start with the first half. Fix $0 \neq a \in R$ such that $aI_{pq} \subset I_q^{[p]}$ for all $q \geqslant 1$. Then by induction we can prove for any $q_1,q_2$
$$a^{(q_2-1)/(p-1)}I_{q_1q_2}=a^{1+p+\ldots+q_2/p}I_{q_1q_2}\subset I_{q_1}^{[q_2]}.$$
This implies $a^{q_2}I_{q_1q_2} \subset I_{q_1}^{[q_2]}$. Let $\fv_1=\nu(a)$, then
$q_2\fv_1+\nu(I_{q_1q_2}) \subset \nu(I_{q_1}^{[q_2]})$, so $1/q_1\fv_1+1/q_1q_2\nu(I_{q_1q_2}) \subset 1/q_1q_2\nu(I_{q_1}^{[q_2]})$. We fix $\fv_2=\nu(c) \in S$ as in \autoref{OK_basis_gives_containment_from_the_other_side2} such that for any ideal $I \subset R$, $1/q\nu(I^{[q]})\cap \ZZ^d \subset \nu(I)-\fv_2$. Then $1/q_1\fv_1+1/q_1q_2\nu(I_{q_1q_2})\cap 1/q_1\ZZ^d \subset 1/q_1\nu(I_{q_1})-1/q_1\fv_2$. Thus
$$1/q_1(\fv_1+\fv_2)+1/q_1q_2\nu(I_{q_1q_2})\cap 1/q_1\ZZ^d \subset 1/q_1\nu(I_{q_1})$$
Now we proceed as in \autoref{key_step_weak_p family}. We fix a closed upward cone $C' \subset C^\bu$ such that $\fx,\fx-\epsilon\fb \in C'^\circ$. We have $\lim_{n \to \infty}[\fx]_n=\fx$ and $\lim_{n \to \infty}[\fx-\epsilon\fb]_n=\fx-\epsilon\fb$, so by  \autoref{K_K_appl2}, for large enough $n$, $[\fx]_n \in 1/nS$ and $[\fx-\epsilon\fb]_n \in 1/nS$. By \autoref{second definition of delta and nabla p-family}, $\fx \in \nabla^{up}$ implies that for any large enough $q_k$, $[\fx]_{q_k} \in 1/q_k\nu^c(I_{q_k}) \subset 1/q_kS$. This implies for $q \geqslant q_k$, $[\fx]_{q_k}\notin 1/q_k(\fv_1+\fv_2)+1/q\nu(I_q)\cap 1/q_k\ZZ^d$. So $[\fx]_{q_k}- 1/q_k(\fv_1+\fv_2)\notin 1/q\nu(I_q)\cap 1/q_k\ZZ^d$. Note that we can choose arbitrarily large $k$ and as $k \to \infty$, $q_k \to \infty$ and $[\fx]_{q_k}-1/q_k(\fv_1+\fv_2) \to \fx$. By \autoref{K_K_appl2}, for large enough $k$ we have $[\fx]_{q_k}-1/q_k(\fv_1+\fv_2) \in 1/q_kS \subset 1/qS \cap 1/q_k\ZZ^d$, so $[\fx]_{q_k}- 1/q_k(\fv_1+\fv_2)\in 1/q\nu^c(I_q)$ for large enough $k$ and large enough $q$ such that $q \geq q_k$. Therefore there exists a sequence $\{\tilde{\fx}_q\}$ such that for large 
$q$, $\tilde{\fx}_q \in 1/q\nu^c(I_q)$ and $\tilde{\fx}_q \to \fx$ as $q \to \infty$. In this case,
$\tilde{\fx}_q-[\fx-\epsilon\fb]_q \to \fx-(\fx-\epsilon\fb)=\epsilon\fb \in C'^\circ$ as $q \to \infty$. We apply \autoref{K_K_appl2} to this sequence in the case $k=e,m_k=p^e=q$, then for large $q$,
$$\tilde{\fx}_q-[\fx-\epsilon\fb]_q \in 1/qS.$$
Also for large $q$, $\tilde{\fx}_q \in 1/q\nu^c(I_q)$ and $[\fx-\epsilon\fb]_q \in 1/qS$, so $[\fx-\epsilon\fb]_q \in 1/q\nu^c(I_q)$. This implies $\fx-\epsilon\fb\in \nabla^{low}$.

The second half is just contrapositive of the first half, so we prove by contradiction. For any $\fx \in \Delta^{low} \in C^\circ$, assume there exists $\epsilon>0$ with $\fx+\epsilon\fb \notin \Delta^{up}$, then $\fx+\epsilon\fb \in C^\circ-\Delta^{up}=\nabla^{up}$. Then by the first half, 
$\fx \in C^\circ$ implies that $\fx=\fx+\epsilon\fb-\epsilon\fb \in \nabla^{low}$, so $\fx \notin \Delta^{low}$, contradiction. So such $\epsilon$ cannot exist, and for any $\fx \in \Delta^{low}$ and $\epsilon>0$, then $\fx+\epsilon\fb \in \Delta^{up}$. 
\end{proof}
\begin{theorem} \label{limit_existence_weakly_inverse_p_in_OK_domain}
Adopt \autoref{F_finite_perefct_residue_field_set_up}. Let $I_\bullet$ be a BBL weakly inverse $p$-family. Then the limit
$$\lim_{q \to \infty}\frac{\ell(R/I_q)}{q^d}=[\kk_\nu:\kk]\vol(\nabla^{up})$$
exists.
\end{theorem}
\begin{proof}
Using \autoref{comp_two_C_circ_ideals}, \autoref{volume_formula_for_BBL_family_indexed_by_p} and \autoref{equality_of-ht_fn_for_nabla_up_and_down_for_weakly_inverse_setup} we get the result.
\end{proof}
OK basis gives us a good estimate of $\nu(I^{[p]})$, but its existence is only guaranteed in a special case. So we pose the following question:
\begin{question}
Assume $R$ has characteristic $p>0$ and is an OK domain. Under what assumptions on $R$, except that $R$ is $F$-finite with perfect residue field, can we find some $\fv \in S$, such that for any ideal $I \subset R$,
$$\nu(I^{[p]}) \subset p\nu(I)-\fv+S$$
holds, where $\nu$ is the corresponding OK valuation?
\end{question}

\subsection{Alternate limit existence proof for BBL weakly inverse $p$-families in $F$-finite rings}
In this subsection, we will use methods introduced by Polstra and Tucker in \cite{polstra2018f} to give an alternate limit existence proof for BBL weakly inverse $p$-families in $F$-finite rings.

\begin{theorem} \label{alternate_limit_existence_proof_for_weakly_inverse_p_family}
Assume $(R,\mathfrak{m},\kk)$ is an $F$-finite local domain of dimension $d$, $I_\bu$ is a BBL weakly inverse $p$-families of ideals indexed by powers of $p$. Then:
\begin{enumerate}
\item The limit
\begin{equation*}
\eta=\lim_{e \to \infty} \dfrac{\ell(R/I_q)}{q^d}
\end{equation*}
exists.
\item There exists a positive constant $C$  such that $1/q^d\ell(R/I_q)-\eta \leqslant C/q$ for all $n \in \mathbb{N}$.
\item If $I_\bu$ is weakly $p$ and weakly inverse $p$-family, then there exists a constant $C$ such that $$\lvert 1/q^d\ell(R/I_q)-\eta\lvert \leqslant C/q.$$
\end{enumerate}
\end{theorem}
\begin{proof}
In this proof, we will use the notation $I_{e}$ instead of $I_{p^e}$ to denote the $e$-th element of the family. Choose $c \in R^\circ$ such that $I_{e}^{[p]} \supseteq cI_{e+1}$ for all $e$. Since $F_*R$ is a finite torsion-free $R$-module of rank $p^{(d+\alpha(R))}$, there is a short exact sequence
\begin{equation*}
0 \to F_{*}R \xrightarrow[]{\phi} R^{\bigoplus p^{(d+\alpha(R))}}   \to N_1 \to 0
\end{equation*}
with $N_1$ a torsion $R$-module. This induces another short exact sequence:
\begin{equation}
0 \to F_{*}R \xrightarrow[]{\phi(F_*c\cdot)} R^{\bigoplus p^{(d+\alpha(R))}}   \to N \to 0.
\end{equation}
Since $c \in R^\circ$, $N$ is still a torsion $R$-module, and $\dim N < \dim R$. Note that $I_{e}^{[p]} \supseteq cI_{e+1}$, therefore $\phi(F_*cF_*I_{e+1}) \subset \phi(F_*I^{[p]}_e) \subset I_e\phi(F_*R) \subset I_eR^{\bigoplus p^{(d+\alpha(R))}}$. Thus we have,
\begin{equation}
F_{*}(R/I_{e+1}) \to (R/I_{e})^{p^{(d+\alpha(R))}} \to N' \to 0
\end{equation}
where $N^{'}$ is a quotient of $N$. This leads to the following inequality:
$$p^{(d+\alpha(R))}\ell(R/I_{e}) \leqslant \ell\left( F_{*}(R/I_{e+1})\right)+\ell(N').$$
Since, $I_{e} \cdot N^{'}=0$, we can assume that $N^{'}$ is a quotient of $\dfrac{N}{I_{e} \cdot N}$, this implies $\ell(N^{'}) \leqslant \ell \left(\dfrac{N}{I_{e} \cdot N}\right)$. Since $I_\bu$ is BBL, we may find an $\fm$-primary ideal $J$ such that $J^{[p^e]} \subset I_e$ for all $q=p^e$. Therefore, 
$$\ell(N^{'}) \leqslant \ell \left( \dfrac{N}{J^{[p^e]}N}\right) \sim e_{HK}(J,N) \cdot p^{e \cdot \dim N} + O(p^{e(\dim N -1)} \leqslant C_{N} \cdot p^{e(d-1)} $$ for some constant $C_{N}$. Now, we also have $\ell_{R} \left( F_{*}(R/I_{e+1})\right) = [F_{*}\kk : \kk] \cdot \ell_{F_{*}R}\left(F_{*}(R/I_{e+1})\right) = p^{\alpha(R)} \cdot \ell_{R}(R/I_{e+1})$. Therefore using Equation 8.2, we have, $p^{(d+\alpha(R))} \cdot \ell (R/I_{e}) - p^{\alpha(R)} \cdot \ell(R/I_{e+1}) \leqslant C_{N}\cdot p^{e(d-1)}$. Now dividing both sides by $p^{ed+d+\alpha(R)}$, we get $ \dfrac{\ell(R/I_{e})}{p^{ed}} -\dfrac{\ell(R/I_{e+1})}{p^{(e+1)d}} \leqslant \dfrac{C_{N}p^{e(d-1)}}{p^{ed+d+\alpha(R)}} = C \cdot 1/p^{e}$, where $C= \dfrac{C_{N}}{p^{d+\alpha(R)}}$. Denote $\beta_{e} = \dfrac{\ell(R/I_{e})}{p^{ed}} $, then $\beta_{e} - \beta_{e+1} \leqslant C \cdot 1/p^{e}$, this implies $ (\beta_{e}-C/p^{e-1})-(\beta_{e+1}-C/p^{e}) \leqslant \dfrac{2C}{p^e}-\dfrac{C}{p^{e-1}} \leqslant 0$, as $2 \cdot C \leqslant p \cdot C$. Therefore $\gamma_{e}= \beta_{e}-C/p^{e-1}$ is an increasing sequence. Now $\mid \gamma_{e} \mid \leqslant (1+ \mid \beta_{e}\mid)$ for all large values of $e$. Note that $\ell(R/I_{e}) \leqslant \ell(R/J^{[p^e]}) \sim e_{HK}(J,R) \cdot p^{ed} + O(p^{e(d-1)})$. So we have $\ell(R/I_{e}) \leqslant D \cdot p^{ed}$ for some constant $D$. Since, this is true for all large $e$, we have that $\mid \beta_{e} \mid$ is bounded. This implies $\eta= \displaystyle \lim_{e \to \infty} \gamma_{e}$ exists, and for any $e$, $\eta-\gamma_e \geqslant 0$. Therefore, $\eta = \displaystyle \lim_{e \to \infty} \dfrac{\ell(R/I_{p^{e}})}{p^{ed}}$ exists, thus (1) is true; and for any $e$, $\eta-\beta_e \geqslant -C/p^{e-1}$, or $\beta_e-\eta \leqslant C/p^{e-1}=Cp/p^e$. Denote $Cp$ by $C$, then (2) is true. (3) is a corollary of (2).
\end{proof}

\section{Volume=Multiplicity formula, Minskowski inequality for families, and positivity} \label{section_9}
In this section we will use the analysis on the sets $\nabla^{up}$, $\nabla^{low}$, $\Delta^{up}$, $\Delta^{low}$ to derive the volume = multiplicity formula. We will also derive Minkowski inequalities for weakly graded families, weakly $p$-families and weakly inverse $p$-families. Finally, we work with the \emph{BAL} family to determine when the limit is nonzero.

We will adopt  \autoref{general_setup}. If we work with a weakly graded family we assume $R$ has arbitrary characteristic; if we work with weakly $p$-families, we assume $R$ has characteristic $p>0$; and if we work with weakly inverse $p$-families, we assume $R$ has characteristic $p$, $R$ is $F$-finite, and its residue field is perfect.
\subsection{Volume=multiplicity formula for weakly graded families}
First we start with a weakly graded family $I_\bullet$. Note that for each $I_n$, the family $I^\bullet_n=\{I^k_n\}$ is a graded family, thus we can associate two sets to the family, denoted by $\nabla^{up}_{I_n}$ and $\nabla^{low}_{I_n}$. Now we consider
$$G^{up}_n(\fx)=\chi(1/n\nabla^{up}_{I_n})(\fx),G^{low}_n(\fx)=\chi(1/n\nabla^{low}_{I_n})(\fx),$$
The above two sequences of functions are equal almost everywhere, i.e., for each $n \in \NN$, $G^{up}_n(\fx)=G^{low}_n(\fx)$ for $\fx$ outside a set of measure $0$. We define
$$\nabla^{up}_{\infty}=\{\fx \in C^\circ: \displaystyle \limsup_{n \to \infty}G^{up}_n(\fx)=1\},\nabla^{low}_{\infty}=\{\fx \in C^\circ: \displaystyle  \liminf_{n \to \infty}G^{low}_n(\fx)=1\}.$$
We want to claim that $\nabla^{up}_{\infty}$ and $\nabla^{low}_{\infty}$ differ by a set of measure $0$.

We denote $\Delta^*_*=C^\circ\ba\nabla^*_*$ for $\nabla^{up}_{I_n},\nabla^{low}_{I_n},\nabla^{up}_{\infty},\nabla^{low}_{\infty}$. By  \autoref{ideal_prop_of_delta_up_and_low}, $\Delta^{up}_{I_n},\Delta^{low}_{I_n}$ are $C^\circ$-ideals. Moreover, we have the following:
\begin{proposition}\label{second_definition_of_delta_and_nabla3}
For a given BBL family $I_\bu$, we have:
\begin{enumerate}
\item $\nabla^{up}_{\infty}=\{\fx \in C^\circ: x \in 1/n\nabla^{up}_{I_n} \textup{ for infinitely many }n\}$.
\item $\nabla^{low}_{\infty}=\{\fx \in C^\circ: x \in 1/n\nabla^{low}_{I_n} \textup{ for }n \gg 0\}$.
\item $\Delta^{up}_{\infty}=\{\fx \in C^\circ: x \in 1/n\Delta^{up}_{I_n} \textup{ for }n \gg 0\}$.
\item $\Delta^{low}_{\infty}=\{\fx \in C^\circ: x \in 1/n\Delta^{low}_{I_n} \textup{ for infinitely many }n\}$.

\end{enumerate}
\end{proposition}
\begin{proof}
Same as \autoref{second_definition_of_delta_and_nabla}.
\end{proof}
\begin{lemma} \label{delta_up_low_infnty_ideal_result}
We have $\Delta^{up}_{\infty} \subset \Delta^{low}_{\infty}$, and these two sets are $C^\circ$-ideals.   
\end{lemma}
\begin{proof}
The containment is true because for every $n$, $\Delta^{up}_{I_n}\subset \Delta^{low}_{I_n}$. Now assume $\fx \in \Delta^{up}_{\infty}$ and $\fu \in C^\circ$. Then $\fx \in 1/n\Delta^{up}_{I_n}$ for large $n$. But since $\Delta^{up}_{I_n}$ is an $C^\circ$-ideal, $n\fx+n\fu \in \Delta^{up}_{I_n}$, so $\fx+\fu \in 1/n\Delta^{up}_{I_n}$ for large $n$. So $\fx+\fu \in \Delta^{up}_{\infty}$, and $\Delta^{up}_{\infty}$ is a $C^\circ$-ideal. Similarly we can prove $\Delta^{low}_{\infty}$ is a $C^\circ$-ideal.   
\end{proof}
Now we take into consideration the two sets $\Delta^{up}$ and $\Delta^{low}$ associated to the original graded family $I_\bullet$.
\begin{lemma} \label{delta_infty_ht_fn_reln_with_delta}
Let $I_\bullet$ be a BBL weakly graded family. Then:
\begin{enumerate}
\item For any $\fx \in \Delta^{up}$ and $\epsilon>0$, $\fx+\epsilon\fb \in \Delta^{up}_{\infty}$.   
\item For any $\fx \in \Delta^{low}_{\infty}$ and $\epsilon>0$, $\fx+\epsilon\fb \in \Delta^{low}$.
\end{enumerate}
\end{lemma}
\begin{proof}
(1): This part does not rely on the weakly graded assumption. We take $\fx \in \Delta^{up}$ and $C' \subset C^\bu$ such that $\fx \in C'^\circ$. By \autoref{second_definition_of_delta_and_nabla}, for large $n$, $[\fx]_n \in 1/n\nu(I_n)$. For such $n$ and any $k$, $k[\fx]_n \in 1/n\nu(I^k_n)$, $[\fx]_n \in 1/nk\nu(I^k_n)$. Let
$$\tilde{\fx}_{n,k}=[n\fx+n\epsilon\fb]_k/n-[\fx]_n$$
We see $|\epsilon\fb-\tilde{\fx}_{n,k}| \leq |([n\fx+n\epsilon\fb]_k-(n\fx+n\epsilon\fb))/n|+|\fx-[\fx]_n| \leqslant d/kn+d/n \to 0$ as $n \to \infty$ and this convergence is uniform with respect to $k$. Thus by 
\autoref{K_K_appl2} there exists $N$ independent of $k$ such that whenever $n \geq N$, $\tilde{\fx}_{n,k}\in 1/nkS$ for any $k$. In particular, $\tilde{\fx}_{n,k}\in 1/nkS$ for large enough $n,k$, so $[n\fx+n\epsilon\fb]_k/n \in 1/nk\nu(I^k_n)$, $[n\fx+n\epsilon\fb]_k \in 1/k\nu(I^k_n)$. We fix $n$ and let $k \to \infty$, this implies $n\fx+n\epsilon\fb \in \Delta^{up}_{I_n}$, so $\fx+\epsilon\fb \in 1/n\Delta^{up}_{I_n}$. Then let $n \to \infty$, we get $\fx+\epsilon\fb \in \Delta^{up}_{\infty}$.

(2) Fix $c \in R^\circ$ such that $cI_nI_m \subset I_{n+m}$ for every $m,n$, and let $\fv=\nu(c)$. Then for every $m,n$, $c^mI^m_n \subset I_{mn}$, and $m\fv+\nu(I^m_n) \subset \nu(I_{mn})$. Take $\fx \in \Delta^{low}_{\infty}$. There exists a sequence $n_k$ such that $\fx \in 1/n_k\Delta^{low}_{I_{n_k}}$. Since $I^\bu_{n_k}$ is a graded family for every $k$, $\fx+\frac{\epsilon}{2}\fb \in 1/n_k\Delta^{up}_{I_{n_k}}$ and $n_k\fx+\frac{n_k\epsilon}{2}\fb \in \Delta^{up}_{I_{n_k}}$. So for each $k$, there exists $N_k$ such that for $m \geqslant N_k$, $[n_k\fx+\frac{n_k\epsilon}{2}\fb]_m \in 1/m\nu(I^m_{n_k})$. But $m\fv+\nu(I^m_{n_k}) \subset \nu(I_{mn_k})$. Therefore $[n_k\fx+\frac{n_k\epsilon}{2}\fb]_m+\fv \in 1/m\nu(I_{mn_k})$ and $[n_k\fx+\frac{n_k\epsilon}{2}\fb]_m/n_k+1/n_k\fv \in 1/mn_k\nu(I_{mn_k})$. When $n_k \to \infty$, $[n_k\fx+\frac{n_k\epsilon}{2}\fb]_m/n_k+1/n_k\fv \to \fx+\frac{\epsilon}{2}\fb$ and this convergence is uniform in $m$. We consider
$$\tilde{\fx}_{m,k}=[\fx+\epsilon\fb]_{mn_k}-[n_k\fx+\frac{n_k\epsilon}{2}\fb]_m/n_k-1/n_k\fv$$
Then as $k\to \infty$, $\tilde{\fx}_{m,k} \to \frac{\epsilon}{2}\fb \in C'^\circ$ which is uniform in $m$, so by \autoref{K_K_appl2}, $\tilde{\fx}_{m,k} \in 1/mn_kS$ for large $m,n_k$. This implies $[\fx+\epsilon\fb]_{mn_k} \in 1/mn_k\nu(I_{mn_k})$ for large $m,k$. By definition $\fx+\epsilon\fb \in \Delta^{low}$.
\end{proof}
\begin{corollary} \label{volume_equality_for_delta_and_delta_infinity}
The four sets $\Delta^{up},\Delta^{up}_{\infty},\Delta^{low},\Delta^{low}_{\infty}$ are the same up to a set of measure $0$. In particular their complement in $C^\circ$ have the same volume.    
\end{corollary}
\begin{proof}
All the four sets are $C^\circ$-ideals in $C$. By  \autoref{delta_infty_ht_fn_reln_with_delta}, we get for any $\epsilon>0$,
$$\varphi_{\Delta^{up}}+\epsilon\geqslant \varphi_{\Delta^{up}_{\infty}},\varphi_{\Delta^{low}_{\infty}}+\epsilon\geqslant \varphi_{\Delta^{low}}.$$
 \autoref{delta_low_and_up_same_ht_fn} tells us $\varphi_{\Delta^{low}}=\varphi_{\Delta^{up}}$, and we always have $\Delta^{up}_{\infty} \subset \Delta^{low}_{\infty}$, so $\varphi_{\Delta^{up}_{\infty}}\geqslant \varphi_{\Delta^{low}_{\infty}}$. Combining all these inequalities we get
$$\varphi_{\Delta^{up}}=\varphi_{\Delta^{up}_{\infty}}=\varphi_{\Delta^{low}_{\infty}}=\varphi_{\Delta^{low}}.$$
So the four $C^\circ$-ideals have the same height function, and they are the same up to a set of measure $0$.
\end{proof}
\begin{theorem} \label{volume_multiplicity_formula_for_OK_domain}
For a BBL weakly graded family $I_\bullet$, the volume=multiplicity formula holds. That is,
$$\lim_{n \to \infty} d!\frac{\ell(R/I_n)}{n^d}=\lim_{n \to \infty}\frac{e(I_n,R)}{n^d}.$$
\end{theorem}
\begin{proof}
We apply the results on a single graded family $I^\bullet_n$ to get
$$e(I_n,R)= d!\cdot [\kk_\nu:\kk]\vol(\nabla^{up}_{I_n})=d! \cdot n^d[\kk_\nu:\kk]\vol(1/n\nabla^{up}_{I_n}).$$
Thus
$$\frac{e(I_n,R)}{n^d}= d!\cdot [\kk_\nu:\kk]\vol(1/n\nabla^{up}_{I_n})= d!\cdot [\kk_\nu:\kk]\int_{\mathbb{R}^d}G^{up}_n(\fx)d\fx.$$
But for $\fx$ outside a set of measure $0$, $G^{up}_n(\fx)=G^{low}_n(\fx)$ for all $n$, therefore
$$\chi(\nabla^{low}_{\infty})(\fx) \leqslant  \displaystyle \liminf_{n \to \infty}G^{up}_n(\fx) \leqslant \displaystyle \limsup_{n \to \infty}G^{up}_n(\fx) \leqslant \chi(\nabla^{up}_{\infty})(\fx).$$
Since all the functions are bounded on a bounded set, we apply  \autoref{Fatou's lemma} and \autoref{reverse Fatou's lemma} to get
$$\vol(\nabla^{low}_{\infty})\leqslant \displaystyle \liminf_{n \to \infty}\int_{\mathbb{R}^d}G^{up}_n(\fx)d\fx \leqslant \displaystyle \limsup_{n \to \infty}\int_{\mathbb{R}^d}G^{up}_n(\fx)d\fx \leqslant \vol(\nabla^{up}_{\infty}).$$
Thus
$$\lim_{n \to \infty}\frac{e(I_n,R)}{n^d}= d!\cdot [\kk_\nu:\kk]\vol(\nabla^{low}_{\infty}).$$
But we have
$$\lim_{n \to \infty} d! \frac{\ell(R/I_n)}{n^d}=d! \cdot [\kk_\nu:\kk]\vol(\nabla^{low}).$$
and $\vol(\nabla^{low}_{\infty})=\vol(\nabla^{low})$ by  \autoref{volume_equality_for_delta_and_delta_infinity}.
\end{proof}

\subsection{Volume=multiplicity for weak $p$-families and weak inverse $p$-families}
Next we assume the ring $R$ has characteristic $p>0$ and start with a family $I_\bullet$ indexed by powers of $p$. For each $I_{q_0}$, the family $I^\bullet_{q_0}=\{I^{[q]}_{q_0}\}$ is a $p$-family, thus we can associate two sets to the family, denoted by $\nabla^{up}_{I_q}$ and $\nabla^{low}_{I_q}$. Now we consider
$$G^{up}_q(\fx)=\chi(1/q\nabla^{up}_{I_q})(\fx),G^{low}_q(\fx)=\chi(1/q\nabla^{low}_{I_q})(\fx),$$
The above two sequences of functions are equal almost everywhere. We define
$$\nabla^{up}_{\infty}=\{\fx \in C^\circ:\displaystyle \limsup_{q \to \infty}G^{up}_q(\fx)=1\},\nabla^{low}_{\infty}=\{\fx \in C^\circ:\lim_{q \to \infty}G^{low}_q(\fx)=1\}.$$
We denote $\Delta^*_*=C^\circ\ba\nabla^*_*$ for $\nabla^{up}_{I_q},\nabla^{low}_{I_q},\nabla^{up}_{\infty},\nabla^{low}_{\infty}$. 
\begin{proposition}\label{second definition of delta and nabla4}
For a given BBL family $I_\bu$, we have:
\begin{enumerate}
\item $\nabla^{up}_{\infty}=\{\fx \in C^\circ: x \in 1/q\nabla^{up}_{I_q} \textup{ for infinitely many }q\}$
\item $\nabla^{low}_{\infty}=\{\fx \in C^\circ: x \in 1/q\nabla^{low}_{I_q} \textup{ for }q \gg 0\}$
\item $\Delta^{up}_{\infty}=\{\fx \in C^\circ: x \in 1/q\Delta^{up}_{I_q} \textup{ for }q \gg 0\}$
\item $\Delta^{low}_{\infty}=\{\fx \in C^\circ: x \in 1/q\Delta^{low}_{I_q} \textup{ for infinitely many }q\}$

\end{enumerate}
\end{proposition}
\begin{proof}
Same as \autoref{second_definition_of_delta_and_nabla} with $n$ replaced by $q$.
\end{proof}
\begin{lemma}
We have $\Delta^{up}_{\infty} \subset \Delta^{low}_{\infty}$, and these two sets are $C^\circ$-ideals.   
\end{lemma}
\begin{proof}
The proof is the same as  \autoref{delta_up_low_infnty_ideal_result}, where we replace $n$ by $q$ and $n_k$ by $q_k$.    
\end{proof}
\begin{lemma}
We assume either $I_\bu$ is a weak $p$-family, or $R$ is $F$-finite with perfect residue field and $I_\bu$ is a weak inverse $p$-family. Then:
\begin{enumerate}
\item For any $\fx \in \Delta^{up}$ and $\epsilon>0$, $\fx+\epsilon\fb \in \Delta^{up}_{\infty}$.   
\item For any $\fx \in \Delta^{low}_{\infty}$ and $\epsilon>0$, $\fx+\epsilon\fb \in \Delta^{low}$.
\end{enumerate}
\end{lemma}
\begin{proof}
Here the proof of (1) are the same as  \autoref{delta_infty_ht_fn_reln_with_delta} (1) and does not rely on the fact that $I_\bu$ is a weak $p$ or weak inverse $p$-family. So we only prove (2).   

First we assume $I_\bu$ is weakly $p$-family. Fix $c \in R^\circ$ such that $cI^{[p]}_q \subset I_{pq}$ for every $m,n$, and let $\fv=\nu(c)$. Then for every $q_1,q_2$, $c^{q_2}I^{[q_2]}_{q_1} \subset I_{q_1q_2}$, and $q_2\fv+\nu(I^{[q_2]}_{q_1}) \subset \nu(I_{q_2q_1})$. Take $\fx \in \Delta^{low}_{\infty}$. There exists a sequence $q_k$ such that $\fx \in 1/q_k\Delta^{low}_{I_{q_k}}$. Since $I^\bu_{q_k}$ is a graded family for every $k$, $\fx+\frac{\epsilon}{2}\fb \in 1/q_k\Delta^{up}_{I_{q_k}}$ and $q_k\fx+\frac{q_k\epsilon}{2}\fb \in \Delta^{up}_{I_{q_k}}$. So for each $k$, there exists $Q_k$ such that for $q \geqslant Q_k$, $[q_k\fx+\frac{q_k\epsilon}{2}\fb]_q \in 1/q\nu(I^{[q]}_{q_k})$. But $q\fv+\nu(I^{[q]}_{q_k}) \subset \nu(I_{qq_k})$, so $[q_k\fx+\frac{q_k\epsilon}{2}\fb]_q+\fv \in 1/q\nu(I_{qq_k})$, $[q_k\fx+\frac{q_k\epsilon}{2}\fb]_q/q_k+1/q_k\fv \in 1/qq_k\nu(I_{qq_k})$. This is true for large enough $q$ and $k$; similar to \autoref{delta_infty_ht_fn_reln_with_delta}, we can prove
$$\tilde{\fx}_{q,q_k}=[\fx+\epsilon\fb]_{qq_k}-[q_k\fx+\frac{q_k\epsilon}{2}\fb]_q/q_k-1/q_k\fv \to \frac{\epsilon}{2}\fb$$
and for large enough $q,q_k$, $\tilde{\fx}_{q,q_k} \in 1/qq_kS$, so $[\fx+\epsilon\fb]_{qq_k} \in 1/qq_k\nu(I_{qq_k})$, and $\fx+\epsilon\fb \in \Delta^{low}$.

Now we assume $I_\bu$ is weakly inverse $p$-family. Fix $c \in R^\circ$ such that $cI_{pq} \subset I^{[p]}_{q}$ for every $q$. Then by \autoref{OK_basis_gives_containment_from_the_other_side}, let $\fv=\nu(c)$, then for every $q_1,q_2$, $\nu(I^{[q_2]}_{q_1}) \subset q_2\nu(I_{q_1})-q_2\fv+S$. Choose $\fx \in \Delta^{low}_{\infty}$, then there exists a sequence $\{q_k\}$ such that $q_k \to \infty$ as $k \to \infty$ and $\fx \in 1/q_k\Delta^{low}_{I_{q_k}}$. This means $q_k\fx \in \Delta^{low}_{I_{q_k}}$. So there exists a sequence $Q_k \geq q_k$ such that for any $k$, $[q_k\fx]_{Q_k} \in 1/Q_k\nu(I^{[Q_k]}_{q_k})$. Note that $[q_k\fx]_{Q_k}/q_k=\lfloor Q_kq_k\fx\rfloor/Q_kq_k=[\fx]_{Q_kq_k}$, so $[\fx]_{Q_kq_k} \in 1/Q_kq_k\nu(I^{[Q_k]}_{q_k})$. By the choice of $\fv$, $[\fx]_{Q_kq_k} \in 1/q_k\nu(I_{q_k})-1/q_k\fv+1/Q_kq_kS$. We choose $\fy_{1,k} \in 1/q_k\nu(I_{q_k})$ and $\fy_{2,k} \in 1/Q_kq_kS$ such that $[\fx]_{Q_kq_k}=\fy_{1,k}+\fy_{2,k}-1/q_k\fv$. Note that $\fy_{1,k}$ lies in a bounded set; actually, since $[\fx]_{Q_kq_k}+1/q_k\fv \to \fx$, $[\fx]_{Q_kq_k}+1/q_k\fv$ is bounded. And $[\fx]_{Q_kq_k}+1/q_k\fv-\fy_{1,k} \in 1/Q_kq_kS \subset C$, we see $[\fx]_{Q_kq_k}+1/q_k\fv\geqslant_{\fa}\fy_{1,k}$, so $\fy_{1,k}$ is bounded. So after replacing $k$ by a subsequence, we may assume $\fy_{1,k}$ converges to $\fy_1$. But $[\fx]_{Q_kq_k}+1/q_k\fv \to \fx$, so $\fy_{2,k}$ converges to $\fy_2$. By the choice of $\fy_{1,k}$ and $\fy_{2,k}$, these sequences lie in the closed cone $C$, so $\fy_1,\fy_2 \in C$. Taking limit we get $\fx=\fy_1+\fy_2$. Now consider the sequence 
$$\tilde{\fx}_k=[\fx+\epsilon\fb]_{q_k}-\fy_{1,k}$$
Then as $k \to \infty$, $\tilde{\fx}_k \to \epsilon\fb+\fy_2 \in C^\circ$. By \autoref{K_K_appl2}, $\tilde{\fx}_k \in 1/q_kS$ for large $k$. So $\fy_{1,k} \in 1/q_k\nu(I_{q_k})$ implies $[\fx+\epsilon\fb]_{q_k} \in 1/q_k\nu(I_{q_k})$ and $\fx+\epsilon\fb \in \Delta^{low}$.

\end{proof}
\begin{corollary}\label{volume_equality_for_delta_and_delta_infinity_for_p_family}
The four sets $\Delta^{up},\Delta^{up}_{\infty},\Delta^{low},\Delta^{low}_{\infty}$ are the same up to a set of measure $0$. In particular their complement in $C^\circ$ have the same volume.    
\end{corollary}
\begin{proof}
Same as \autoref{volume_equality_for_delta_and_delta_infinity}.
\end{proof}
\begin{theorem} \label{volume_multiplicity_formula_for_weakly_p_weakly_inverse_p_for_OK_domain}
For a BBL family $I_\bullet$, assume either $I_\bu$ is weak $p$-family, or $I_\bu$ is weak inverse $p$-family and $R$ is $F$-finite with perfect residue field, then the volume=multiplicity formula holds. That is,
$$\lim_{q \to \infty}\frac{\ell(R/I_q)}{q^d}=\lim_{q \to \infty}\frac{e_{HK}(I_q,R)}{q^d}.$$
\end{theorem}
\begin{proof}
Use the same argument as in \autoref{volume_multiplicity_formula_for_OK_domain}, and then use \autoref{volume_equality_for_delta_and_delta_infinity_for_p_family}.
\end{proof}

\subsection{Minkowski inequalities} \label{Minkoswki section}
Next we prove some Minkowski inequality for families of ideals.
\begin{theorem}[Minkowski inequality for weakly graded families] \label{Minkowski_weakly_graded}
If $I_\bullet$ and $J_\bullet$ are two BBL weakly graded families, then so is $I_\bullet J_\bullet$, and
$$(\lim_{n \to \infty}\frac{\ell(R/I_n)}{n^d})^{1/d}+(\lim_{n \to \infty}\frac{\ell(R/J_n)}{n^d})^{1/d} \geqslant (\lim_{n \to \infty}\frac{\ell(R/I_nJ_n)}{n^d})^{1/d}.$$
\end{theorem}
\begin{proof}
First we prove $\Delta_{I_\bullet}^{up\circ}\cap \QQ^d+\Delta_{J_\bullet}^{up\circ}\cap\QQ^d \subset \Delta_{I_\bullet J_\bullet}^{low}$. Choose $0 \neq \fu \in \Delta^{up\circ}_{I_\bu}\cap\QQ^d$ and $0 \neq \fv \in \Delta^{up\circ}_{J_\bu}\cap\QQ^d$. Choose a closed cone $C' \subset C^\bu$ whose interior contains $\fu,\fv$. By \autoref{second_definition_of_delta_and_nabla}, for large $n$, $[\fu]_n \in 1/n\nu(I_n)$ and $[\fv]_n \in 1/n\nu(J_n)$. Thus $[\fu]_n+[\fv]_n \in 1/n\nu(I_nJ_n)$. But $\fu,\fv \in \QQ^d$, so there is a common denominator $N_0$ such that for any $n$, $[\fu]_{nN_0}=\fu,[\fv]_{nN_0}=\fv,[\fu+\fv]_{nN_0}=\fu+\fv=[\fu]_{nN_0}+[\fv]_{nN_0}$. Thus for large $n$, $[\fu+\fv]_{nN_0} \in 1/nN_0\nu(I_{nN_0}J_{nN_0})$, so $\fu+\fv \in \Delta^{low}_{I_\bu J_\bu}$. So we have $\Delta_{I_\bullet}^{up\circ}\cap \QQ^d+\Delta_{J_\bullet}^{up\circ}\cap\QQ^d \subset \Delta_{I_\bullet J_\bullet}^{low}$.
Since $\Delta_{I_\bullet}^{up\circ}$ is an open set and $\QQ^d$ is dense in $\RR^d$, $\Delta_{I_\bullet}^{up\circ}\cap \QQ^d$ is dense in $\Delta_{I_\bullet}^{up\circ}$. Thus

$$\Delta_{I_\bullet}^{up\circ}+\Delta_{J_\bullet}^{up\circ}\\
\subset \overline{\Delta_{I_\bullet}^{up\circ}\cap \QQ^d}+\overline{\Delta_{J_\bullet}^{up\circ}\cap\QQ^d}\\
\subset \overline{\Delta_{I_\bullet J_\bullet}^{low}}$$

Now by applying Minkowski's inequality to get
$$\vol(C^\circ\ba\Delta^{up\circ}_{I_\bu})^{1/d}+\vol(C^\circ\ba\Delta^{up\circ}_{J_\bu})^{1/d} \geqslant \vol(C^\circ\ba(\Delta^{up\circ}_{I_\bu}+\Delta^{up\circ}_{J_\bu}))^{1/d} \geqslant \vol(C^\circ\ba\overline{\Delta^{low}_{I_\bu J_\bu}})^{1/d}$$
And we are done.
\end{proof}
\begin{theorem}[Minkowski inequality for weakly $p$-families and weakly inverse $p$-families] \label{Minkowski_inequality_for_ weakly_p-families_and_weakly_inverse_p-families}
Assume $R$ is an OK domain of characteristic $p>0$. If $I_\bullet$ and $J_\bullet$ are two BBL weakly $p$-families, or two BBL weakly inverse $p$-families when $R$ is $F$-finite and has perfect residue field, then so is $I_\bullet J_\bullet$, and
$$(\lim_{q \to \infty}\frac{\ell(R/I_q)}{q^d})^{1/d}+(\lim_{q \to \infty}\frac{\ell(R/J_q)}{q^d})^{1/d} \geqslant (\lim_{q \to \infty}\frac{\ell(R/I_qJ_q)}{q^d})^{1/d}.$$
\end{theorem}
The proof is similar to  \autoref{Minkowski_weakly_graded} except that we replace $n$ by $q$ and $\QQ^d$ by $\ZZ[1/p]^d$.

\subsection{Positivity} \label{Positivity subsection}
In this subsection we determine when the limit
$$\lim_{n \to \infty}\frac{\ell(R/I_n)}{n^d}$$
or
$$\lim_{q \to \infty}\frac{\ell(R/I_q)}{q^d}$$
is positive. The following definition is essential in this part.
\begin{definition}
Assume $I_\bu$ is a family of $R$-ideals indexed either by $\NN$ or powers of $p$. We say $I_\bu$ is bounded above linearly, or BAL for short, if:
\begin{enumerate}
\item When $I_\bu$ is a family indexed by $\NN$, there exists $c \in \NN$ such that $I_{n} \subset \fm^{\lfloor n/c \rfloor}$;
\item When $I_\bu$ is a family indexed by powers of $p$, there exists $q_0 \in \NN$ such that $I_{qq_0} \subset \fm^{[q]}$.
\end{enumerate}
\end{definition}
For graded families and $p$-families we already have the following results, which says the limit of a BBL family is nonzero if and only if it is also BAL.

Positivity results for graded families in a regular local ring 
$R$ were studied in \cite{mustactǎ2002multiplicities}, and for 
$p$-families of ideals when 
$R$ is an OK domain in \cite{hernandez2018local}.

We now present a general statement that directly implies the results from these works.
\begin{theorem} \label{BAL_criteria}
Let $R$ be an OK domain. Assume $I_\bu$ is a family indexed by $\NN$ or powers of $p$, and associate the $C^\circ$-ideal $\Delta^{low}$ to $I_\bu$. Then $\Delta^{low} \neq C^\circ$ if and only if $I_\bu$ is BAL.   
\end{theorem}
We start with the following lemma.
\begin{lemma} \label{BAL_characterization_result}
Let $C$ be an upward closed cone, $\Delta$ be a $C^\circ$-ideal in $C^\circ$. Then the following properties are equivalent:
\begin{enumerate}
\item $\Delta=C^\circ$.
\item $\varphi_\Delta=\varphi_{C^\circ}$.
\item $\varphi_\Delta\leqslant\varphi_{C^\circ}$.
\item $\varphi_\Delta(\f0)\leqslant 0$.
\item $\varphi_\Delta(\f0)=0$.
\item $\vol(C^\circ\backslash\Delta)=0$.
\end{enumerate}
\end{lemma}
\begin{proof}
We will prove (1) is equivalent to (2), (2) is equivalent to (3), (3) implies (4) implies (5) implies (3), and (1) implies (6) implies (2).

Note that (1) implies (2) is trivial and (2) implies $C^\circ=\Delta^{\circ} \subset \Delta$. But $\Delta\subset C^\circ$, so $\Delta=C^\circ$, this gives (2) implies (1). (2) implies (3) is trivial and (3) implies $\Delta \subset C^\circ$ implies $\varphi_\Delta\geqslant\varphi_{C^\circ}$, so (3) implies (2). Again (3) implies (4) is trivial and (4) implies $\Delta \subset C^\circ$ implies $\varphi_\Delta\geqslant\varphi_{C^\circ}$, so $\varphi_\Delta(\f0)\geqslant\varphi_{C^\circ}(\f0)=0$, so (5) is true. Now (5) implies (3) by applying \autoref{ht_func_char} to the case $\fx=0$. (1) implies (6) is trivial and we prove (6) implies (2) by the contra-positive statement. We always have $\Delta\subset C^\circ$, so if (2) fails, then by  \autoref{comp_two_C_circ_ideals}, $C^\circ\backslash\Delta$ has a nonempty interior, so its volume is nonzero, so (6) fails.
\end{proof}
\begin{proof}[Proof of \autoref{BAL_criteria}]
We will assume $I_\bu$ is indexed by $\NN$; the proof for families indexed by powers of $p$ is the same by substituting $n$ by $q$.

Assume $I_\bu$ is BAL and there exists $c \in \NN$ such that $I_{n} \subset \fm^{\lfloor n/c \rfloor}$ for any $n$. Choose $\fv$ be the minimal valuation of a set of generators of $\fm$, then $\fm \subset R \cap \FF_{\geqslant \fv}$. We have $\fv>\f0$ by definition of OK valuation. Thus for any $n$, $I_{n} \subset \fm^{\lfloor n/c \rfloor} \subset R \cap \FF_{\geqslant \lfloor n/c \rfloor\fv}$. Take a real number $c'$ with $0<c'<1/c$, then there exists a large $N$ such that $n \geqslant N$ implies $\lfloor n/c \rfloor>nc'$, so $\lfloor n/c \rfloor \fv>nc'\fv$, and $I_{n} \subset R \cap \FF_{\geqslant nc' \fv}$. So let $H=H_{<c'\fv}$ with $c'\fv>0$, we have $1/n\nu(I_n) \subset C\ba H$. Taking $n \to \infty$, this implies $\Delta^{low} \subset C^\circ \ba H$. But $C^\circ \cap H \neq \emptyset$, so $C^\circ \ba H \subsetneq C^\circ$, so $\Delta^{low} \subsetneq C^\circ$.

Now we assume $\Delta^{low} \subsetneq C^\circ$, that is, $\nabla^{low} \neq \emptyset$. By previous lemma, this means $\varphi_{\Delta^{low}}(\f0)>0$, so we can pick $\alpha>0$ such that $\alpha\fb \in \nabla^{low}$. We may assume $\alpha=1$ by rescaling $\fb$ because rescaling does not change its direction. For some fixed $N_1$ and $n \geq N_1$, $[\fb]_n \notin 1/n\nu(I_n)$. We choose an upward subcone $C' \subset C^\bu$ such that $\fb \in C'^\circ$. We fix a small $\delta_1>0$ satisfying $\overline{B(\fb,3\delta_1)} \subset C'^\circ$ and $\|\fb\|-3\delta_1=\delta_2>0$. For large $n$, we have $\|\fb-[\fb]_n\|<\delta_1$. We choose any $\fv \in C$ with $\|\fv\|<\delta_1$. For large $n$, $\|\fv-[\fv]_n\|<\delta_1$, so $\|[\fv]_n\|<2\delta_1$. This bound on $n$ only depends on $\delta$, but not on $\fv$. So we see $\|\fb-([\fb]_n-[\fv]_n)\| \leqslant \|\fb-[\fb]_n\|+\|[\fv]_n\|<3\delta_1$, so $[\fb]_n-[\fv]_n \in C'^\circ$ and $\|[\fb]_n-[\fv]_n\|>\|\fb\|-3\delta_1=\delta_2$. By \autoref{K_K_appl}, for large $n$ only depending on $\delta_2$, whenever $\|\fy\| \geq \delta_2$ and $\fy \in 1/n\reg(C')$, $\fy \in 1/nS$. Thus $[\fb]_n-[\fv]_n \in 1/nS$, and $[\fb]_n \notin 1/n\nu(I_n)$ implies 
$[\fv]_n \notin 1/n\nu(I_n)$. This is true for all $\fv \in C$ with $\|\fv\|<\delta_1$, so for large $n$, $1/n\nu(I_n) \subset C\ba B(\f0,\delta_1)$. By  \autoref{trunc_bound}, there exists $\fv_0>0$ such that $C\cap H_{<\fv} \subset C \cap B(\f0,\delta_1)$, so $C\ba B(\f0,\delta_1) \subset C\cap H_{\geqslant\fv}$. This means
$$I_n \subset R \cap \FF_{\geqslant n\fv}$$
And we can find some $c$ such that $R \cap \FF_{\geqslant n\fv} \subset \fm^{\lfloor n/c \rfloor}$. Thus $I_n \subset \fm^{\lfloor n/c \rfloor}$ for large $n$, and we may enlarge $c$ such that $I_n \subset \fm^{\lfloor n/c \rfloor}$ for all $n$. So $I_n$ is BAL.
\end{proof}
 
\begin{theorem} \label{positivity_for_BAL_in_OK_domain}
Let $R$ be an OK domain.
\begin{enumerate} 
\item Let $I_\bu$ be a BBL weakly graded family of ideals. Then $I_\bu$ is BAL if and only if
$$\lim_{n \to \infty}\frac{\ell(R/I_n)}{n^d}>0.$$
\item  

Assume $R$ has characteristic $p>0$. Assume either $I_\bu$ is a BBL weakly $p$-family of ideals, or a BBL weakly inverse $p$-family of ideals when $R$ is $F$-finite with perfect residue field. Then $I_\bu$ is BAL if and only if
$$\lim_{q \to \infty}\frac{\ell(R/I_q)}{q^d}>0.$$
\end{enumerate}
\end{theorem}
\begin{proof}
Using \autoref{BAL_criteria} and \autoref{BAL_characterization_result} we get the result.
\end{proof}

\section{Reduction to OK domain case} \label{section_10}
In this section we want to extend the results from the OK domain case to the more general setting. In general, if $R$ is a Noetherian local ring with $\dim \left(\nil(\hat{R})\right)<\dim R$ then using the following sequence of lemmas we can pass from $R$ to its completion $\hat{R}$, then $\hat{R}$ to $\hat{R}/\nil(\hat{R})$, i.e., we can assume $R$ is a complete reduced local ring. Finally pass to $R'=R/P$ for  $P \in \Min(R)$, i.e., $\dim R/P=\dim R$.

\begin{lemma} \label{completion_reduction}
Let $(R,\fm,\kk)$ be a Noetherian local ring, $\hat{R}$ be its completion with restpect to the maximal ideal $\fm$. Then for any $\fm$-primary ideal $I$, $\ell(R/I)=l(\hat{R}/I\hat{R})$.    
\end{lemma}
\begin{lemma} \label{BBL_length_bound}
    Let $(R,\fm,\kk)$ be a $d$-dimensional Noetherian local ring of any characteristic, and $I_{\bu}$ be a BBL family of ideals indexed by $\NN$ or by powers of $p$ if the ring has positive characteristic. Let $M$ be an f.g. $R$-module, then there exists $\alpha > 0$ such that $$ \ell_{R} \left( M/I_{n}M \right) \leqslant \alpha \cdot n^{\dim M}$$ for every $n \in \NN$ if the family is indexed by $N$. Otherwise $$\ell_{R} \left( M/I_{q}M \right) \leqslant \alpha \cdot q^{\dim M}$$ for all $q=p^e$ if the family is indexed by powers of $p$.
\end{lemma}
\begin{proof}
    It follows from \cite[Lemma 5.17]{hernandez2018local}.
\end{proof}
\begin{corollary} \label{result_to_reduce_setup}
   Let $(R,\fm,\kk)$ and $I_{\bu}$ be as in \autoref{BBL_length_bound}. If $N$ is an ideal of $R$ and $A=R/N$, then there exists $\beta > 0$ such that $$0 \leqslant \ell_{R}(R/I_{n}) - \ell_{A}(A/I_{n}A) \leqslant \beta \cdot n^{\dim N}$$ for all $n$ if the family is indexed by $\NN$. Otherwise $$\leqslant \ell_{R}(R/I_{q}) - \ell_{A}(A/I_{q}A) \leqslant \beta \cdot q^{\dim N}$$ for all $q=p^e$ if the family is indexed by powers of $p$.
\end{corollary}
\begin{proof}

    It follows from \cite[Corollary 5.18]{hernandez2018local}.
\end{proof}
\begin{remark} \label{Why_R_reduced}
 Let   $\dim \left(\nil(R)\right)<\dim R$ then using  \autoref{result_to_reduce_setup} we have that the asymptotic length remains unchanged going modulo $\nil(R)$, i.e., we can assume that $R$ is reduced, and has only finitely many minimal primes $P_{i}$ such that $\dim R/P_{i}=\dim R$.
\end{remark}
\begin{lemma} \label{dimension_bound}
Let $(R,\fm,\kk)$ be a Noetherian reduced local ring, $\Min(R)=\{P \in \Spec(R),\dim R/P=\dim R\}$. Then:
\begin{enumerate}
\item Consider the natural map
$\iota:R \to \Pi_{P \in \Min(R)}R/P$. Then $K_1=Ker(\iota),K_2=Coker(\iota)$ are finitely generated $R$-modules with $\dim K_1,\dim K_2<\dim R$.
\item There is a map
$\iota': \Pi_{P \in \Min(R)}R/P \to R$ such that  $K_3=Ker(\iota'),K_4=Coker(\iota')$ are finitely generated $R$-modules with $\dim K_3,\dim K_4<\dim R$.
\end{enumerate}
\end{lemma}
\begin{proof}
(1) By assumption for $P \in \Min(R)$, $\nil(R_P)= \nil(R)_P=0$. But $R_P$ is an Artinian ring, so $PR_P=0$ and $R_P=R_P/PR_P$ is a field. This means $\iota_P$ is an isomorphism. By definition of $K_i$ for $i=1,2$, if $P \in \Min(R)$, $(K_i)_P=0$. So $\Supp(K_i)$ does not contain any minimal prime of $R$, so $\dim K_i<\dim R$.

(2) Let $M=\Pi_{P \in \Min(R)}R/P$ be a finitely generated $R$-module, $S=R\ba\cup_{P \in \Min(R)}P$. $S$ is a multiplicative set with $\Spec(S^{-1}R)=\Min(R)$. If $P \notin \Min(R)$, then $S^{-1}R_P=M_P=0$. If $P \in \Min(R)$, then for $P \neq Q \in \Min(R)$, $P \nsubseteq Q$; thus $M_P=R_P$. This means $S^{-1}\iota: S^{-1}R \to S^{-1}M$ is an isomorphism, so it has an inverse $\iota_1$. Since $R$ is Noetherian and $M$ is finitely generated, $\Hom_{S^{-1}R}(S^{-1}M,S^{-1}R)=S^{-1}\Hom_R(M,R)$, thus there exists $\iota' \in \Hom_R(M,R)$ such that $\iota_1=\iota/s$ for some $s \in S$. Now $\iota'/1:S^{-1}M \to S^{-1}R$ is still an isomorpism with kernel $S^{-1}K_3$ and cokernel $S^{-1}K_4$, so $S^{-1}K_3=S^{-1}K_4=0$. For $i=3,4$, $\Supp(K_i)$ does not contain any minimal prime in $\Min(R)$, so $\dim K_i<\dim R$.
\end{proof} 
\begin{lemma} \label{reduction_to_domain_result}
Let $(R,\fm,\kk)$ be a Noetherian reduced local ring, $\Min(R)=\{P \in \Spec(R),\dim R/P=\dim R\}$.
\begin{enumerate}
\item Let $I_\bu$ be a BBL family of ideals indexed by $\NN$. Then there exists a constant $C$ such that
$$\left|\sum_{P \in \Min(R)}\ell_{R/P}((R/P)/I_n(R/P))-\ell_{R}(R/I_nR)\right| \leqslant Cn^{d-1}.$$
\item Let $R$ has characteristic $p>0$ and  $I_\bu$ be a BBL family of ideals indexed by powers of $p$. Then there exists a constant $C$ such that
$$ \left|\sum_{P \in \Min(R)}\ell_{R/P}((R/P)/I_q(R/P))-\ell_{R}(R/I_qR)\right| \leqslant Cq^{d-1}.$$
\end{enumerate}
\end{lemma}
\begin{proof}
We only prove (1), and the proof of (2) will be the same if we replace $n$ by $q=p^e$. Let $M=\Pi_{P \in \Min(R)}R/P$. For the rest of the proof we take length as an $R$-module.\\ Using the following exact sequence
$$R \xrightarrow[]{\iota} M \to K_2 \to 0.$$
 we get an exact sequence
$$R/I_n \xrightarrow[]{\iota} M/I_nM \to K_2/I_nK_2 \to 0,$$
which implies
$$\ell(M/I_nM)\leqslant \ell(R/I_n)+\ell(K_2/I_nK_2).$$
There exists another exact sequence
$$M \xrightarrow[]{\iota'} R \to K_4 \to 0.$$
So we have an exact sequence
$$M/I_nM \xrightarrow[]{\iota} R/I_n \to K_4/I_nK_4 \to 0,$$
which implies
$$\ell(R/I_nR)\leqslant \ell(M/I_nM)+\ell(K_4/I_nK_4).$$
So
$$|\ell(R/I_nR)- \ell(M/I_nM)|\leqslant \max\{\ell(K_2/I_nK_2),\ell(K_4/I_nK_4)\}.$$
Since $I_n$ is BBL, there exists an $\fm$-primary ideal $J$ such that $J^n \subset I_n$, so using \autoref{dimension_bound} we have
$$\ell(K_2/I_nK_2)\leqslant e(J,K_2)n^{\dim K_2}+O(n^{\dim K_2-1}) \leqslant C_2n^{d-1},$$
$$\ell(K_4/I_nK_4)\leqslant e(J,K_4)n^{\dim K_4}+O(n^{\dim K_4-1}) \leqslant C_4n^{d-1}.$$
Let $C=\max\{C_2,C_4\}$ we have
$\max\{\ell(K_2/I_nK_2),\ell(K_4/I_nK_4)\} \leqslant Cn^{d-1}$, so
$$|\ell(R/I_nR)- \ell(M/I_nM)|\leqslant Cn^{d-1}.$$
If $M=\Pi_{P \in \Min(R)}R/P$ then $\ell_{R}(M/I_nM)=\sum_{P \in \Min(R)}\ell_{R/P}((R/P)/I_n(R/P)).$
\end{proof}
The following setup will be used in the rest of this section.
\begin{setup} \label{setup_with_nilradical_assumption}
Let $R$ be a Noetherian local ring of dimension $d$ with $\dim \left(\nil(\hat{R})\right)<\dim R$. We define $\Min(R)=\{P_{i} \in \Spec(R), \dim R_{i}= \dim R/P_{i}= \dim R\}$. Moreover, we assume $R$ has positive characteristic when we work with families of ideals indexed by power of $p$.
\end{setup}
\begin{theorem} \label{main_result_for_OK_domain_reduction}
Adopt \autoref{setup_with_nilradical_assumption}. Then: 
\begin{enumerate}
\item Let $I_\bu$ be a BBL family of ideals indexed by $\NN$. Then
$$\lim_{n \to \infty}\sum_{P \in \Min(R)}\ell_{R/P}((R/P)/I_n(R/P))/n^d-\ell_{R}(R/I_nR)/n^d=0.$$
In particular,
$\displaystyle\lim_{n \to \infty}\ell_{R}(R/I_nR)/n^d$ exists if and only if $\displaystyle\lim_{n \to \infty} \ell_{R/P}\left((R/P)/I_n(R/P)\right)/n^d$ exists for every $P \in \Min(R)$.
\item Assume moreover $R$ has characteristic $p>0$. Let $I_\bu$ be a BBL family of ideals indexed by powers of $p$. Then
$$\lim_{q \to \infty}\sum_{P \in \Min(R)}\ell_{R/P}((R/P)/I_q(R/P))/q^d-\ell_{R}(R/I_qR)/q^d=0.$$
In particular,
$\displaystyle \lim_{q \to \infty}\ell_{R}(R/I_qR)/q^d$ exists if and only if $\displaystyle \lim_{q \to \infty}\ell_{R/P}((R/P)/I_q(R/P))/q^d$ exists for every $P \in \Min(R)$.
\end{enumerate}
\end{theorem}
\begin{proof}
The result follows using \autoref{reduction_to_domain_result}.
\end{proof}

\begin{corollary} \label{Multiplicity_sum_result}
    Adopt \autoref{setup_with_nilradical_assumption}. Then
    \begin{enumerate}
        \item If $I$ is an $\fm$-primary ideal in $R$, then we have 
        $$ e(I,R)= \sum_{P \in \Min(R)} e(I\cdot R/P,R/P)$$
        \item Moreover if $R$ has positive characteristic, then we have
        $$ e_{HK}(I,R)= \sum_{P \in \Min(R) } e_{HK}(I\cdot R/P,R/P) $$.
    \end{enumerate}
\end{corollary}
\begin{proof}
    The result follows from \autoref{Why_R_reduced} and using \autoref{main_result_for_OK_domain_reduction} with $I_{n}=I^{n}$ and $I_{q}=I^{[q]}$.
\end{proof}
\begin{lemma} \cite[Corollary 3.7]{hernandez2018local} \label{Complete_local_domain_is_OK}
A complete local domain containing a field is an OK domain.
\end{lemma}
\begin{theorem} \label{general_weakly_graded_limit_exist}
Adopt \autoref{setup_with_nilradical_assumption}. Then for any BBL weakly graded families $I_\bu$ of $R$-ideals,
$$\lim_{n \to \infty}\frac{\ell(R/I_n)}{n^d}$$
exists.
\end{theorem}
\begin{proof}
First we may pass to completion to assume $R$ is complete and reduced (see \autoref{Why_R_reduced}). In this case $R^\circ$ maps to $(\hat{R})^\circ$, so weakly graded families extend to another weakly graded families. The colengths of ideals and the BBL property stay the same. By \autoref{main_result_for_OK_domain_reduction}, the limit $ \displaystyle \lim_{n \to \infty}\frac{\ell(R/I_n)}{n^d}$ exists if and only if for every $P \in \Min(R)$, set $\overline{R}=R/P$, then $ \displaystyle \lim_{n \to \infty}\dfrac{\ell(\overline{R}/I_n\overline{R})}{n^d}$ exists. Since $R^\circ \to \overline{R}^\circ$ under the natural projection, $I_n$ is weakly graded implies $I_n\overline{R}$ is weakly graded. Moreover $\overline{R}$ is a complete local domain, i.e., an OK domain, so the limit $ \displaystyle \lim_{n \to \infty}\dfrac{\ell(\overline{R}/I_n\overline{R})}{n^d}$ exists by \autoref{limit_existence_weakly_graded_OK_domain}. 
\end{proof}
\begin{theorem} \label{general_weakly_p_and_weakly_inverse_p_exists}
Adopt \autoref{setup_with_nilradical_assumption}.. Then for any BBL families $I_\bu$ of $R$-ideals which is either weakly $p$-family, or weakly inverse $p$-family when $R$ is $F$-finite with perfect residue field, then
$$\lim_{q \to \infty}\frac{\ell(R/I_q)}{q^d}$$
exists.
\end{theorem}
\begin{proof}
The result follows from \autoref{limit_existence_weakly_p_family_OK_domain} , \autoref{limit_existence_weakly_inverse_p_in_OK_domain} and the argument used in \autoref{general_weakly_graded_limit_exist}. Note that $F$-finiteness and the perfect property of the residue field remain unchanged from $R$ to $\hat{R}/P$.
\end{proof}
\begin{remark}
The assumption $\dim \left( \nil(\hat{R})\right)<\dim R$ is the most general assumption we can add. Counterexamples can be found in (\cite{cutkosky2014asymptotic},\cite{hernandez2018local}) for rings $R$ that violates this condition, i.e., we can always construct graded families and $p$-families whose corresponding limits do not exist.  
\end{remark}
\begin{theorem} \label{Minkowski_general_for_weakly_graded}
   Adopt \autoref{setup_with_nilradical_assumption}. If  $I_\bullet$ and $J_\bullet$ are two BBL weakly graded families, then so is $I_\bullet J_\bullet$, and
$$\left(\lim_{n \to \infty}\frac{\ell(R/I_n)}{n^d}\right)^{1/d}+\left(\lim_{n \to \infty}\frac{\ell(R/J_n)}{n^d}\right)^{1/d} \geqslant \left(\lim_{n \to \infty}\frac{\ell(R/I_nJ_n)}{n^d}\right)^{1/d}.$$
\end{theorem}
\begin{proof}
    Using \autoref{completion_reduction} and \autoref{Why_R_reduced} we can assume $R$ is complete reduced and has only finitely many minimal primes $P_{1},\dots,P_{k}$. If $R_{i}=R/P_{i}$ for $1\leqslant i \leqslant k$, then \autoref{main_result_for_OK_domain_reduction} gives us $$\lim_{n \to \infty}\frac{\ell(R/I_n)}{n^d} = \sum_{i=1}^{k} \lim_{n \to \infty} \ell_{R_{i}}(R_{i}/I_nR_{i})/n^d,$$ $$\lim_{n \to \infty}\frac{\ell(R/J_n)}{n^d} = \sum_{i=1}^{k} \lim_{n \to \infty} \ell_{R_{i}}(R_{i}/J_nR_{i})/n^d$$ and $$\lim_{n \to \infty}\frac{\ell(R/I_nJ_{n})}{n^d} = \sum_{i=1}^{k} \lim_{n \to \infty} \ell_{R_{i}}(R_{i}/I_nJ_{n}R_{i})/n^d$$
    Let $a_{i}= \lim_{n \to \infty}\ell_{R_{i}}(R_{i}/I_nR_{i})/n^d$, $b_{i}=\lim_{n \to \infty}\ell_{R_{i}}(R_{i}/J_nR_{i})/n^d$, and \phantom{$c_{i}=\lim_{n \to \infty}\ell_{R_{i}}(R_{i}/I_nJ_{n}R_{i})/n^d$}
    $c_{i}=\lim_{n \to \infty}\ell_{R_{i}}(R_{i}/I_nJ_{n}R_{i})/n^d$, then by \autoref{Minkowski_weakly_graded} we have $a_{i}^{1/d}+b_{i}^{1/d} \geqslant c_{i}^{1/d}$ for $1 \leqslant i \leqslant k$. Set $\overline{a_{i}} = a_{i}^{1/d}$, $\overline{b_{i}} = b_{i}^{1/d}$. Using Minkowski inequality we have
    \begin{flalign*}
    \left( \sum_{i=1}^{k} a_{i}\right)^{1/d}+\left( \sum_{i=1}^{k} b_{i}\right)^{1/d} &= \left( \sum_{i=1}^{k} \overline{a_{i}}^{d}\right)^{1/d}+\left( \sum_{i=1}^{k} \overline{b_{i}}^{d}\right)^{1/d} \\
    & \geqslant \left( \sum_{i=1}^{k} (\overline{a_{i}}+\overline{b_{i}})^{d}\right)^{1/d} \\
    & \geqslant \left(\sum_{i=1}^{k} c_{i}\right)^{1/d}
    \end{flalign*}
    finishing the proof.

\end{proof}

\begin{theorem} \label{Minkowski_general_weakly_p_weakly_inverse_p}
   Adopt \autoref{setup_with_nilradical_assumption}.
   
   \begin{enumerate}
       \item 
   If  $I_\bullet$ and $J_\bullet$ are two BBL weakly $p$-families, then so is $I_\bullet J_\bullet$, and
$$\left(\lim_{q \to \infty}\frac{\ell(R/I_q)}{q^d}\right)^{1/d}+\left(\lim_{q \to \infty}\frac{\ell(R/J_q)}{q^d}\right)^{1/d} \geqslant \left(\lim_{q \to \infty}\frac{\ell(R/I_qJ_q)}{q^d}\right)^{1/d}.$$

\item Moreover if $R$ is $F$-finite with perfect residue field, and $I_\bullet$ and $J_\bullet$ are two BBL weakly inverse $p$-families, then so is $I_\bullet J_\bullet$, and
$$\left(\lim_{q \to \infty}\frac{\ell(R/I_q)}{q^d}\right)^{1/d}+\left(\lim_{q \to \infty}\frac{\ell(R/J_q)}{q^d}\right)^{1/d} \geqslant \left(\lim_{q \to \infty}\frac{\ell(R/I_qJ_q)}{q^d}\right)^{1/d}.$$
\end{enumerate}
\end{theorem}
\begin{proof}
   The result follows from  \autoref{completion_reduction}, \autoref{Why_R_reduced}, \autoref{main_result_for_OK_domain_reduction}, \autoref{Minkowski_inequality_for_ weakly_p-families_and_weakly_inverse_p-families}, and the steps used in \autoref{Minkowski_general_for_weakly_graded}.
\end{proof}

\begin{theorem} \label{general_volume_multiplicity_formula_for_weakly_graded_family}
    Let $R$ be a Noetherian local ring with $\dim \left(\nil(\hat{R})\right)<\dim R$. Let $I_\bullet$ be a BBL weakly graded family, then volume =  multiplicity formula holds. That is,
$$\lim_{n \to \infty} d!\frac{\ell(R/I_n)}{n^d}=\lim_{n \to \infty}\frac{e(I_n,R)}{n^d}.$$
\end{theorem}
\begin{proof}
    Using \autoref{completion_reduction} and \autoref{Why_R_reduced} we can assume $R$ is complete reduced and has only finitely many minimal primes $P_{1},\dots,P_{k}$ with $R_{i}=R/P_{i}$ and $\dim R/P_{i}=d$.
    Since each $R_{i}$ is a complete local domain and therefore an OK domain by \autoref{Complete_local_domain_is_OK}, using \autoref{main_result_for_OK_domain_reduction}, \autoref{volume_multiplicity_formula_for_OK_domain}, and \autoref{Multiplicity_sum_result}, we have
\begin{flalign*}
    \lim_{n \to \infty} d! \dfrac{\ell_{R}({R/I_{n}})}{n^{d}}  & =\sum_{i=1}^{n} \lim_{n \to \infty} d! \dfrac{\ell_{R_{i}}(R_{i}/I_{n}R_{i})}{n^{d}}\\
    & =\sum_{i=1}^{n} \lim_{n \to \infty}\dfrac{e(I_{n}R_{i},R_{i})}{n^{d}}\\
    & =\lim_{n\to \infty}\dfrac{e(I_{n},R)}{n^{d}}.
    \end{flalign*}
\end{proof}
\begin{theorem} \label{general_volume_multiplicity_formula_for_weakly_p_weakly_inverse_p_family}
    Adopt \autoref{setup_with_nilradical_assumption}. For a BBL family $I_\bullet$, assume either $I_\bu$ is weakly $p$-family, or $I_\bu$ is weakly inverse $p$-family and $R$ is $F$-finite with perfect residue field, then  volume = multiplicity formula holds. That is,
$$\lim_{q \to \infty}\frac{\ell(R/I_q)}{q^d}=\lim_{q \to \infty}\frac{e_{HK}(I_q,R)}{q^d}.$$
\end{theorem}

\begin{proof}
  The result follows from \autoref{completion_reduction}, \autoref{Why_R_reduced}, \autoref{Complete_local_domain_is_OK}, \autoref{main_result_for_OK_domain_reduction}, \autoref{Multiplicity_sum_result}, \autoref{volume_multiplicity_formula_for_weakly_p_weakly_inverse_p_for_OK_domain}, and the steps mentioned in \autoref{general_volume_multiplicity_formula_for_weakly_graded_family}.
\end{proof}

\vspace{-.9em}

\bibliographystyle{acm}
\bibliography{reference}

\end{document}